\documentclass[10pt,a4paper,reqno]{amsart}
\usepackage{cite}
\usepackage{amsmath}
\usepackage{amsfonts}
\usepackage{amssymb}
\usepackage{amsthm}
\usepackage{newlfont}
\usepackage{a4}
\usepackage{enumerate}
\usepackage{graphicx,color}
\usepackage{bbm}

\theoremstyle{definition}
\newtheorem{defn}{Definition}[section]
\newtheorem{thm}[defn]{Theorem}

\newtheorem{tvr}[defn]{Proposition}
\newtheorem{cor}[defn]{Corollary}
\theoremstyle{remark}
\newtheorem{example}{Example}[section]
\theoremstyle{remark}

\newlength{\defbaselineskip}
\newcommand{\setlinespacing}[1]%
           {\setlength{\baselineskip}{#1 \defbaselineskip}}

\newcommand{\R}{\mathbb{R}}
\newcommand{\Z}{\mathbb{Z}}

\newcommand{\N}{\mathbb{N}}
\newcommand{\ep}{\varepsilon}
\newcommand{\la}{\lambda}
\newcommand{\set}[2]{\left\{#1  \mid #2 \right\}}
\newcommand{\map}{\rightarrow}

\newcommand{\abs}[1]{\left\vert#1\right\vert}
\newcommand{\comb}[2]{\begin{pmatrix}
     #1\\
     #2
  \end{pmatrix}}
\newcommand{\setcomb}[2]{
\left\{
\begin{smallmatrix}
#1 \\ #2
\end{smallmatrix}
 \right\}  }

\newcommand{\setb}[2]{\left\{#1 \, \big|\, #2 \right\}}

\newcommand{\wt}{\widetilde}

\allowdisplaybreaks

\begin{document}

\title[Multivariate cosine transforms and polynomials]
{Discrete transforms and orthogonal polynomials of (anti)symmetric multivariate cosine functions}

\author[J. Hrivn\'ak]{Ji\v{r}\'{i} Hrivn\'{a}k$^{1}$}
\author[L. Motlochov\'a]{Lenka Motlochov\'{a}$^{1,2}$}

\date{\today}
\begin{abstract}\small
The discrete cosine transforms of types V--VIII are generalized to the antisymmetric and symmetric multivariate discrete cosine transforms. Four families of discretely and continuously orthogonal Chebyshev-like polynomials corresponding to the antisymmetric and symmetric generalizations of cosine functions are introduced. Each family forms an orthogonal basis of the space of all polynomials with respect to some weighted integral. Cubature formulas, which correspond to these families of polynomials and which stem from the developed discrete cosine transforms, are derived. Examples of three-dimensional  interpolation formulas and three-dimensional explicit forms of the polynomials are presented.    
\end{abstract}

\maketitle

\noindent
$^1$ Department of physics, Faculty of Nuclear Sciences and Physical Engineering, Czech Technical University in
\phantom{$^1$} Prague, B\v{r}ehov\'a~7, CZ-115 19 Prague, Czech Republic\\
$^2$ D\'epartement de math\'ematiques et de statistique, Universit\'e de Montr\'eal, Qu\'ebec, Canada
\vspace{5pt}

\noindent
\textit{E-mail:} jiri.hrivnak@fjfi.cvut.cz, lenka.motlochova@fjfi.cvut.cz



\section{Introduction}
This paper aims to complete and extend the study of antisymmetric and symmetric multivariate generalizations of common cosine functions of one variable from \cite{KPtrig}. 
First, the set of four symmetric and four antisymmetric discrete Fourier-like transforms from \cite{KPtrig} is extended. Then four families of Chebyshev-like orthogonal polynomials are introduced and the entire collection of the 16 discrete cosine transforms (DCTs) is used to derive the corresponding numerical integration formulas. 

The antisymmetric and symmetric cosine functions of $n$ variables are introduced in \cite{KPtrig} as determinants and permanents of matrices whose entries are cosine functions of one variable. The lowest-dimensional nontrivial case $n=2$ is detailed in \cite{HP}. These real-valued functions have several remarkable properties such as continuous and discrete orthogonality which lead to continuous and discrete analogues of Fourier transforms. The discrete orthogonality relations of these functions are a consequence of ubiquitous DCTs and their Cartesian product multidimensional generalizations. There are eight known different types of DCTs based on various boundary conditions \cite{Brit}. Only the first four transforms are generalized to the multidimensional symmetric cosine functions \cite{KPtrig}.  Therefore, the remaining four transforms of types  V-VIII  need to be developed to obtain the full collections of antisymmetric and symmetric cosine transforms with all possible boundary conditions. The resulting 16 transforms are then available for similar applications as multidimensional DCTs. Besides a straightforward utilization of these transforms to interpolation methods, they may also serve as a starting point for Chebyshev-like polynomial analysis.

The Chebyshev polynomials of one variable are well-known and extensively studied orthogonal polynomials connected to efficient methods of numerical integration and approximations. The Chebyshev polynomials are of the four basic kinds \cite{Hand}, the first and third kinds related to the cosine functions, the second and fourth kinds related to the sine functions. Since both families of antisymmetric and symmetric cosine functions  are based on the one-dimensional cosine functions, they admit a multidimensional generalization of the one-dimensional Chebyshev polynomials of the first and third kinds. The results are four families of orthogonal polynomials, which can be viewed as the symmetric and antisymmetric Chebyshev-like polynomials of the first and third kinds. Several generalizations of Chebyshev polynomials to higher dimensions are known -- in fact for the two-dimensional case the resulting polynomials become, up to a multiplication by constant, special cases of two-variable analogues of Jacobi polynomials \cite{koorn,koorn2,koorn1}. In full generality, multidimensional Chebyshev polynomials of the first kind are constructed, for example, in \cite{hoff,beer}, for broader theoretical background see also \cite{jac,mac}. As is the case of the Chebyshev polynomials of one variable, the multivariate Chebyshev-like polynomials inherit properties from the generalized cosine functions. The link with these cosine functions provides tools to generalize efficient numerical integration formulas of the classical Chebyshev polynomials.  

One of the most important application of the classical Chebyshev polynomials is the calculation of  numerical quadratures -- formulas equating a weighted integral of a polynomial function not exceeding a specific degree with a linear combination of polynomial values at some points called nodes \cite{Hand, riv}. The main point of quadrature formulas, which are in a multidimensional setting known as cubature formulas, is to replace integration by finite summing. The specific degree, to which cubature formulas hold exactly for polynomials, represents a degree of precision of a given formula. 

If a cubature formula holds only for polynomials of degree at most $2d-1$, then
the least number of nodes is equal to the dimension of the space of multivariate polynomials of degree at most $d-1$.  Any cubature formulas which satisfy the lowest bound of the number of nodes have the maximal degree of precision and are called the Gaussian cubature formulas \cite{Xuort, xua2,xuc2,xug2}. In addition, it is known that such formulas exist only if the number of real distinct common zeros of the corresponding orthogonal polynomials of degree $d$ is also exactly the dimension of the space of the multivariate polynomials of degree at most $d-1$ \cite{Xuort}. It appears that among the 16 types of cubature formulas, resulting from 16 discrete transforms, 4 are Gaussian. The remaining 12, even though they are not optimal, extend the options for numerical calculation of multivariate integrals. 

Both antisymmetric and symmetric cosine functions can be restricted to their fundamental domain -- a certain simplex in $\R^n$. For practical applications, it is essential that action of the permutation group $S_n$ on this simplex results in the entire space $\R^n$. Thus, a domain of any function of $n$ real variables can be split into blocks -- covered by sufficiently small copies of the fundamental domain. Since the fundamental domain is a subset of the $n-$dimensional cube, this method results in $n!$ times more blocks than the standard splitting into  $n-$dimensional cubes. The multidimensional (anti)symmetric DCTs and the direct calculation of the corresponding interpolations are then performed in each of the blocks separately. Moreover, the resulting DCT coefficients can be further analyzed and serve as an input to multidimensional analogues of data hiding \cite{Lin} and image recognition \cite{Man} methods. Even though these methods can be applied to any sampled function, the more a given data set obeys the (anti)symmetries and boundary conditions of the DCTs at hand, the more suitable is the given method. The applicability of these methods to interpolation of potentials \cite{YuAnNy} can be tested -- for instance, any potential in physics, which depends on the radial distance from the origin only (gravitational, Coloumb, Yukawa), is also symmetric as a function in $\R^3$. 

The 16 developed cubature formulas offer new options for numerical integration on the transformed fundamental domain via a certain substitution. The transformed domain is, however, of a nonstandard shape -- two main approaches, discussed together with applications in discretizations of PDEs and spectral approximations in \cite{munthe1,munthe2}, can be followed to rectify this impracticality. Both these methods result in integration formulas on the simplex of the shape of the original fundamental domain. The first method straightens the transformed domain back to the original via a suitable straightening map; this methods results in higher densities of the nodes around the corners and edges of the simplex. The second method inscribes the original fundamental domain into the transformed domain and considers only the intersection points as nodes. The main advantage of both methods is that the resulting discrete calculus on the simplex allows a block decomposition of any bounded domain and thus, similarly to the (anti)symmetric cosine functions interpolation methods, can process practically any data set.              

In Section \ref{uvod}, notation and terminology as well as relevant facts from \cite{KPtrig} are reviewed. Simplified forms of three special cases of the generalized cosine functions, which appear as denominators in the definition of the Chebyshev-like multivariate polynomials, are deduced. The continuous orthogonality relations of the generalized cosine functions are presented. In Section \ref{DMCT}, the one-dimensional DCTs of types V--VIII are reviewed and generalized to the antisymmetric and symmetric multivariate DCTs. The interpolation formulas in terms of antisymmetric and symmetric cosine functions are developed.  In Section~\ref{seccheb}, the four families of the multivariate Chebyshev-like polynomials are introduced and three-dimensional examples presented. It is proved that each family of polynomials forms an orthogonal basis of the space of all polynomials with respect to the scalar product given by a weighted integral. In Section \ref{seccub}, the cubature formulas for each family of polynomials are derived. The last section contains concluding remarks and addresses follow-up questions.


\section{Symmetric and antisymmetric multivariate cosine functions}\label{uvod}
\subsection{Definitions, properties, and special values}\

The symmetric and antisymmetric multivariate generalizations of the cosine functions are defined and their properties detailed in \cite{KPtrig}. The antisymmetric cosine functions $\cos^-_\lambda(x)$ and the symmetric cosine functions $\cos^+_\lambda(x)$ of variable $x=(x_1,\dots,x_n)\in\R^n$ and labeled by parameter $\lambda=(\lambda_1,\dots,\lambda_n)\in\R^n$ are defined as determinants and permanents, respectively, of the matrices with the entries $\cos (\pi \la_i x_j  )$, i.e., taking a permutation $\sigma\in S_n$ with its sign $\mathrm{sgn} (\sigma)$ one has
\begin{equation}\label{trigfun}
\begin{aligned}
\cos^-_\lambda(x)
&=\sum_{\sigma\in S_n} \mathrm{sgn}\,(\sigma)\cos{(\pi\lambda_{\sigma(1)}x_1)}\cos{(\pi\lambda_{\sigma(2)}x_2)}\cdots\cos{(\pi\lambda_{\sigma(n)}x_n)}\,,\\
\cos^+_\lambda(x)&=\sum_{\sigma\in S_n} \cos{(\pi\lambda_{\sigma(1)}x_1)}\cos{(\pi\lambda_{\sigma(2)}x_2)}\cdots\cos{(\pi\lambda_{\sigma(n)}x_n)}\,.
\end{aligned}
\end{equation}
The $\cos^\pm_\lambda(x)$ functions  are continuous and have continuous derivatives of all degrees in $\R^n$. Moreover, they are invariant or anti-invariant under all permutations in $\sigma\in S_n$, i.e.,
\begin{alignat}{3}
&\cos^-_\lambda(\sigma(x))=\mathrm{sgn}(\sigma)\cos^-_\lambda(x)\,,&\quad &\cos^-_{\sigma(\lambda)}(x)=\mathrm{sgn}(\sigma)\cos^-_\lambda(x)\,,\label{perm}\\
&\cos^+_\lambda(\sigma(x))=\cos^+_\lambda(x)\,,&\quad &\cos^+_{\sigma(\lambda)}(x)=\cos^+_\lambda(x)\,,\label{perms}
\end{alignat}
where $\sigma(x)=(x_{\sigma(1)},x_{\sigma(2)},\dots,x_{\sigma(n)})$ and $\sigma(\lambda)=(\lambda_{\sigma(1)},\lambda_{\sigma(2)},\dots,\lambda_{\sigma(n)})$. Additionally, they are symmetric 
with respect to the alternations of signs -- denoting $\tau_i$ the change of sign of $x_i$, i.e., $\tau_i(x)=\tau_i(x_1,\dots,x_i,\dots,x_n)\equiv(x_1,\dots,-x_i,\dots,x_n)$ it holds that
\begin{equation}\label{sign}
\begin{alignedat}{2}
&\cos^-_\lambda(\tau_i(x))=\cos^-_\lambda(x)\,,&\quad &\cos^-_{\tau_i(\lambda)}(x)=\cos^-_\lambda(x)\,,\\
&\cos^+_\lambda(\tau_i(x))=\cos^+_\lambda(x)\,,&\quad &\cos^+_{\tau_i(\lambda)}(x)=\cos^+_\lambda(x)\,.
\end{alignedat}
\end{equation}

Considering the functions $\cos^\pm_k(x)$ with their parameter having only integer values $k \in \Z ^n$ and denoting 
\begin{equation}\label{rho}
	\rho\equiv\left(\frac12,\dots,\frac12\right),
\end{equation}
the symmetry related to the periodicity of the one-dimensional cosine functions, i.e., for $t=(t_1,\dots,t_n)\in\Z^n$ we obtain
\begin{equation}\label{period}
\begin{alignedat}{3}
&\cos^-_{k}(x+2 t)=\cos^-_k(x)\,,&\quad& \cos^-_{k+\rho}(x+2 t)=(-1)^{t_1+\dots+t_n}\cos^-_{k+\rho}(x),\\
&\cos^+_{k}(x+2 t)=\cos^+_k(x)\,,&\quad&\cos^+_{k+\rho}(x+2 t)=(-1)^{t_1+\dots+t_n}\cos^+_{k+\rho}(x).
\end{alignedat}
\end{equation}
Introducing the two index sets
\begin{align}
P^+\equiv &\set{(k_1,\dots,k_n)\in\Z^n}{k_1\geq k_2\geq\dots\geq  k_n\geq0},	\\
P^{++}\equiv &\set{(k_1,\dots,k_n)\in\Z^n}{k_1> k_2>\dots>  k_n\geq0},
\end{align}
the relations \eqref{perm}, \eqref{perms}, and \eqref{sign} imply that we consider only the following restricted values of $k\in \Z^n$,
\begin{equation}\label{functions}
\begin{alignedat}{2}
&\cos^-_k(x)\,,\cos^-_{k+\rho}(x)\,:&\quad& k\in P^{++}\,,\\
&\cos^+_k(x)\,,\cos^+_{k+\rho}(x)\,:&\quad& k\in P^+\,.
\end{alignedat}
\end{equation}
By the relations \eqref{perm}--\eqref{period} we consider the antisymmetric and symmetric cosine functions $\cos^\pm_k(x)$ labeled by $k\in \Z^n$ on the closure of the fundamental domain $F(\wt{S}_n^{\mathrm{aff}})$ of the form
\begin{equation}
F(\wt{S}_n^{\mathrm{aff}})=\set{(x_1,x_2,\dots,x_n)\in \R^n}{1\geq x_1\geq x_2\geq\dots \geq x_n\geq 0}\,.
\end{equation} 
Due to \eqref{perm}, \eqref{sign}, and the identity $\cos{\pi\left(k_i+\frac{1}{2}\right)}=0$, valid for $k_i\in \Z$ and $i\in \{1,\dots,n\}$,  we omit those boundaries of $F(\wt{S}_n^{\mathrm{aff}})$ for which
\begin{itemize}
\item  $x_i=x_{i+1}$, $i\in \{1,\dots,n-1\}$ in the case of $\cos^-_k(x)$,
\item  $x_i=x_{i+1}$, $i\in \{1,\dots,n-1\}$ or $x_1=1$ in the case of $\cos^-_{k+\rho}(x)$,
\item  $x_i=1$, $i\in \{1,\dots,n\}$ in the case of $\cos^+_{k+\rho}(x)$.
\end{itemize}
The contour plots of the cuts of the three-dimensional antisymmetric and symmetric cosine functions 
are depicted in Fig.~\ref{cosa} and Fig.~\ref{coss} respectively. 
\begin{figure}
\includegraphics[width=3.5cm]{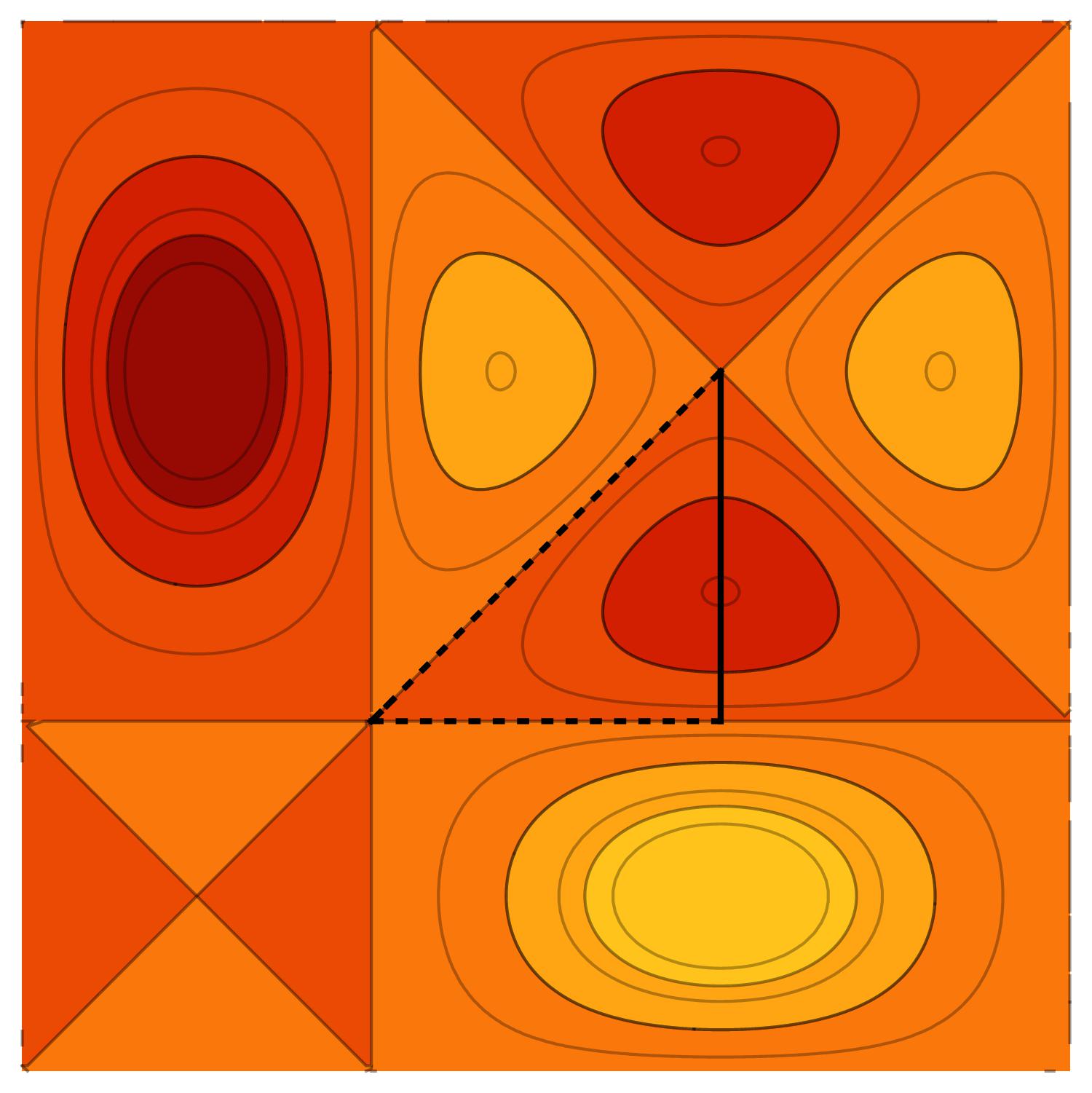}
\hspace{8pt}
\includegraphics[width=3.5cm]{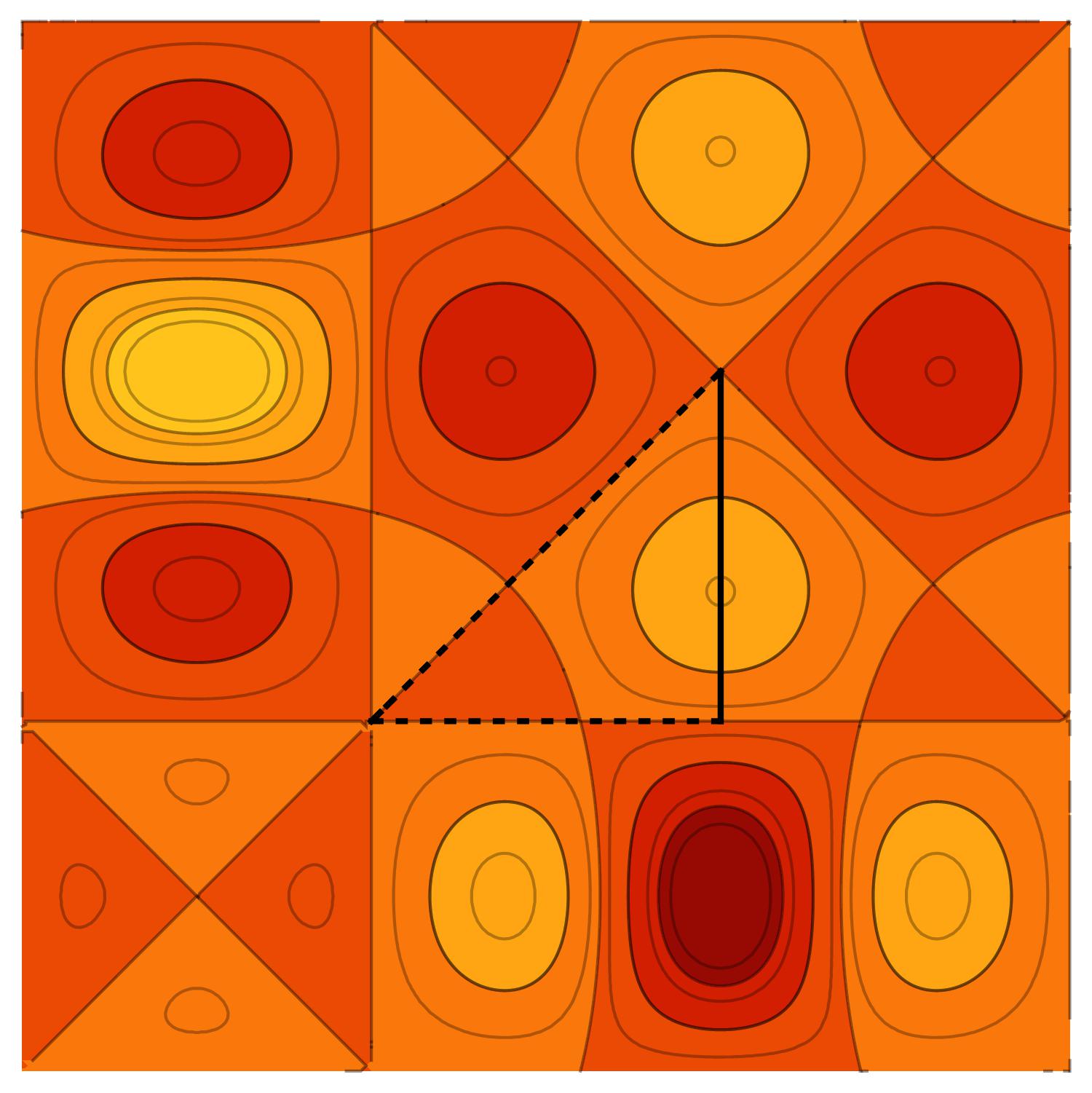}
\hspace{8pt}
\includegraphics[width=3.5cm]{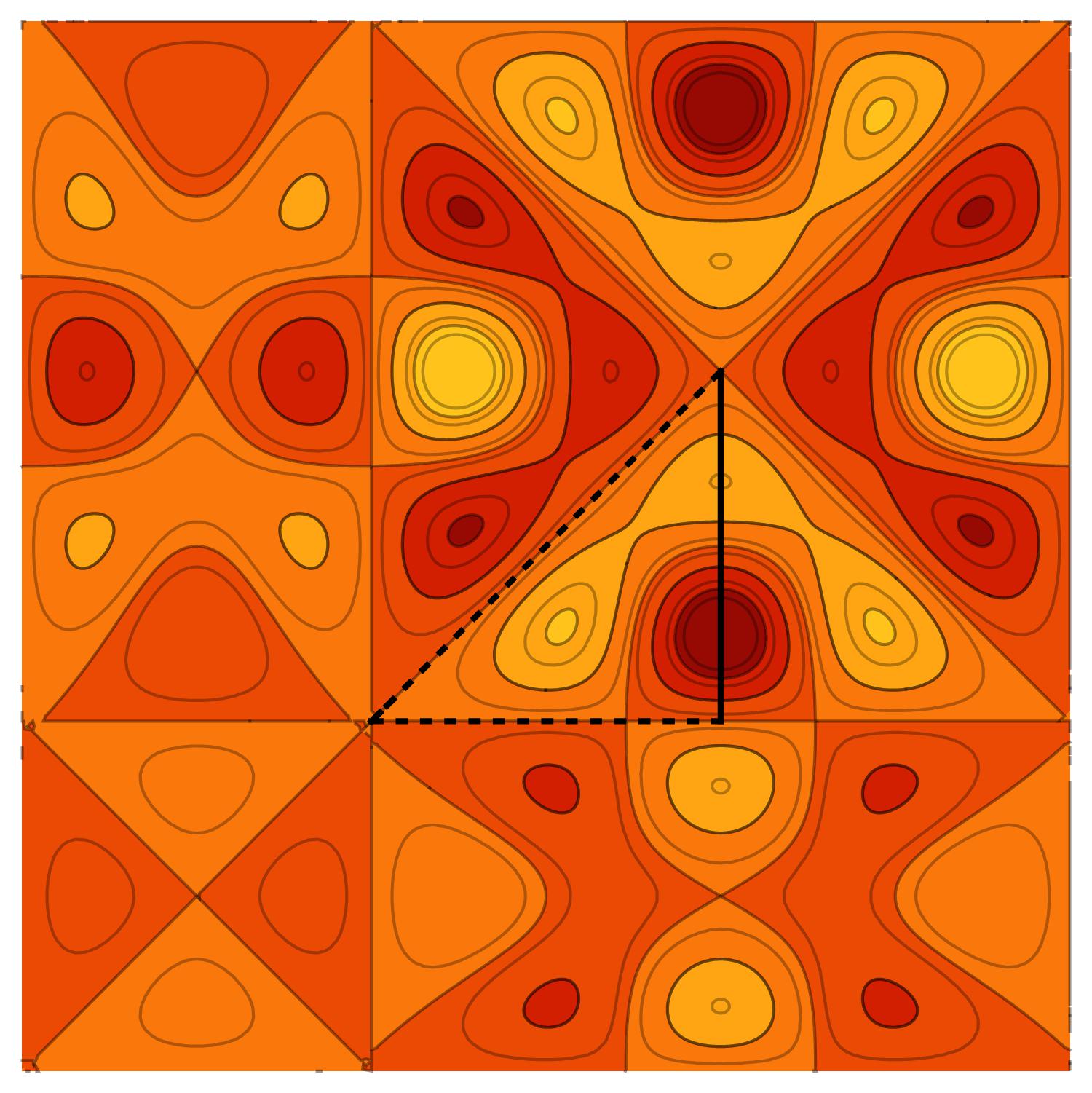}
\caption{\small The contour plots of the graph cuts ($x_3=\frac13$) of the antisymmetric cosine functions $\cos^-_k(x)$.  The cut of the boundary of the fundamental domain $F(\wt{S}_3^\mathrm{aff})$ is depicted as the black triangle with the dashed lines representing the part of the boundary for which all the antisymmetric cosine functions vanish.}
\label{cosa}
\end{figure}
\begin{figure}
\includegraphics[width=3.5cm]{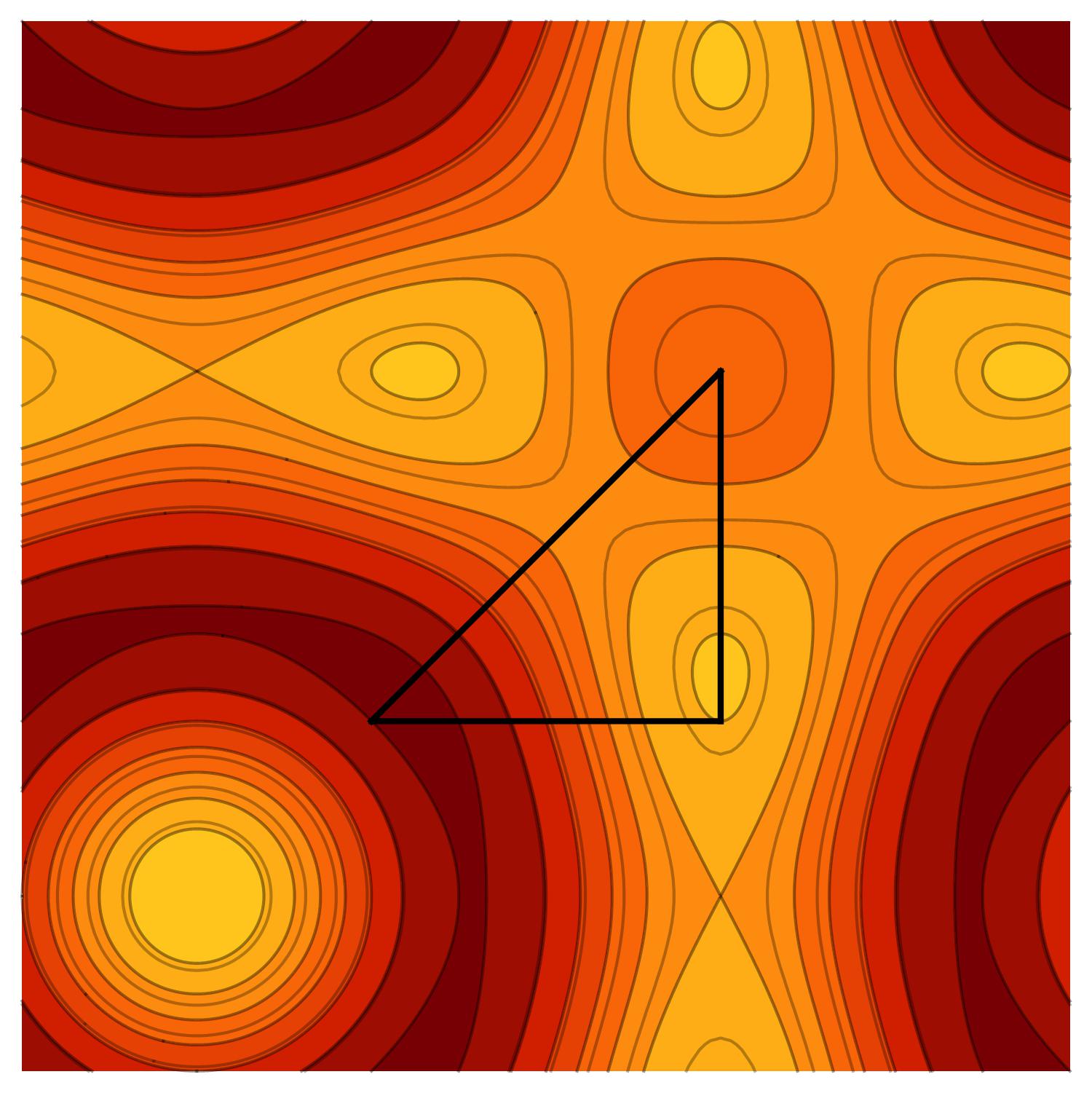}
\hspace{8pt}
\includegraphics[width=3.5cm]{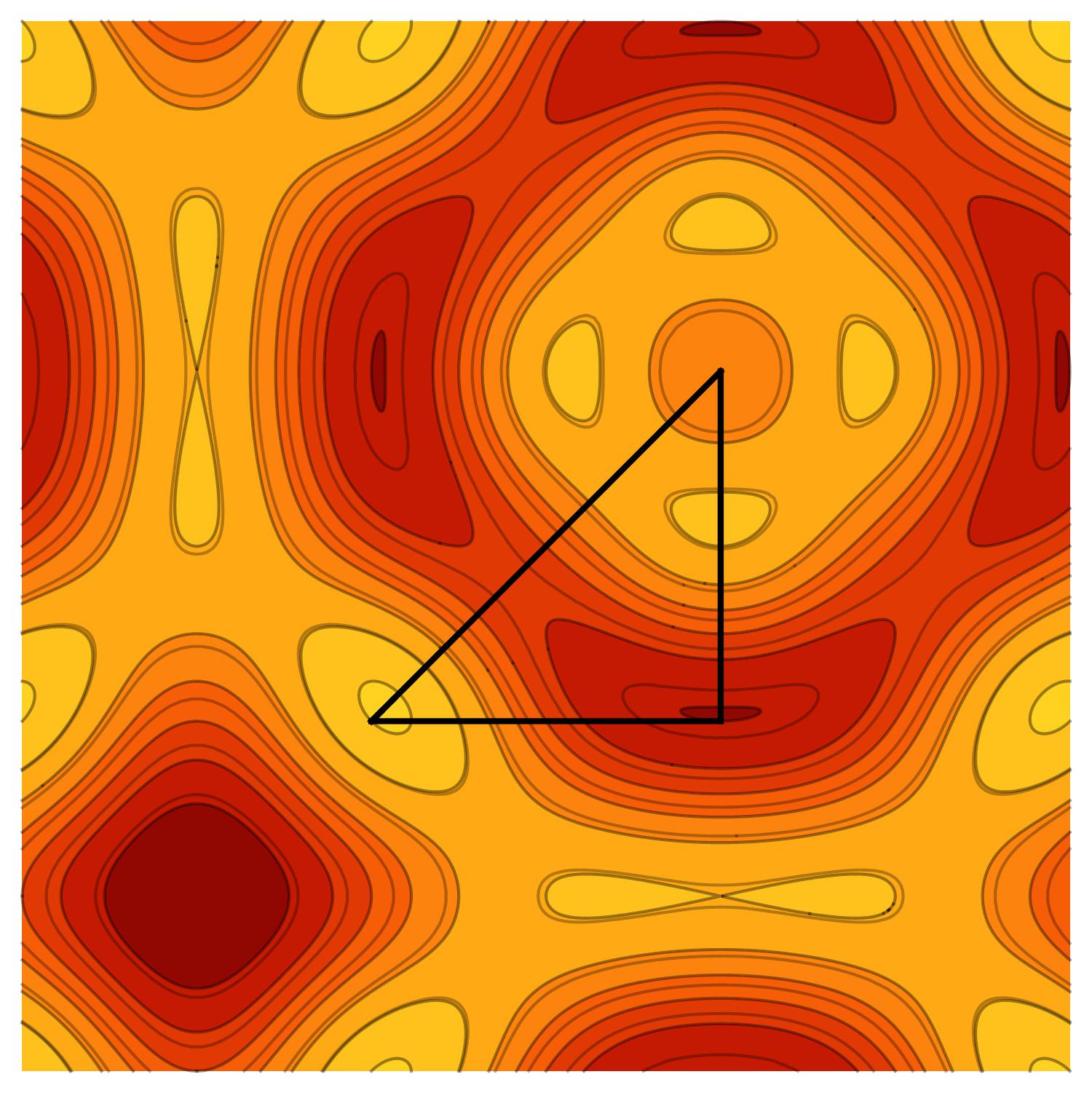}
\hspace{8pt}
\includegraphics[width=3.5cm]{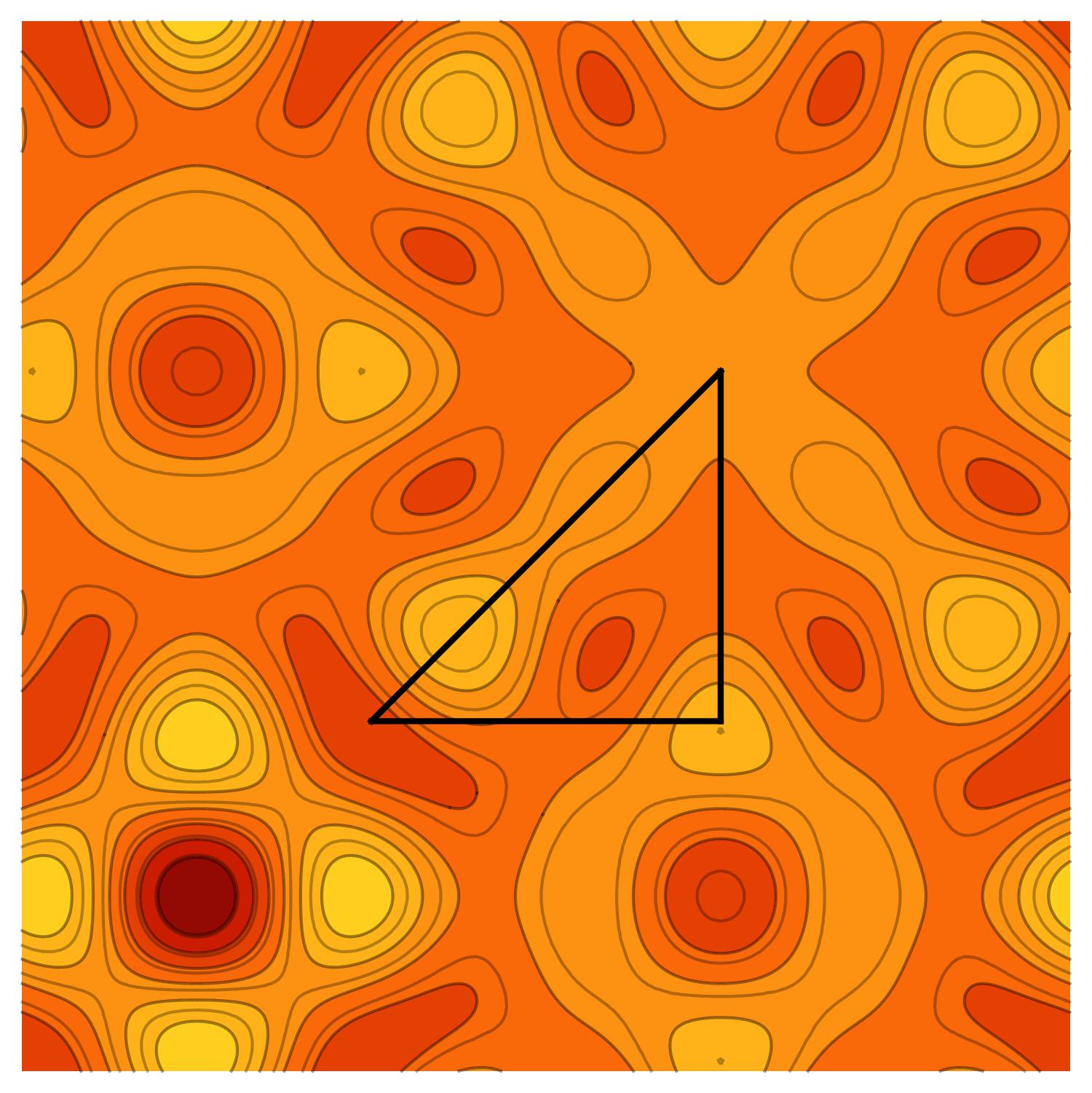}
\caption{\small The contour plots of the graph cuts ($x_3=\frac13$) of the symmetric cosine functions $\cos^+_k(x)$.  The cut of the boundary of the fundamental domain $F(\wt{S}_3^\mathrm{aff})$ is depicted as the black triangle.}
\label{coss}
\end{figure}

For several special choices of  $k\in \Z^n$, the functions \eqref{trigfun} are expressed as products of the one-dimensional cosine and sine functions \cite{KPtrig}. In addition to the formulas from \cite{KPtrig}, we calculate this form for the special cases needed for the analysis of the related orthogonal polynomials. Denoting   
\begin{align}
\rho_1\equiv & (n-1,n-2,\dots,0),\label{rho1}\\
\rho_2\equiv & \left(n-\frac12,n-\frac32,\dots,\frac32,\frac12\right), \label{rho2}
\end{align}
we derive the form of the functions  $\cos^+_\rho(x)$, $\cos^-_{\rho_1}(x)$ and $\cos^-_{\rho_2}(x)$ in the following proposition.  
\begin{tvr}\label{prop1}
Let $k\in \N$ is given by
\begin{equation}
k=\begin{cases}
\frac{n-1}{2} & \text{for $n$ odd},\\ 
\frac{n}{2}& \text{for $n$ even}.
\end{cases}
\end{equation}
Then it holds that
\begin{align}
\cos^-_{\rho_1}(x_1,\dots,x_n)&=(-1)^k2^{(n-1)^2}\prod_{1\leq i<j\leq n}\sin{\left(\frac{\pi}{2}(x_i+x_j)\right)}\sin{\left(\frac{\pi}{2}(x_i-x_j)\right)}\,,\label{lowestant}\\
\cos^-_{\rho_2}(x_1,\dots,x_n)&=(-1)^k2^{n(n-1)}\prod_{i=1}^n \cos{\left(\frac{\pi}{2}x_i\right)}\prod_{1\leq i<j\leq n}\sin{\left(\frac{\pi}{2}(x_i+x_j)\right)}\sin{\left(\frac{\pi}{2}(x_i-x_j)\right)}\,,\label{lowestanthalf}\\
\cos^+_{\rho}(x_1,\dots,x_n)&=n!\prod_{i=1}^n\cos{\left(\frac{\pi}{2}x_i\right)}\,.\label{lowestsymhalf}
\end{align}
where $\rho_1$, $\rho_2$, and $\rho$ are given by \eqref{rho1}, \eqref{rho2}, and \eqref{rho}, respectively.
\end{tvr}

\begin{proof}
We have from definition \eqref{trigfun} that 
$$\cos^-_{\rho_1}(x_1,\dots,x_n)=\det\left(\begin{smallmatrix}
\cos{(\pi(n-1)x_1)}&\cos{(\pi(n-2)x_1)}&\cdots&\cos{(\pi x_1)}&1\\
\cos{(\pi(n-1)x_2)}&\cos{(\pi(n-2)x_2)}&\cdots&\cos{(\pi x_2)}&1\\
\vdots&\vdots&&\vdots&\vdots\\
\cos{(\pi(n-1)x_n)}&\cos{(\pi(n-2)x_n)}&\cdots&\cos{(\pi x_n)}&1
\end{smallmatrix}\right)\,.$$
Using the trigonometric identity for the powers of cosine function
$$
2^{m-1}\cos^m{\theta}=\begin{cases}
\cos{(m\theta)}+\sum_{k=1}^{\frac{m-1}{2}}\binom{m}{k}\cos{(m-2k)\theta}& \text{if $m$ is odd},\\ 
\cos{(m\theta)}+\frac12{ \binom{m}{ \frac{m}{2}}}+\sum_{k=1}^{\frac{m}{2}-1}\binom{m }{ k}\cos{(m-2k)\theta}& \text{if $m$ is even},
\end{cases}
$$
we bring the determinant to the following form:
\begin{align*}
\cos^-_{\rho_1}(x_1,\dots,x_n)
&=(-1)^k2^{\frac{(n-2)(n-1)}{2}}\det\left(\begin{smallmatrix}
1&\cos{(\pi x_1)}&\cdots&\cos^{n-2}{(\pi x_1)}&\cos^{n-1}{(\pi x_1)}\\
1&\cos{(\pi x_2)}&\cdots&\cos^{n-2}{(\pi x_2)}&\cos^{n-1}{(\pi x_2)}\\
\vdots&\vdots&&\vdots&\vdots\\
1&\cos{(\pi x_n)}&\cdots&\cos^{n-2}{(\pi x_n)}&\cos^{n-1}{(\pi x_n)}\\
\end{smallmatrix}\right)\,.
\end{align*}
Taking into account that the last determinant is of the Vandermonde type and that the trigonometric identity 
$$\cos{(\pi x_j)}-\cos{(\pi x_i)}=2\sin{\left(\frac{\pi}{2}(x_i+x_j)\right)}\sin{\left(\frac{\pi}{2}(x_i-x_j)\right)}$$ holds, we obtain that
$$\cos^-_{\rho_1}(x_1,\dots,x_n)= (-1)^k2^{\frac{(n-2)(n-1)}{2}}\prod_{1\leq i<j\leq n}\left(\cos{(\pi x_j)}-\cos{(\pi x_i)}\right),$$
and the identity \eqref{lowestant} follows.
The proof of formula \eqref{lowestanthalf} is similar and
relation \eqref{lowestsymhalf} follows directly from definition \eqref{trigfun}.
\end{proof}
\noindent
Note that the equalities \eqref{lowestant}--\eqref{lowestanthalf} allow us to analyze the zeros of the functions $\cos^-_{\rho_1}$, $\cos^-_{\rho_2}$, and  $\cos^+_{\rho}$:
\begin{itemize}
\item $\cos^-_{\rho_1}(x)=0$ if and only if $x=(x_1,\dots,x_n)\in F(\wt{S}_n^{\mathrm{aff}})$ satisfies $x_i=x_{i+1}$, 
\item $\cos^-_{\rho_2}(x)=0$ if and only if $x=(x_1,\dots,x_n)\in F(\wt{S}_n^{\mathrm{aff}})$ satisfies $x_i=x_{i+1}$ or $x_1=1$, and 
\item $\cos^+_{\rho}(x)=0$ if and only if $x=(x_1,\dots,x_n)\in F(\wt{S}_n^{\mathrm{aff}})$ satisfies $x_1=1$.
\end{itemize}
Observing that all zero values are located on the boundaries of  $F(\wt{S}_n^{\mathrm{aff}})$ we conclude as follows.
\begin{cor}\label{nonzero}
The functions	$\cos^-_{\rho_1}$, $\cos^-_{\rho_2}$, and  $\cos^+_{\rho}$ are nonzero in the interior of the fundamental domain  $F(\wt{S}_n^{\mathrm{aff}})$.
\end{cor}

\subsection{Continuous orthogonality}\label{seccon}\

The antisymmetric and symmetric cosine functions \eqref{functions}  are mutually continously orthogonal within each family. Denoting the order of the stabilizer of the point $k\in \R^n $ under the action of the permutation group $S_n$ by $H_k$, i.e.,
\begin{equation}\label{Hk}
H_k= \#\set{\sigma k=k}{\sigma\in S_n}	,
\end{equation}
and introducing the symbol $h_k=h_{k_1}\dots h_{k_n}$ by
 $$h_{k_i}=\begin{cases} 1 &\text{if $k_i=0$,}\\ \frac12&\text{otherwise,} 
\end{cases}$$
we obtain
\begin{alignat}{3}
&\int_{F(\wt{S}_n^\mathrm{aff})}\cos^-_{k}(x)\cos^-_{k'}(x)\,dx &\,=\,& h_k\delta_{kk'}\,,&\quad& k,k'\in P^{++}\,, \label{O1}\\
&\int_{F(\wt{S}_n^\mathrm{aff})}\cos^-_{k+\rho}(x)\cos^-_{k'+\rho}(x)\,dx &\,=\,& 2^{-n}\delta_{kk'}\,,&\quad& k,k'\in P^{++}\,, \label{O2}\\
&\int_{F(\wt{S}_n^\mathrm{aff})}\cos^+_{k}(x)\cos^+_{k'}(x)\,dx &\, =\, & h_kH_k\delta_{kk'}\,,&\quad& k,k'\in P^{+}\,,\label{O3}\\
&\int_{F(\wt{S}_n^\mathrm{aff})}\cos^+_{k+\rho}(x)\cos^+_{k'+\rho}(x)\,dx &\, =\, & 2^{-n}H_k\delta_{kk'}\,,&\quad& k,k'\in P^{+}. \label{O4}
\end{alignat}
The orthogonality relations \eqref{O1}, \eqref{O3} are deduced in \cite{KPtrig} from the continuous orthogonality of the ordinary cosine functions $\cos{(\pi m \theta)}$,
$$\int_0^1 \cos{(\pi m \theta)}\cos{(\pi m' \theta)}=h_m\delta_{mm'},\qquad m,m'\in \Z\,.$$
The remaining two orthogonality relations follow from the continuous orthogonality of the cosine functions  $\cos{\left(\pi\left(m+1/2\right)\theta \right)}$,
\begin{equation}\label{1dimortog}
\int_0^1 \cos{\left(\pi\left(m+\frac12\right)\theta \right)}\cos{\left(\pi\left(m'+\frac12\right)\theta \right)}=\frac12\delta_{mm'}\,.
\end{equation}
Since the antisymmetric cosine functions are anti-invariant with respect to all elements of $S_n$ and the group $S_n$ applied on $F(\wt{S}_n^\mathrm{aff})$ gives the cube $[0,1]^n$, we have
$$\int_{F(\wt{S}_n^\mathrm{aff})}\cos^-_{k+\rho}(x)\cos^-_{k'+\rho}(x)\,dx=\frac{1}{n!}\int_{\left[ 0,1\right]^n }\cos^-_{k+\rho}(x)\cos^-_{k'+\rho}(x)\,dx\,.$$
Using the property $\cos^-_{k+\rho}(x)=\cos^-_{\sigma'(k+\rho)}(\sigma'(x))$ for $\sigma'\in S_n$, which follows from \eqref{perm}, together with definition \eqref{trigfun}, we have 
$$
\cos^-_{k+\rho}(x)\cos^-_{k'+\rho}(x)=\sum_{\sigma,\sigma'\in S_n}\mathrm{sgn}\,(\sigma\sigma')\prod_{i=1}^n\cos{\left(\pi\left(k_{\sigma\sigma'(i)}+\frac12\right)x_{\sigma'(i)}\right)}
\cos{\left(\pi\left(k'_i+\frac12\right)x_{\sigma'(i)}\right)}\,.
$$
Therefore, we rewrite the integral \eqref{O2} as
$$\int_{F(\wt{S}_n^\mathrm{aff})}\cos^-_{k+\rho}(x)\cos^-_{k'+\rho}(x)\,dx=\sum_{\sigma\in S_n}\mathrm{sgn}\,(\sigma)\prod_{i=1}^n\int_0^1\cos{\left(\pi\left(k_{\sigma(i)}+\frac12\right)x_i\right)}
\cos{\left(\pi\left(k'_i+\frac12\right)x_i\right)}dx_i\,,$$
where we made the change of variables from $x_{\sigma'(i)}$ to $x_i$. Using the relation \eqref{1dimortog}, we obtain
$$\int_{F(\wt{S}_n^\mathrm{aff})}\cos^-_{k+\rho}(x)\cos^-_{k'+\rho}(x)\,dx=\sum_{\sigma\in S_n}\mathrm{sgn}\,(\sigma)\prod_{i=1}^n\frac12\delta_{k_{\sigma(i)}k'_i}\,.$$
Since $k_{\sigma(i)}=k'_i$ for all $ i \in \{1,\dots,n\}$ if and only if $\sigma$ is the identity permutation and $k_i=k'_i$, the orthogonality relation \eqref{O2} is derived. The orthogonality relation \eqref{O4} is obtained similarly.

\section{Discrete multivariate cosine transforms}\label{DMCT}
\subsection{DCTs of types V--VIII}\label{DCT}\

The one-dimensional DCTs and their Cartesian product generalizations have many applications in mathematics, physics, and engineering. They are of efficient use in various domains of information processing, e.g., image, video, or audio processing. All types of DCTs arise from discretized solutions of the harmonic oscillator equation with different choices of boundary conditions applied on its two-point boundary \cite{Brit,Strang}. These conditions moreover can be applied at grid points or at mid-grid points which are distributed uniformly in the one-dimensional interval. Requiring the Neumann conditions at both points of the boundary at grid points generates the transform DCT I; the Neumann conditions at both points applied at mid-grid generate DCT II. Requiring the Neumann condition at one point and the Dirichlet condition at the other point at grid points generates the transform DCT III; the Neumann condition and the Dirichlet conditions applied at mid-grid generate DCT IV. Note that for these four types of DCTs both boundary conditions are always applied simultaneously at grid points or at mid-grid -- the other DCTs of types V--VIII are obtained similarly by applying one condition at grid point and the other at mid-grid. The DCTs express a function defined on a finite grid of points as a sum of cosine functions with different frequencies and amplitudes.  They have numerous useful properties as the unitary property and convolution properties \cite{Brit}. The multivariate symmetric and antisymmetric generalizations of the cosine transforms of types I--IV are found in \cite{HP,KPtrig}.  
In order to derive the symmetric and antisymmetric DCTs of types V--VIII, which have not so far appeared in the literature, we first review their one-dimensional versions. 
\subsubsection{DCT V}\

For $N\in \N$ are the cosine functions $\cos{\left(\pi k s\right)}$, $k=0,\dots,N-1$, 
defined on the finite grid of points 
\begin{equation}\label{gridV}
s\in\setb{\frac{2r}{2N-1}}{r=0,\dots,N-1}
\end{equation} 
pairwise discretely orthogonal,
\begin{equation}\label{DCTV}
\sum_{r=0}^{N-1}c_r\cos{\left(\frac{2\pi k r}{2N-1}\right)}\cos{\left(\frac{2\pi k' r}{2N-1}\right)}=\frac{2N-1}{4c_k}\delta_{kk'}\,,
\end{equation}
where \begin{equation}\label{coefc}
c_r=\begin{cases}\frac12&\text{if $r=0$ or $r=N$},\\
 1& \text{otherwise}.
 \end{cases}
 \end{equation}
Therefore, any discrete function $f$ given on the finite grid \eqref{gridV} is expressed in terms of cosine functions as
\begin{equation}\label{tranV}
f(s)=\sum_{k=0}^{N-1}A_k\cos{\left(\pi k s\right)},\quad A_k=\frac{4c_k}{2N-1}\sum_{r=0}^{N-1}c_rf\left(\frac{2 r}{2N-1}\right)\cos{\left(\frac{2\pi k r}{2N-1}\right)}\,.
\end{equation}
The formulas \eqref{tranV} determine the transform DCT V.

\subsubsection{DCT VI}\

For $N\in \N$ are the cosine functions $\cos{\left(\pi k s\right)}$, $k=0,\dots,N-1$,
defined on the finite grid of points 
\begin{equation}\label{gridVI}
s\in\setb{\frac{2\left(r+\frac12\right)}{2N-1}}{r=0,\dots,N-1}
\end{equation}
pairwise orthogonal,
\begin{equation}\label{DCTVI}
\sum_{r=0}^{N-1}c_{r+1}\cos{\left(\frac{2\pi k \left(r+\frac12\right)}{2N-1}\right)}\cos{\left(\frac{2\pi k' \left(r+\frac12\right)}{2N-1}\right)}=\frac{2N-1}{4c_k}\delta_{kk'}\,,
\end{equation}
where $c_r$ are determined by \eqref{coefc}. Therefore, any discrete function $f$ given on the finite grid \eqref{gridVI} is expressed in terms of cosine functions as
\begin{equation}\label{tranVI}f(s)=\sum_{k=0}^{N-1}A_k\cos{\left(\pi k s\right)},\quad A_k=\frac{4c_k}{2N-1}\sum_{r=0}^{N-1}c_{r+1}f\left(\frac{2 \left(r+\frac12\right)}{2N-1}\right)\cos{\left(\frac{2\pi k \left(r+\frac12\right)}{2N-1}\right)}\,.
\end{equation}
The formulas \eqref{tranVI} determine the transform DCT VI.

\subsubsection{DCT VII}\

For $N\in \N$ are the cosine functions $\cos{\left(\pi \left(k+\frac12\right) s\right)}$, $k=0,\dots,N-1$,
defined on the finite grid of points \eqref{gridV},
pairwise discretely orthogonal,
\begin{equation}\label{DCTVII}
\sum_{r=0}^{N-1}c_r\cos{\left(\frac{2\pi \left(k+\frac12\right) r}{2N-1}\right)}\cos{\left(\frac{2\pi \left(k'+\frac12\right) r}{2N-1}\right)}=\frac{2N-1}{4c_{k+1}}\delta_{kk'}\,,
\end{equation}
where $c_r$ is determined by \eqref{coefc}. Therefore, any discrete function $f$ given on the finite grid \eqref{gridV} is expressed in terms of cosine functions as
\begin{equation}\label{tranVII}
f(s)=\sum_{k=0}^{N-1}A_k\cos{\left(\pi \left(k+\frac12\right) s\right)},\quad A_k=\frac{4c_{k+1}}{2N-1}\sum_{r=0}^{N-1}c_rf\left(\frac{2 r}{2N-1}\right)\cos{\left(\frac{2\pi \left(k+\frac12\right) r}{2N-1}\right)}\,.
\end{equation}
The formulas \eqref{tranVII} determine the transform DCT VII.

\subsubsection{DCT VIII}\

For $N\in \N$ are the cosine functions
$\cos{\left(\pi \left(k+\frac12\right)s\right)}$, $k=0,\dots,N-1$,
defined on the finite grid of points 
\begin{equation}\label{gridVIII}
s\in\setb{\frac{2\left(r+\frac12\right)}{2N+1}}{r=0,\dots,N-1}
\end{equation} pairwise orthogonal,
\begin{equation}\label{DCTVIII}
\sum_{r=0}^{N-1}\cos{\left(\frac{2\pi \left(k+\frac12\right) \left(r+\frac12\right)}{2N+1}\right)}\cos{\left(\frac{2\pi \left(k'+\frac12\right) \left(r+\frac12\right)}{2N+1}\right)}=\frac{2N+1}{4}\delta_{kk'}\,.
\end{equation}
Therefore, any discrete function $f$ given on the finite grid \eqref{gridVIII} is expressed in terms of cosine functions as
\begin{equation}\label{tranVIII}f(s)=\sum_{k=0}^{N-1}A_k\cos{\left(\pi \left(k+\frac12\right) s\right)},\quad A_k=\frac{4}{2N+1}\sum_{r=0}^{N-1}f\left(\frac{2 \left(r+\frac12\right)}{2N+1}\right)\cos{\left(\frac{2\pi \left(k+\frac12\right) \left(r+\frac12\right)}{2N+1}\right)}\,.
\end{equation}
The formulas \eqref{tranVIII} determine the transform DCT VIII.

\subsection{Antisymmetric discrete multivariate cosine transforms}\label{AMDCT}\

The antisymmetric discrete multivariate cosine transforms (AMDCTs), which can be viewed as antisymmetric multivariate generalizations of DCTs, are derived using the one-dimensional DCTs from Section \ref{DCT}. The four types of AMDCT, connected to DCTs I--IV, are contained in \cite{KPtrig}. Our goal is to complete the list of AMDCTs by developing the remaining four transforms of types V--VIII. First, we introduce the set of labels 
\begin{equation}
D_N\equiv\set{(k_1,\dots,k_n)\in\Z^n}{N-1\geq k_i\geq0,\, i=1,\dots,n}
\end{equation} 
and to any point $(k_1,\dots,k_n)\in D_N$ we assign two values
\begin{align}
d_k &= c_{k_1}\dots c_{k_n},\label{dk}\\
\wt d_k &= c_{k_1+1}\dots c_{k_n+1},\label{dkt}
\end{align}
 where  $c_{k_i}$ are determined by \eqref{coefc}.
The discrete calculus is performed on three types of grids inside $F(\wt{S}_n^\mathrm{aff})$; two of them are subsets of  two cubic grids, 
\begin{align}
 C_N\equiv &\setb{\left(\frac{2r_1}{2N-1},\dots,\frac{2r_n}{2N-1}\right)}{  (r_1,\dots,r_n)\in D_N }, \label{CN} \\ 
\wt C_N\equiv &\setb{\left(\frac{2\left(r_1+\frac12\right)}{2N-1},\dots,\frac{2\left(r_n+\frac12\right)}{2N-1}\right)}{(r_1,\dots,r_n)\in D_N}.\label{CNt}
\end{align} 
To any point $s\in C_N$, which is according to \eqref{CN} labeled by the point $(r_1,\dots ,r_n)\in D_N$, we assign the value
\begin{equation}\label{ep}
	\ep_s\equiv c_{r_1}\dots c_{r_n},
\end{equation}
and to any point $s\in \wt C_N$, which is according to \eqref{CNt} labeled by the point $(r_1,\dots ,r_n)\in D_N$, we assign the value
\begin{equation}\label{eps}
	\wt \ep_s\equiv c_{r_1+1}\dots c_{r_n+1}.
\end{equation}
In the following subsection we detail the proof of the AMDCT V transform. The proofs of the remaining transforms VI--VIII are similar. 
\subsubsection{AMDCT V}\

For $N\in \N$ we consider the antisymmetric cosine functions $\cos^-_k(s)$ labeled by the index set
\begin{equation}\label{parV}
D_N^-\equiv\set{(k_1,\dots,k_n)\in D_N}{ k_1>k_2>\dots>k_n}
\end{equation} 
and restricted to the finite grid of points contained in $C_N \subset F(\wt{S}^{\mathrm{aff}}_n)$,
\begin{equation}\label{funV}
 F_N^{\mathrm{V},-}\equiv\set{\left(\frac{2r_1}{2N-1},\dots,\frac{2r_n}{2N-1}\right)}{(r_1,\dots ,r_n)\in D^-_N} \,.
\end{equation} 
The scalar product of any two functions $f,g:F_N^{\mathrm{V},-}\map\R$ given on the points of the grid $F^{\mathrm{V},-}_N$ is defined by
\begin{equation}\label{scalV}
\langle f,g\rangle=\sum_{s\in F_N^{\mathrm{V},-}}\ep_s f(s)g(s)\,,
\end{equation} 
where $\ep_s$ is given by \eqref{ep}.
Using the orthogonality relation of one-dimensional cosine functions \eqref{DCTV}, we show that the antisymmetric cosine functions labeled by parameters in $k,k'\in D_N^-,$ are pairwise discretely orthogonal, i.e.,
\begin{equation}\label{AMDCTV}
\langle\cos_k^-,\cos_{k'}^-\rangle=\sum_{s\in F_N^{\mathrm{V},-}}\ep_s\cos_k^-(s)\cos_{k'}^-(s)=d_k^{-1}\left(\frac{2N-1}{4}\right)^n\delta_{kk'},
\end{equation}
where $d_k$ is given by \eqref{dk}.
Denoting the set of points 
\begin{equation}\label{funVN}
 F_N\equiv\set{\left(\frac{2r_1}{2N-1},\dots,\frac{2r_n}{2N-1}\right)\in C_N}{ r_1\geq r_2\geq \dots\geq r_n} \,,
\end{equation} 
we observe that the functions $\cos_k^-(s)$, $k\in D_N^-$ vanish on the points $s\in F_N \setminus F_N^{\mathrm{V},-}$ and thus $$\langle\cos_k^-,\cos_{k'}^-\rangle=\sum_{s\in F_N}\ep_s H^{-1}_s \cos_k^-(s)\cos_{k'}^-(s),$$
with the order of the stabilizer $H_s$ determined by \eqref{Hk}.
Acting by all permutations of $S_n$ on the grid $F_N$ we obtain the entire cube $C_N$, i.e., 
$S_nF_N=C_N$,
and therefore we have 
$$\langle\cos_k^-,\cos_{k'}^-\rangle=\frac{1}{n!}\sum_{r_1,\dots,r_n=0}^{N-1}c_{r_1}\dots c_{r_n}\cos_k^-\left(\frac{2r_1}{2N-1},\dots, \frac{2r_n}{2N-1}\right)\cos_{k'}^-\left(\frac{2r_1}{2N-1},\dots, \frac{2r_n}{2N-1}\right).$$
Moreover, the product of two antisymmetric cosine functions labeled by $k,k'\in D_N^-$ is rewritten due to the anti-invariance under permutations as 
$$\cos^-_{k}(x)\cos^-_{k'}(x)=\sum_{\sigma,\sigma'\in S_n}\mathrm{sgn}\,(\sigma\sigma')\prod_{i=1}^n\cos{\left(\pi k_{\sigma\sigma'(i)}x_{\sigma'(i)}\right)}
\cos{\left(\pi k'_i x_{\sigma'(i)}\right)}\,.$$
Together with the $S_n$-invariance of the set $D_N$ this implies that
$$
\langle\cos_k^-,\cos_{k'}^-\rangle=\sum_{\sigma\in S_n}\mathrm{sgn}\,{\sigma}\prod_{i=1}^n\sum_{r_i=0}^{N-1}c_{r_i}\cos{\left(\frac{2\pi k_{\sigma(i)}r_i}{2N-1}\right)}\cos{\left(\frac{2\pi k'_{i}r_i}{2N-1}\right)}\,.
$$
Finally, we apply the orthogonality relation of one-dimensional cosine functions \eqref{DCTV} to obtain \eqref{AMDCTV},
$$\langle\cos_k^-,\cos_{k'}^-\rangle=\sum_{\sigma\in S_n}\mathrm{sgn}\,(\sigma)\prod_{i=1}^n\frac{2N-1}{4
c_{k_{\sigma(i)}}}\delta_{k_{\sigma(i)}k'_i}=d_k^{-1}\left(\frac{2N-1}{4}\right)^n\delta_{kk'}\,.$$

Due to the relation \eqref{AMDCTV}, we expand any function $f: F_N^{\mathrm{V},-}\map\R$ in terms of antisymmetric cosine functions as follows:
$$f(s)=\sum_{k\in D_N^-}A_k\cos^-_k(s)\quad\text{with}\quad A_k=d_k\left(\frac{4}{2N-1}\right)^n\sum_{s\in F_N^{\mathrm{V},-}}\ep_sf(s)\cos_k^-(s)\,.$$
Validity of the expansion follows from the fact that the number of points in $D_N^-$ is equal to the number of points in $F^{\mathrm{V},-}_N$ and thus the functions $\cos^-_k$ with $k\in D_N^-$ form an orthogonal basis of the space of all functions $f:F^{\mathrm{V},-}_N\map\R$ with the scalar product \eqref{scalV}. The remaining transforms AMDCT VI--VIII are deduced similarly.

\subsubsection{AMDCT VI}\

For $N\in \N$ we consider the antisymmetric cosine functions $\cos^-_k(s)$ labeled by the index set $k\in D^-_N$ and
restricted to the finite grid of points contained in $\wt C_N\subset F(\wt{S}^{\mathrm{aff}}_n)$, 
\begin{equation}\label{funVI}
F^{\mathrm{VI},-}_N\equiv\set{\left(\frac{2\left(r_1+\frac12\right)}{2N-1},\dots,\frac{2\left(r_n+\frac12\right)}{2N-1}\right)}{(r_1,\dots,r_n)\in D^-_N}.
\end{equation}
The antisymmetric cosine functions labelled by parameters in $k,k'\in D_N^-,$ are pairwise discretely orthogonal on the grid $F^{\mathrm{VI},-}_N$, i.e.,
\begin{equation}
\sum_{s\in F_N^{\mathrm{VI},-}}\wt\ep_s\cos_k^-(s)\cos_{k'}^-(s)=d_k^{-1}\left(\frac{2N-1}{4}\right)^n\delta_{kk'},
\end{equation}
where $\wt \ep_s$ is given by \eqref{eps}.
Therefore, we expand any function $f: F_N^{\mathrm{VI},-}\map\R$ in terms of antisymmetric cosine functions as
$$f(s)=\sum_{k\in D^-_N}A_k\cos^-_k(s)\quad\text{with}\quad A_k=d_k\left(\frac{4}{2N-1}\right)^n\sum_{s\in F^{\mathrm{VI},-}_N}\wt \ep_sf(s)\cos_k^-(s)\,.$$

\subsubsection{AMDCT VII}\

For  $N\in \N$ are the antisymmetric cosine functions $\cos^-_{k+\rho}(s)$ , $k\in D^-_N$, 
and $\rho=\left(\frac12,\dots,\frac12\right)$ restricted to the finite grid of points  $F^{\mathrm{VII},-}_N\equiv F_N^{\mathrm{V},-}\subset C_N$ pairwise discretely orthogonal with respect to the scalar product \eqref{scalV}, i.e.,
\begin{equation}\label{AMDCTVII}
\sum_{s\in F_N^{\mathrm{VII},-}}\ep_s\cos_{k+\rho}^-(s)\cos_{k'+\rho}^-(s)=\wt d_k^{-1}\left(\frac{2N-1}{4}\right)^n\delta_{kk'}\,,
\end{equation}
where $\wt d_k$ is given by \eqref{dkt}.
Therefore, we expand any function $f: F_N^{\mathrm{VII},-}\map\R$ in terms of antisymmetric cosine functions as follows:
$$f(s)=\sum_{k\in D^-_N}A_k\cos^-_{k+\rho}(s)\quad\text{with}\quad A_k=\wt d_k\left(\frac{4}{2N-1}\right)^n\sum_{s\in F^{\mathrm{VII},-}_N}\ep_sf(s)\cos_{k+\rho}^-(s)\,.$$

\subsubsection{AMDCT VIII}\

For  $N\in \N$ are the antisymmetric cosine functions $\cos^-_{k+\rho}(s)$ , $k\in D^-_N$, 
and $\rho=\left(\frac12,\dots,\frac12\right)$ restricted to the finite grid of points 
\begin{equation}\label{gridVIII-}
F^{\mathrm{VIII},-}_N\equiv\set{\left(\frac{2\left(r_1+\frac12\right)}{2N+1},\dots,\frac{2\left(r_n+\frac12\right)}{2N+1}\right)}{(r_1,\dots,r_n)\in D^-_N}
\end{equation} pairwise discretely orthogonal, i.e., for any $k,k'\in D^-_N$ it holds that
\begin{equation}
\sum_{s\in F^{\mathrm{VIII},-}_N}\cos_{k+\rho}^-(s)\cos_{k'+\rho}^-(s)=\left(\frac{2N+1}{4}\right)^n\delta_{kk'}\,.
\end{equation}
Therefore, we expand any function $f: F_N^{\mathrm{VIII},-}\map\R$ in terms of antisymmetric cosine functions as follows:
$$f(s)=\sum_{k\in D^-_N}A_k\cos^-_{k+\rho}(s)\quad\text{with}\quad A_k=\left(\frac{4}{2N+1}\right)^n\sum_{s\in F^{\mathrm{VIII},-}_N}f(s)\cos_{k+\rho}^-(s)\,.$$

\subsection{Interpolations by antisymmetric multivariate cosine functions}\

Suppose we have a real-valued function $f$ given on $F(\wt{S}^{\mathrm{aff}}_n)$. In Section \ref{AMDCT} we defined three finite grids in $F(\wt{S}^{\mathrm{aff}}_n)$, namely, $F^{\mathrm{VI},-}_N, F^{\mathrm{V},-}_N\equiv F^{\mathrm{VII},-}_N$, and $F^{\mathrm{VIII},-}_N$. We are interested in finding the interpolating polynomial of $f$ in the form of a finite sum of antisymmetric  multivariate cosine functions labeled by $k$ or by $k+\rho$, $k\in D^-_N$ in such a way that it coincides with $f$ on one of the grids in $F(\wt{S}^{\mathrm{aff}}_n)$. We distinguish between four types of interpolating polynomials defined for $x\in F(\wt{S}^{\mathrm{aff}}_n)$ and satisfying different conditions:
\begin{equation}
\begin{alignedat}{3}
&\psi^{\mathrm{V},-}_N(x)&=&\sum_{k\in D^-_N}B_k\cos_k^-(x)\,,&\qquad\qquad& \psi^{\mathrm{V},-}_N(s)=f(s)\,,\quad s\in F^{\mathrm{V},-}_N\,,\\
&\psi^{\mathrm{VI},-}_N(x)&=&\sum_{k\in D^-_N}B_k\cos_{k+\rho}^-(x)\,,&\qquad\qquad& \psi^{\mathrm{VI},-}_N(s)=f(s)\,,\quad s\in F^{\mathrm{VI},-}_N\,,\\
&\psi^{\mathrm{VII},-}_N(x)&=&\sum_{k\in D^-_N}B_k\cos_k^-(x)\,,&\qquad\qquad& \psi^{\mathrm{VII},-}_N(s)=f(s)\,,\quad s\in F^{\mathrm{VII},-}_N\,,\\
&\psi^{\mathrm{VIII},-}_N(x)&=&\sum_{k\in D^-_N}B_k\cos_{k+\rho}^-(x)\,,&\qquad\qquad& \psi^{\mathrm{VIII},-}_N(s)=f(s)\,,\quad s\in F^{\mathrm{VIII},-}_N\,.\\
\end{alignedat}
\end{equation} 
According to Section \ref{AMDCT}, the coefficients $B_k$ are chosen to be equal to corresponding $A_k$. In fact, it is not possible to have other values of the coefficients $B_k$ since it would contradict the fact that the antisymmetric cosine functions labeled by $k$ or by $k+\rho$ with $k\in D^-_N$ are basis vectors of the space of functions given on the corresponding grid.

\begin{example}

For $n=3$ we choose a model function $f$ and interpolate it by $\psi^{\mathrm{V},-}_N(x,y,z)$ and $\psi^{\mathrm{VII},-}_N(x,y,z)$ for various values of $N$. The model function $f$ is a smooth function given by
\begin{equation}\label{modfun}
f_{\alpha,\beta,(x_0,y_0,z_0)}(x,y,z)\equiv\exp{\left(-\frac{(x-x_0)^2+(y-y_0)^2+(z-z_0)^2)}{2\alpha^2}+\beta\right)}
\end{equation}
with fixed values of the parameters $\alpha,\beta,(x_0,y_0,z_0)$ as
\begin{equation}\label{vp}
\alpha=0.079\,,\quad\beta=3\,,\quad(x_0,y_0,z_0)=\left(0.8,0.54,0.3\right).
\end{equation}
Figure \ref{model} contains the function $f$ for fixed parameters \eqref{vp} in the fundamental domain $F(\wt{S}_3^{\mathrm{aff}})$ and cut by the plane $z=\frac{1}{3}$. 
\begin{figure}
\includegraphics[width=2.5cm]{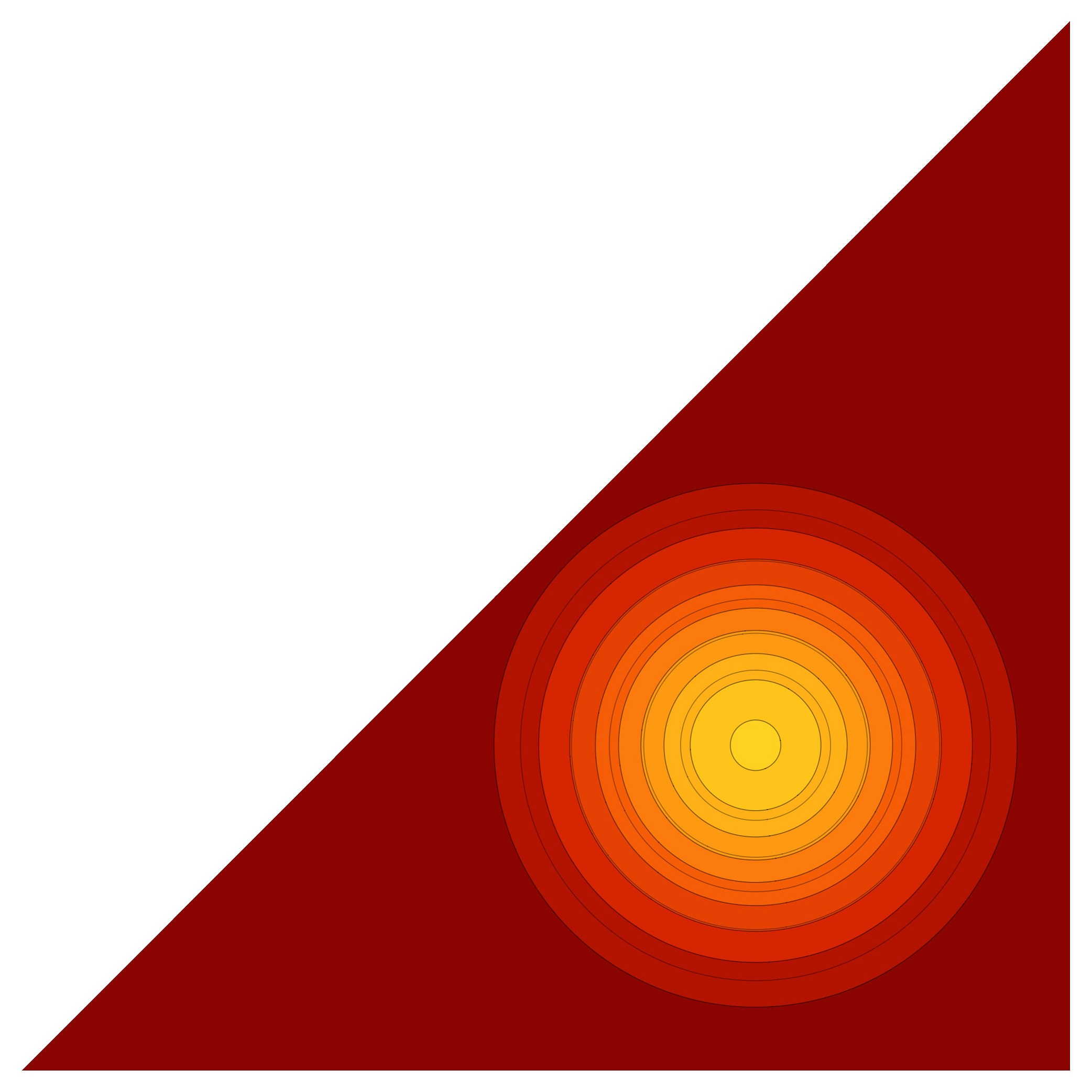}
\includegraphics[width=4cm]{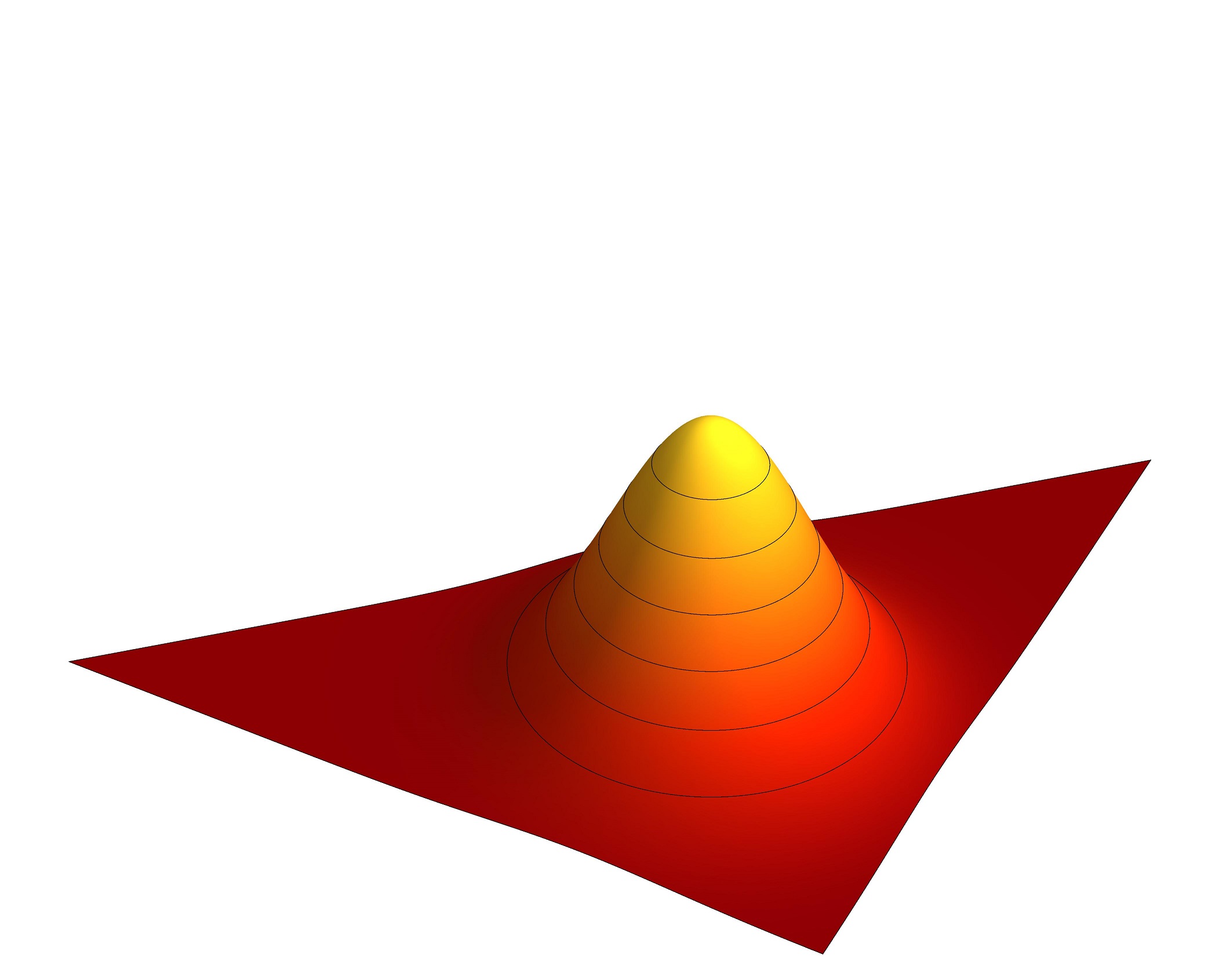}
\caption{The cut of the model function \eqref{modfun} with fixed parameters \eqref{vp} and $z=\frac{1}{3}$.}
\label{model}
\end{figure}
The antisymmetric interpolating polynomials of type V and VII for $N=5,10,15$ are depicted in Figures \ref{fig1} and \ref{fig2}.
\begin{figure}
\includegraphics[width=2.5cm]{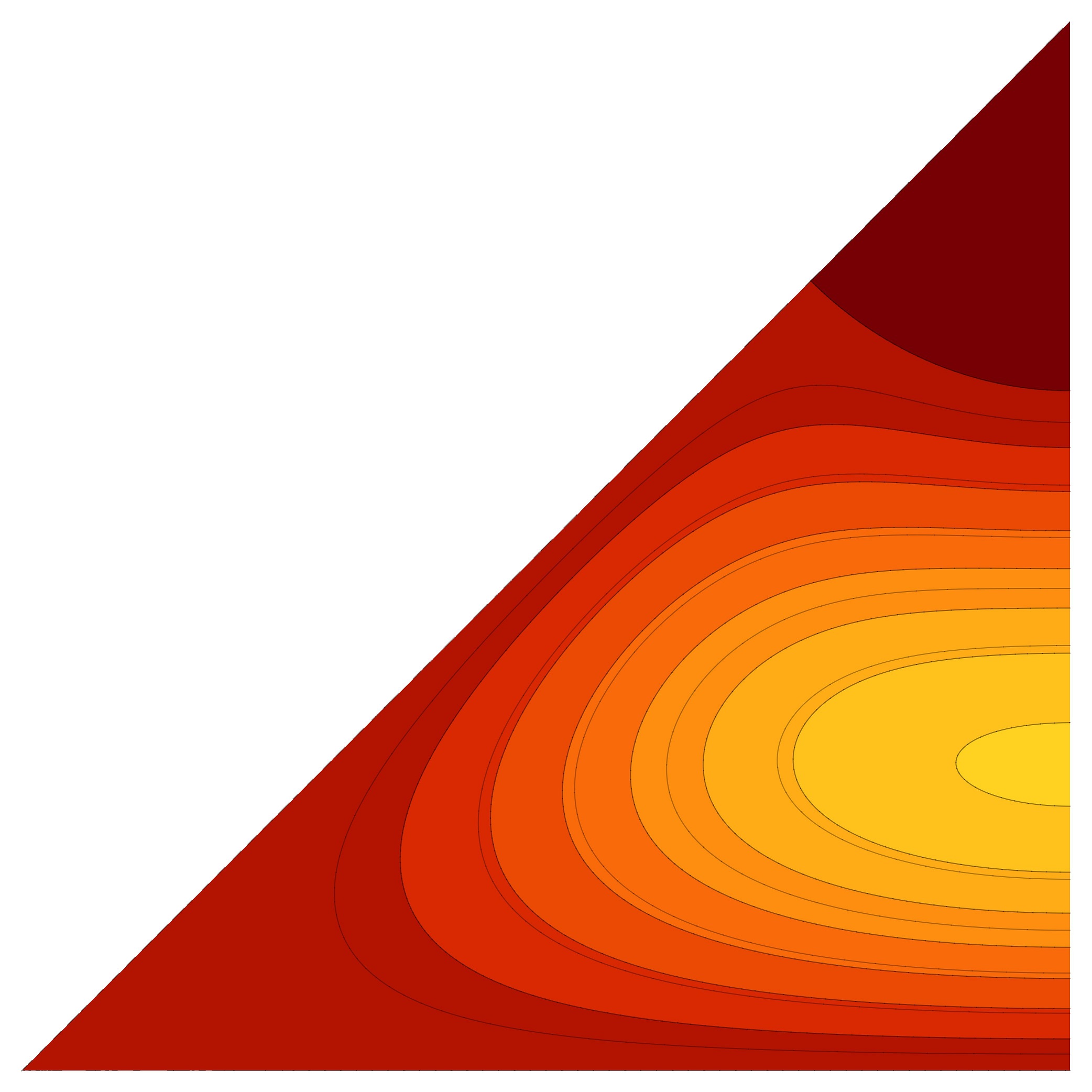}
\hspace{1.3cm}
\includegraphics[width=2.5cm]{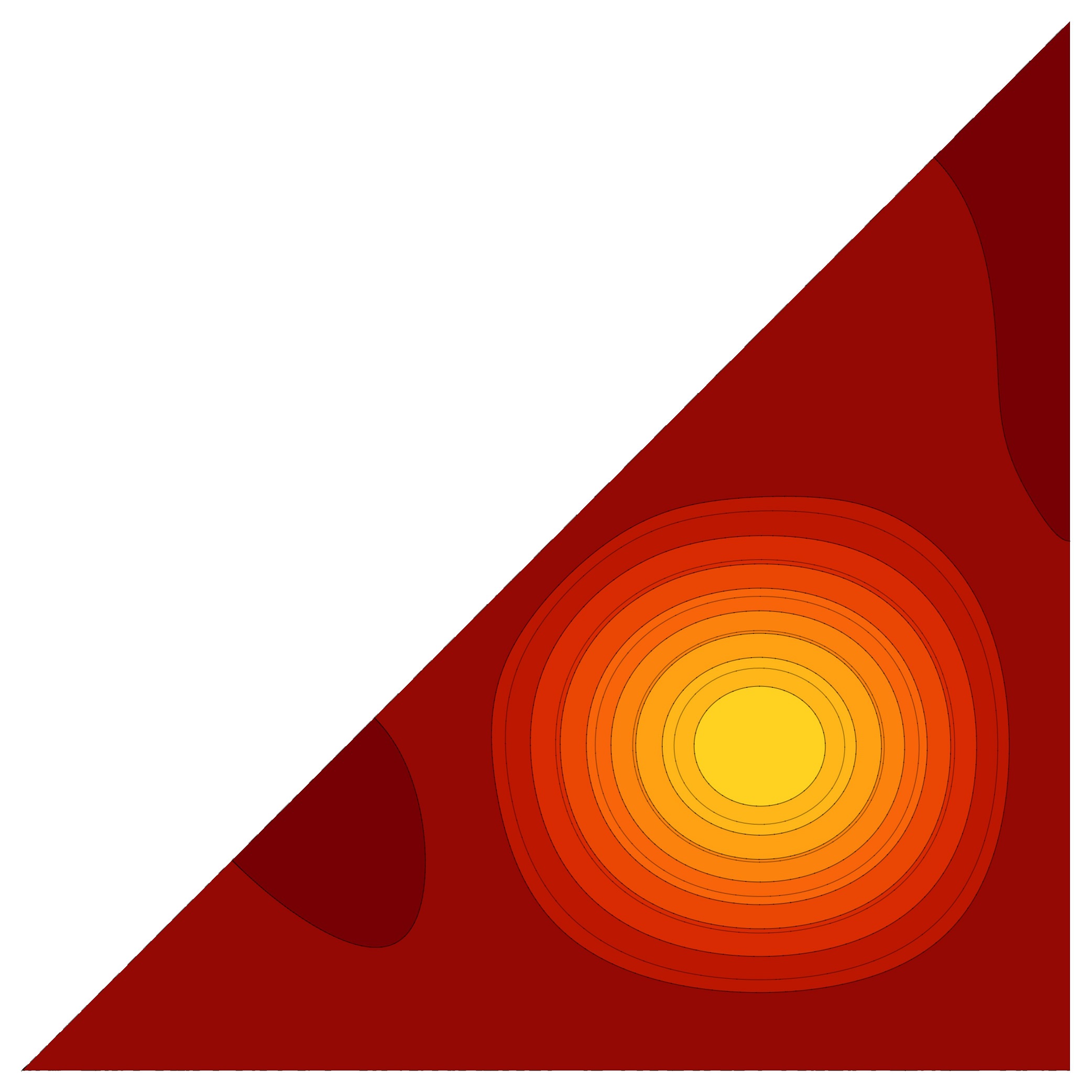}
\hspace{1.3cm}
\includegraphics[width=2.5cm]{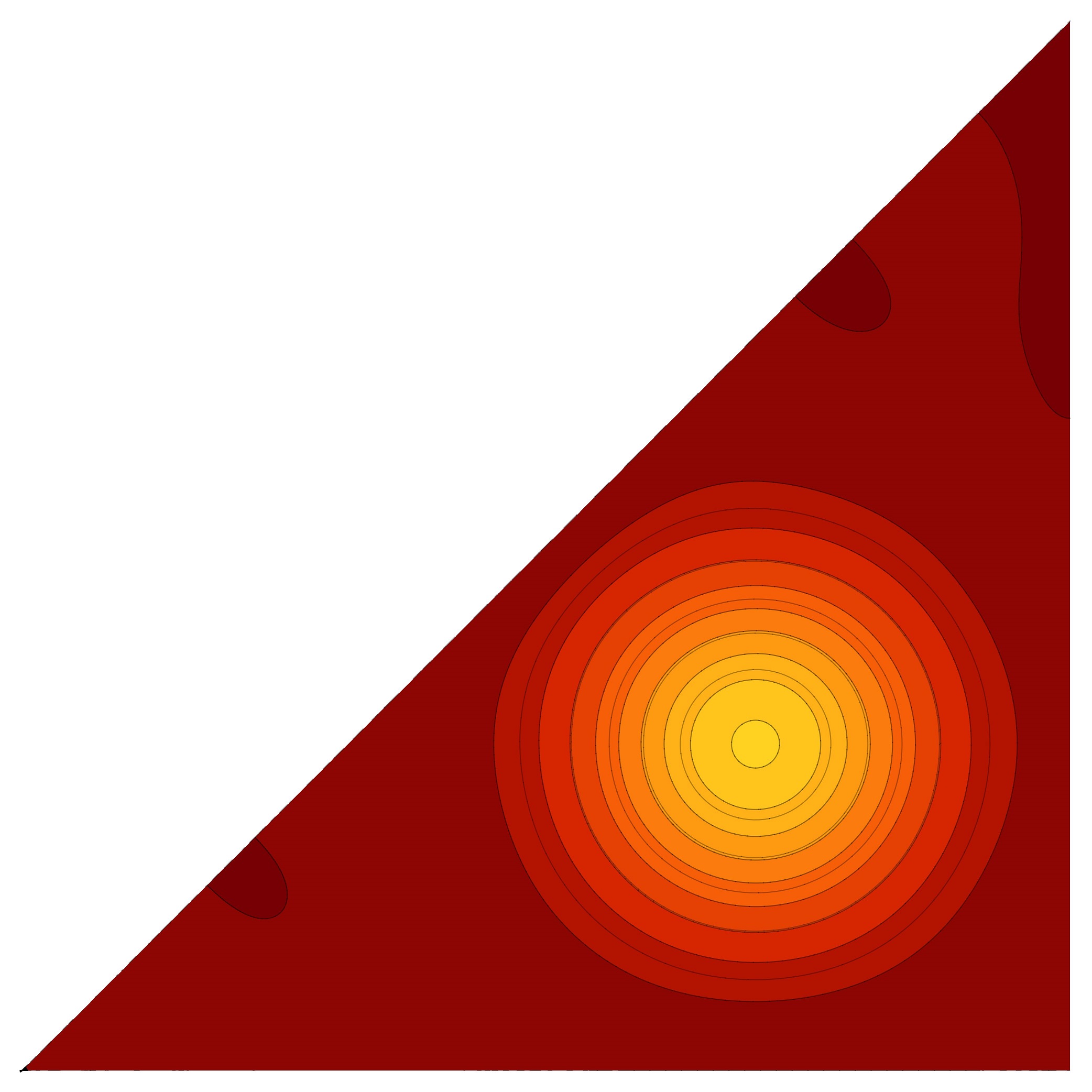}
\\
\includegraphics[width=4cm]{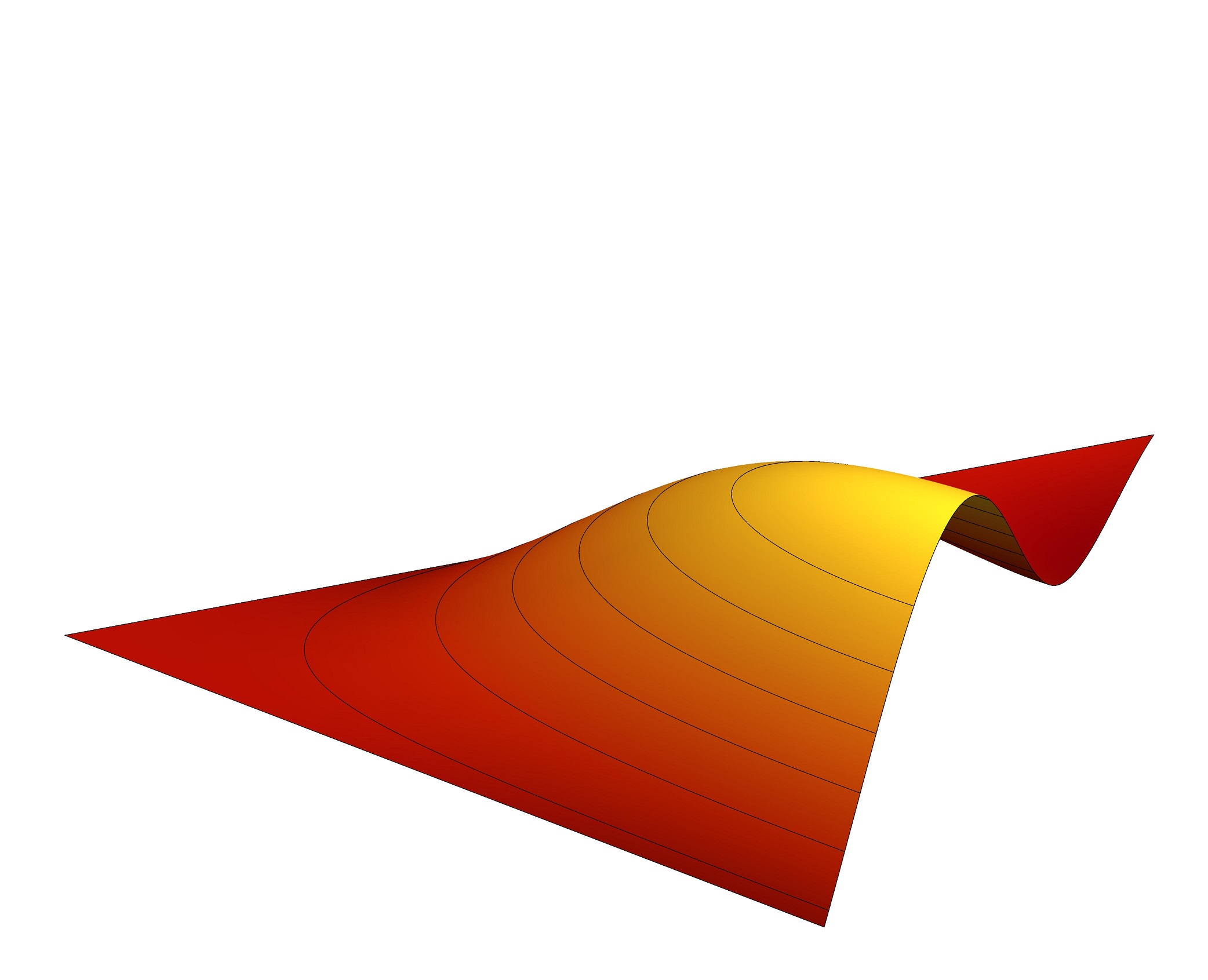}
\includegraphics[width=4cm]{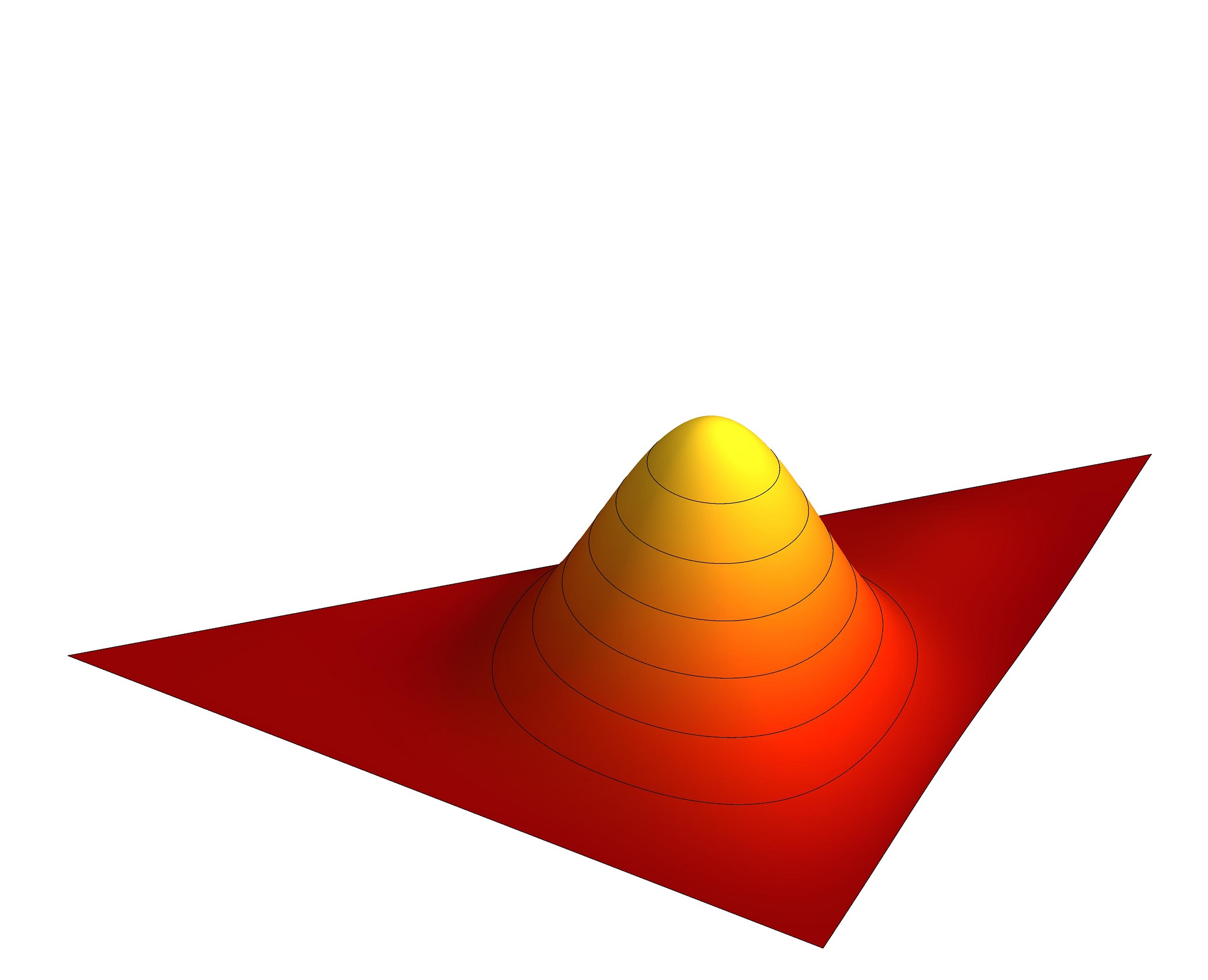}
\includegraphics[width=4cm]{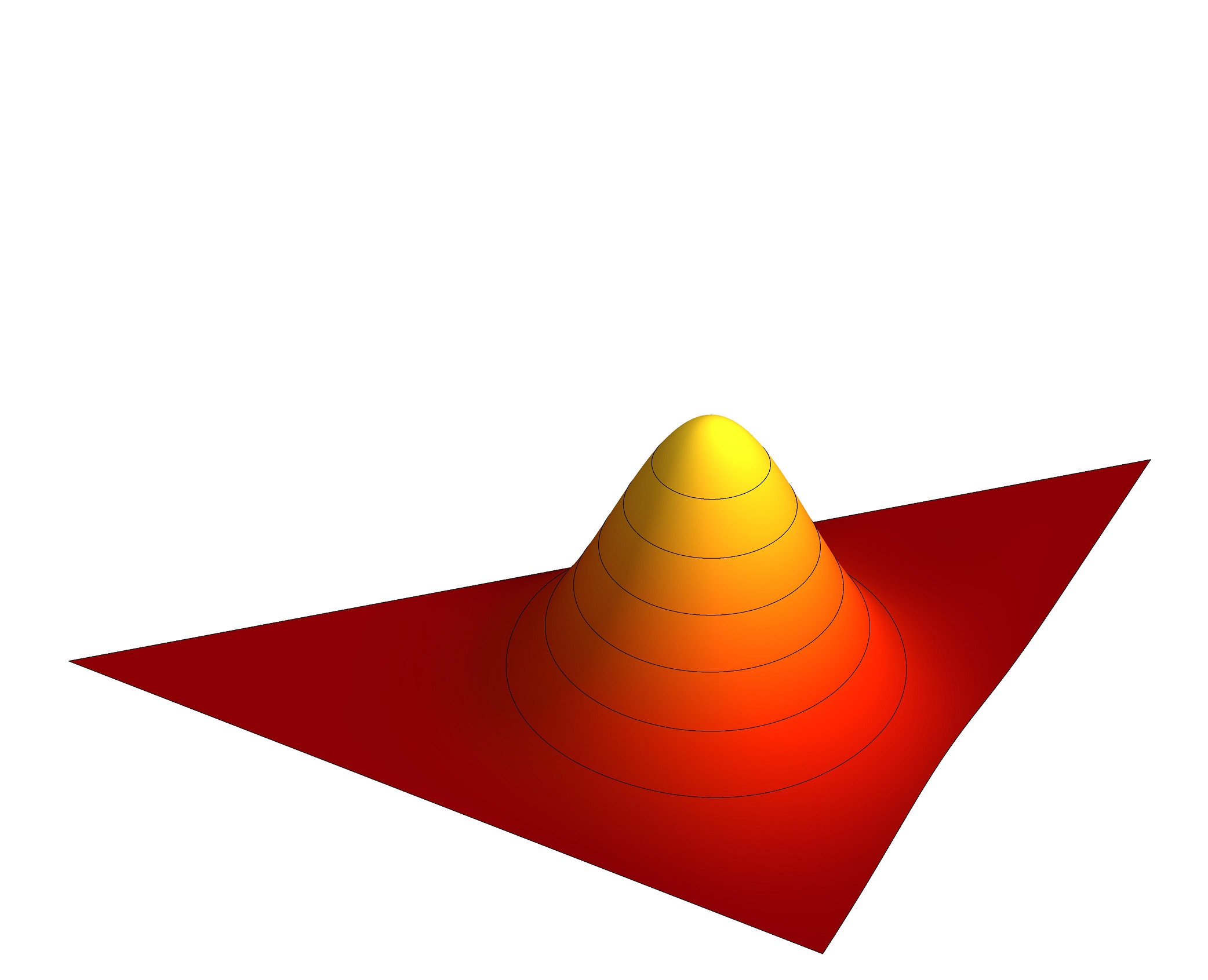}
\caption{The antisymmetric cosine interpolating polynomials $\psi^{\mathrm{V},-}_N(x,y,\frac{1}{3})$ of the model function \eqref{modfun} shown in Fig. \ref{model} with $N=5,10,15$.}
\label{fig1}
\end{figure}

\begin{figure}[ht]
\includegraphics[width=2.5cm]{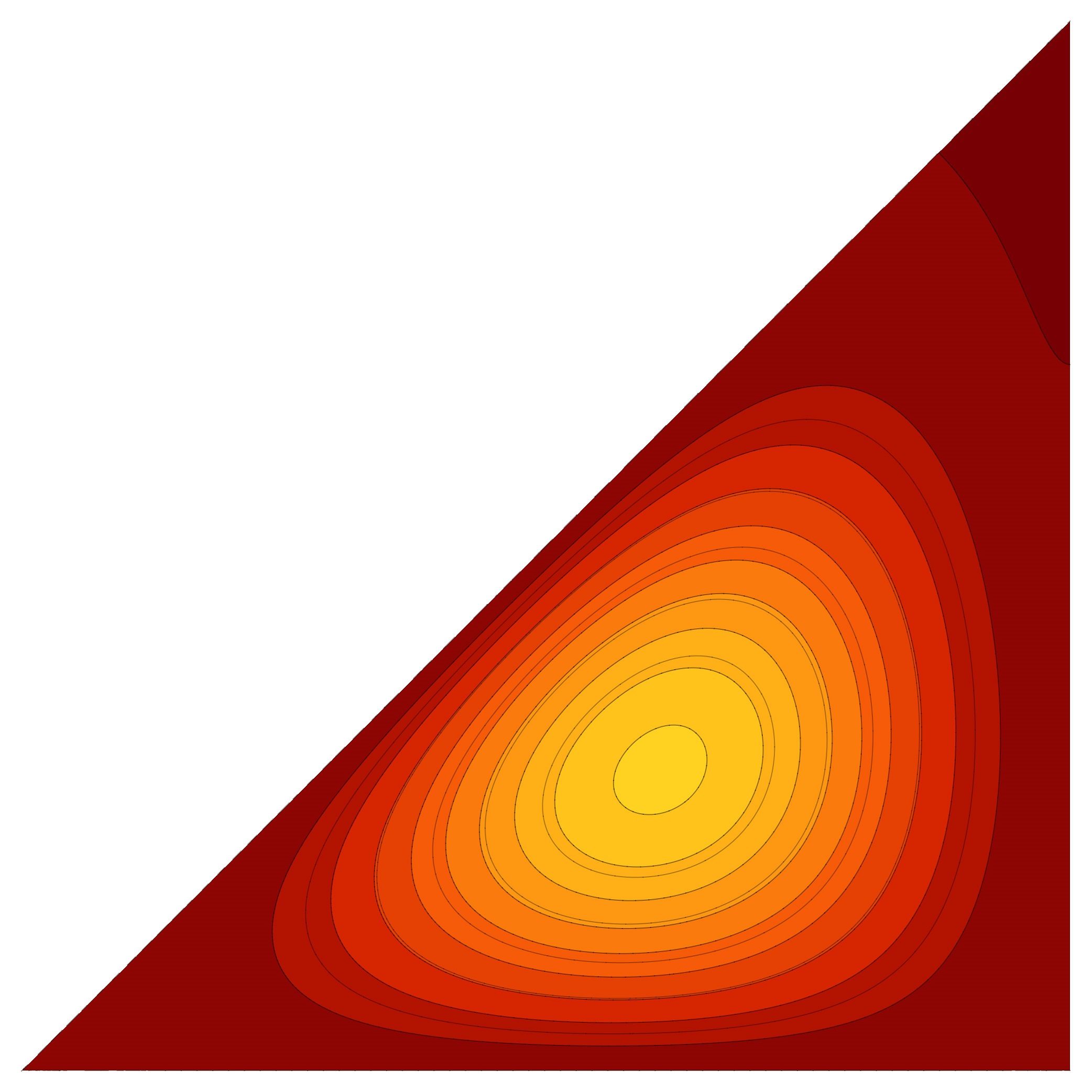}
\hspace{1.3cm}
\includegraphics[width=2.5cm]{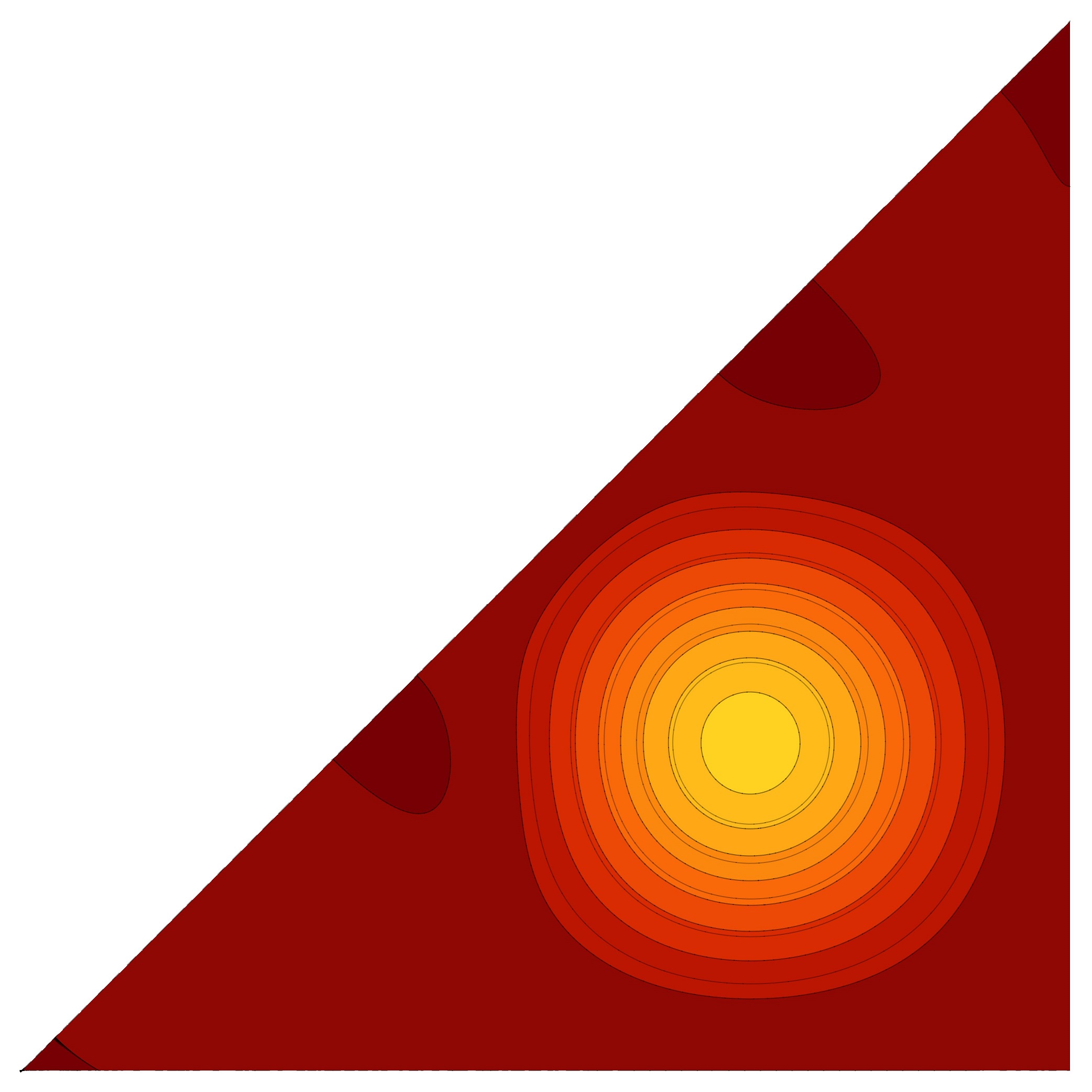}
\hspace{1.3cm}
\includegraphics[width=2.5cm]{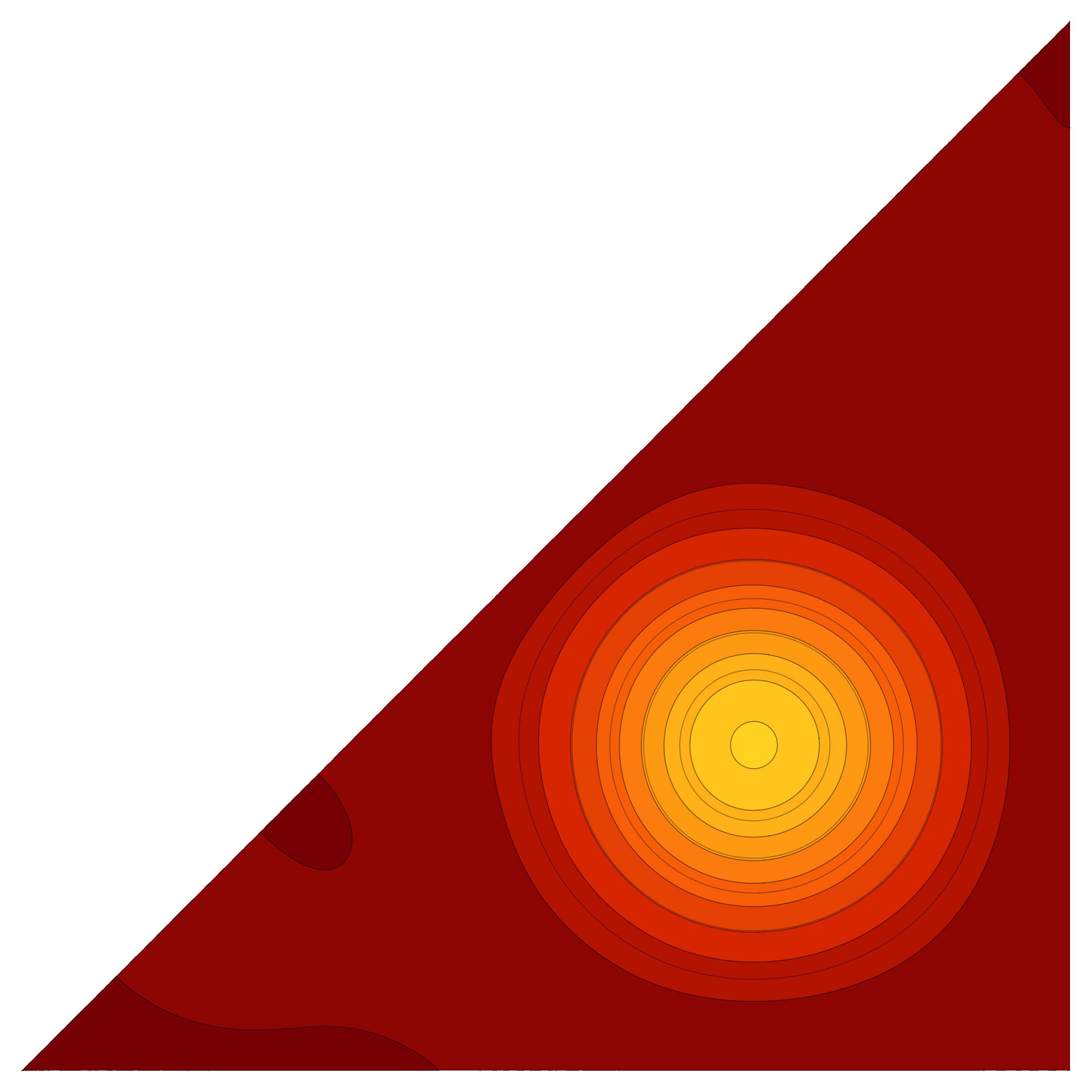}
\\
\includegraphics[width=4cm]{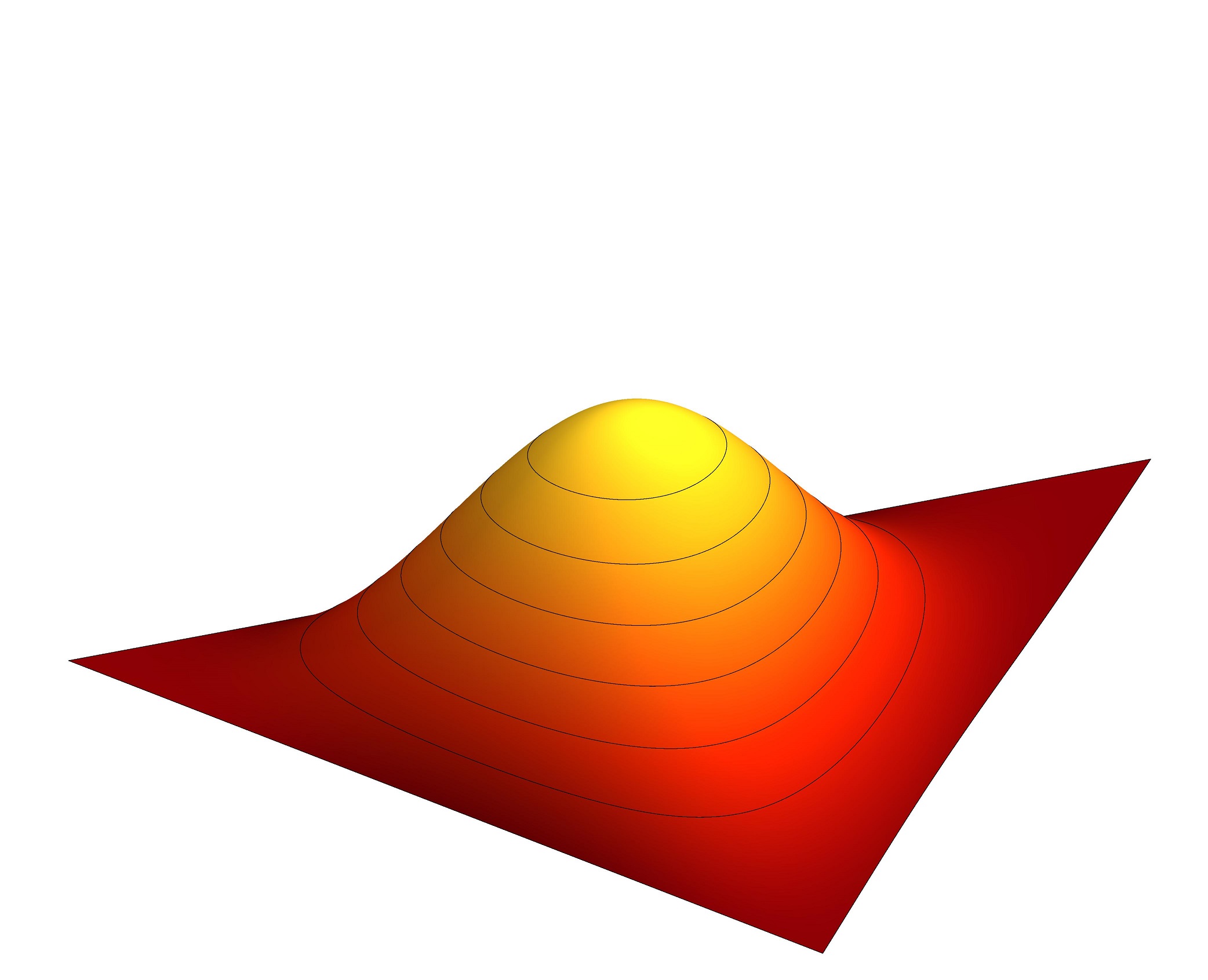}
\includegraphics[width=4cm]{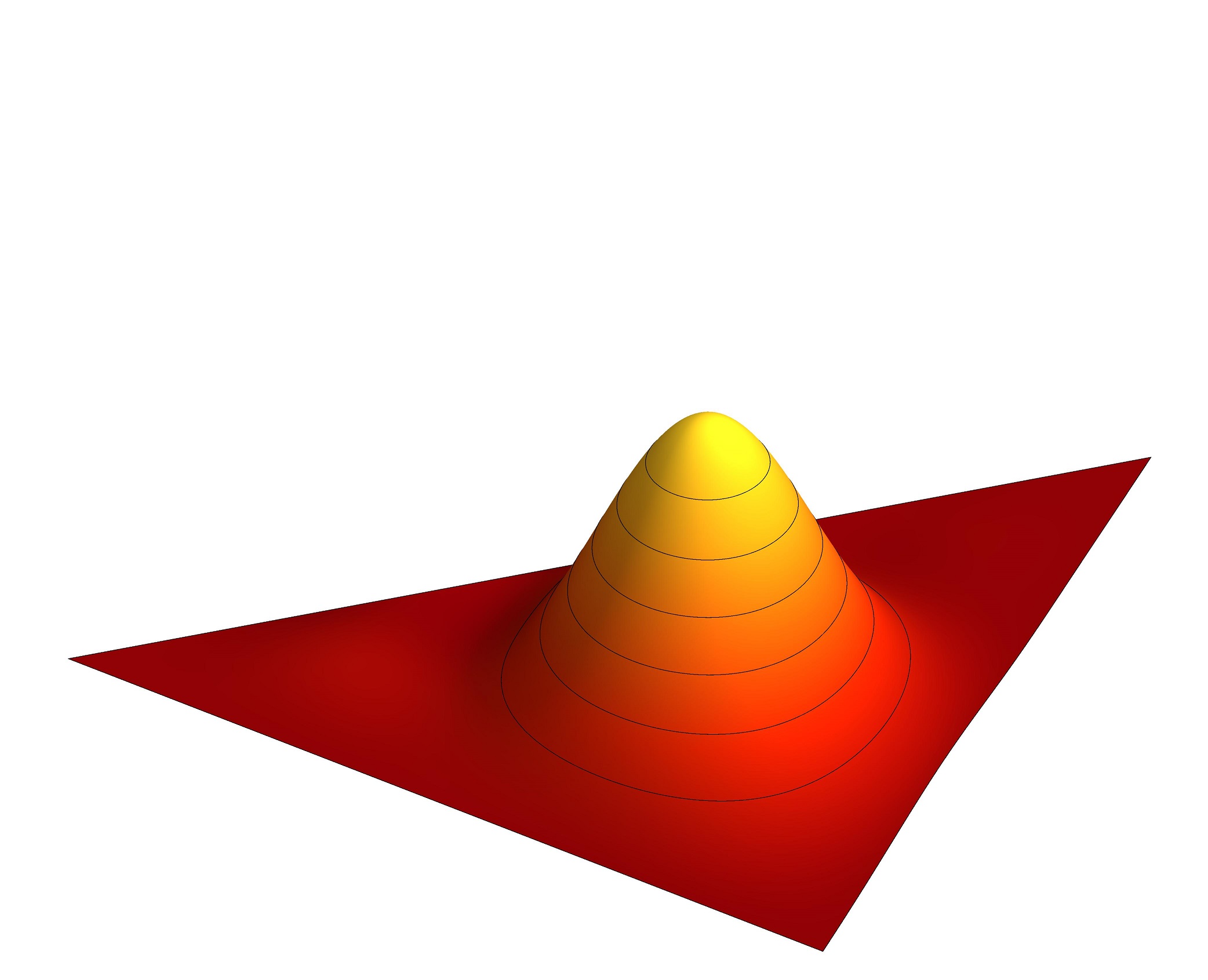}
\includegraphics[width=4cm]{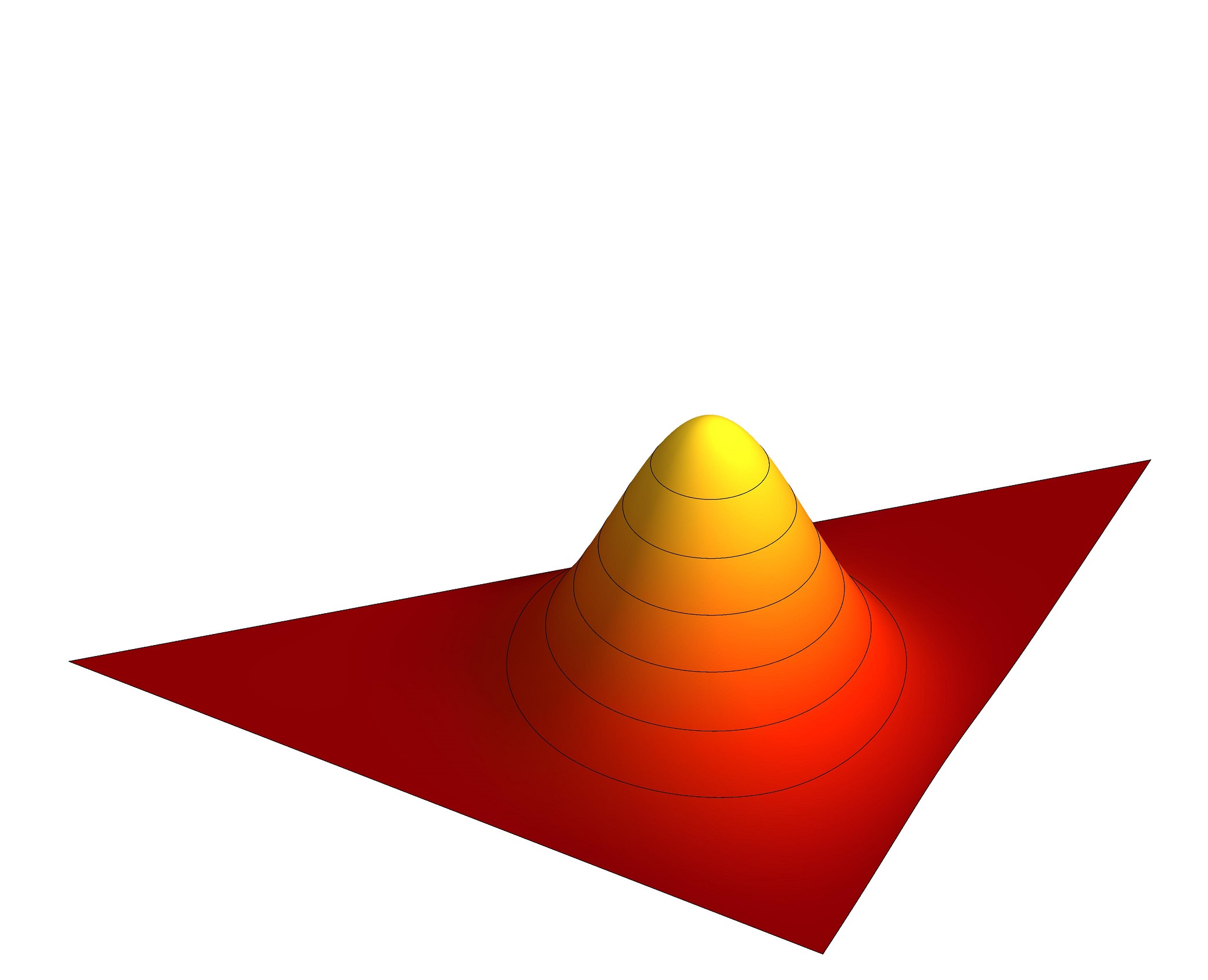}
\caption{The antisymmetric cosine interpolating polynomials $\psi^{\mathrm{VII},-}_N(x,y,\frac{1}{3})$ of the model function \eqref{modfun} shown in Fig. \ref{model} with $N=5,10,15$.}
\label{fig2}
\end{figure}

Integral error estimates of the  approximations of the model function \eqref{modfun} by the antisymmetric interpolating polynomials of types V and VII, 
$$\int_{F(\wt{S}_n^\mathrm{aff})}\abs{f(x,y,z)-\psi^{\mathrm{V},-}_N(x,y,z)}^2\,dx\,dy\,dz\,,\quad \int_{F(\wt{S}_n^\mathrm{aff})}\abs{f(x,y,z)-\psi^{\mathrm{VII},-}_N(x,y,z)}^2\,dx\,dy\,dz\,,$$
are shown in Table \ref{interror} for several values of $N$.
\end{example}

\subsection{Symmetric discrete multivariate cosine transforms}\label{SMDCT}\

The symmetric discrete multivariate cosine transforms (SMDCTs), which can be viewed as symmetric multivariate generalizations of DCTs, are derived using the one-dimensional DCTs from Section \ref{DCT}. The four types of SMDCT, connected to DCTs I--IV, are contained in \cite{KPtrig}. Our goal is to complete the list of SMDCTs by developing the remaining four transforms of types V--VIII.
Note that the coefficients $\ep_s, \wt\ep_s,$ and $d_k, \wt d_k$  are given by \eqref{ep}, \eqref{eps} and  \eqref{dk}, \eqref{dkt}, respectively. In the following subsection we detail the proof of the SMDCT~V transform. The proofs of the remaining transforms VI--VIII are similar. 

\subsubsection{SMDCT V}\

For $N\in \N$ we consider the symmetric cosine functions $\cos^+_k(s)$ labeled by the index set
\begin{equation}\label{par+}
 D^+_N\equiv\set{(k_1,\dots,k_n)\in D_N}{ k_1\geq k_2\geq\dots\geq k_n}
\end{equation} 
and restricted to the finite grid of points contained in $C_N\subset F(\wt{S}^{\mathrm{aff}}_n)$,
\begin{equation}\label{fun+}
F^{\mathrm{V},+}_N\equiv\set{\left(\frac{2r_1}{2N-1},\dots,\frac{2r_n}{2N-1}\right)}{(r_1,\dots,r_n)\in D^+_N}\,. 
\end{equation}
The scalar product of any two functions $f,g:F_N^{\mathrm{V},+}\map\R$ given on the points of the grid $F^{\mathrm{V},+}_N$ is defined by
\begin{equation}\label{scalV+}
\langle f,g\rangle\equiv\sum_{s\in F^{\mathrm{V},+}_N}\ep_sH_s^{-1}f(s)g(s)\,,\quad f,g:F^{\mathrm{V},+}_N\map\R\,,
\end{equation}
where $H_s$ is given by \eqref{Hk}. Using the discrete orthogonality of one-dimensional cosine functions \eqref{DCTV} we show that the symmetric cosine functions labeled by parameters in $k,k'\in D^+_N$ are pairwise discretely orthogonal, i.e.,
\begin{equation}\label{SMDCTV}
\langle\cos_k^+,\cos_{k'}^+\rangle=\sum_{s\in F_N^{\mathrm{V},+}}\ep_sH_s^{-1}\cos_k^+(s)\cos_{k'}^+(s)=\frac{H_k}{d_k}\left(\frac{2N-1}{4}\right)^n\delta_{kk'}\,.
\end{equation}

Note that acting by all permutations in $S_n$ on the grid $F_N^+$ we obtain the whole finite grid $C_N$,
where the points $s\in F_N^+$, which are invariant with respect to some nontrivial permutation in $S_n$, emerge exactly $H_s$ times. Since the symmetric cosine functions are symmetric with respect to $S_n$, we obtain
$$\langle\cos_k^+,\cos_{k'}^+\rangle=\frac{1}{n!}\sum_{r_1,\dots,r_n=0}^{N-1}c_{r_1}\dots c_{r_n}\cos_k^+\left(\frac{2r_1}{2N-1},\dots, \frac{2r_n}{2N-1}\right)\cos_{k'}^+\left(\frac{2r_1}{2N-1},\dots, \frac{2r_n}{2N-1}\right)\,,$$
where $n!$ represents the order of the group $S_n$.
Moreover, the product of two symmetric cosine functions labeled by $k,k'\in D_N^+$ is rewritten due to the invariance under permutations as 
$$\cos^+_{k}(x)\cos^+_{k'}(x)=\sum_{\sigma,\sigma'\in S_n}\prod_{i=1}^n\cos{\left(\pi k_{\sigma\sigma'(i)}x_{\sigma'(i)}\right)}
\cos{\left(\pi k'_i x_{\sigma'(i)}\right)}\,.$$
Together with the $S_n$-invariance of the set $D_N$ this implies that
$$
\langle\cos_k^+,\cos_{k'}^+\rangle=\sum_{\sigma\in S_n}\prod_{i=1}^n\sum_{r_i=0}^{N-1}c_{r_i}\cos{\left(\frac{2\pi k_{\sigma(i)}r_i}{2N-1}\right)}\cos{\left(\frac{2\pi k'_{i}r_i}{2N-1}\right)}\,.
$$
Finally, we apply the one-dimensional orthogonality relation \eqref{DCTV} to obtain \eqref{SMDCTV},
$$\langle\cos_k^+,\cos_{k'}^+\rangle=\sum_{\sigma\in S_n}\prod_{i=1}^n\frac{2N-1}{4c_{k_{\sigma(i)}}}\delta_{k_{\sigma(i)}k'_i}=\frac{H_k}{d_k}\left(\frac{2N-1}{4}\right)^n\delta_{kk'}\,.$$

Due to the relation \eqref{SMDCTV}, we expand any function $f: F_N^{\mathrm{V},+}\map\R$ in terms of symmetric cosine functions as
$$f(s)=\sum_{k\in D^+_N}A_k\cos^+_k(s),\quad\quad A_k=\frac{d_k}{H_k}\left(\frac{4}{2N-1}\right)^n\sum_{s\in F^{\mathrm{V},+}_N}\ep_s H_s^{-1}f(s)\cos_k^+(s)\,.$$
Validity of the expansion follows from the fact that the number of points in $D_N^+$ is equal to the number of points in $F^{\mathrm{V},+}_N$ and thus the functions $\cos^+_k$ with $k\in D_N^+$ form an orthogonal basis of the space of all functions $f:F^{\mathrm{V},+}_N\map\R$ with the scalar product \eqref{scalV+}. The remaining transforms SMDCT VI--VIII are deduced similarly.

\subsubsection{SMDCT VI}\

For $N\in \N$ we consider the antisymmetric cosine functions $\cos^+_k(s)$ labeled by the index set $k\in D^+_N$ and
restricted to the finite grid of points contained in $\wt C_N\subset F(\wt{S}^{\mathrm{aff}}_n)$, 
\begin{equation} 
F^{\mathrm{VI},+}_N\equiv\set{\left(\frac{2\left(r_1+\frac12\right)}{2N-1},\dots,\frac{2\left(r_n+\frac12\right)}{2N-1}\right)}{(r_1,\dots, r_n)\in D^+_N}\,.
\end{equation}
The antisymmetric cosine functions labeled by parameters in $k,k'\in D_N^+,$ are pairwise discretely orthogonal on the grid $F^{\mathrm{VI},+}_N$, i.e., 
\begin{equation}
\sum_{s\in F^{\mathrm{VI},+}_N}\wt\ep_sH_s^{-1}\cos_k^+(s)\cos_{k'}^+(s)=\frac{H_k}{d_k}\left(\frac{2N-1}{4}\right)^n\delta_{kk'}\,.
\end{equation}
Therefore, we expand any function $f: F_N^{\mathrm{VI},+}\map\R$ in terms of symmetric cosine functions as follows:
$$f(s)=\sum_{k\in D^+_N}A_k\cos^+_k(s)\quad\text{with}\quad A_k=\frac{d_k}{H_k}\left(\frac{4}{2N-1}\right)^n\sum_{s\in F^{\mathrm{VI},+}_N}\wt\ep_sH_s^{-1}f(s)\cos_k^-(s)\,.$$

\subsubsection{SMDCT VII}\

For  $N\in \N$ are the symmetric cosine functions $\cos^+_{k+\rho}(s)$ , $k\in D^+_N$, 
and $\rho=\left(\frac12,\dots,\frac12\right)$ restricted to the finite grid of points  $F^{\mathrm{VII},+}_N\equiv F_N^{\mathrm{V},+}\subset C_N$ pairwise discretely orthogonal with respect to the scalar product \eqref{scalV+}, i.e.,
\begin{equation}
\sum_{s\in F^{\mathrm{VII},+}_N}\ep_sH_s^{-1}\cos_{k+\rho}^+(s)\cos_{k'+\rho}^+(s)=\frac{H_k}{\wt d_k}\left(\frac{2N-1}{4}\right)^n\delta_{kk'}\,.
\end{equation}
Therefore, we expand any function $f: F_N^{\mathrm{VII},+}\map\R$ in terms of symmetric cosine functions as 
$$f(s)=\sum_{k\in D^+_N}^{N-1}A_k\cos^+_{k+\rho}(s),\quad A_k=\frac{\wt d_k}{H_k}\left(\frac{4}{2N-1}\right)^n\sum_{s\in F^{\mathrm{VII},+}_N}\ep_sH_s^{-1}f(s)\cos_{k+\rho}^+(s)\,.$$

\subsubsection{SMDCT VIII}\

For  $N\in \N$ are the symmetric cosine functions $\cos^+_{k+\rho}(s)$ , $k\in D^+_N$, 
and $\rho=\left(\frac12,\dots,\frac12\right)$ restricted to the finite grid of points 
\begin{equation}\label{gridVIII+}
F^{\mathrm{VIII},+}_N\equiv\set{\left(\frac{2\left(r_1+\frac12\right)}{2N+1},\dots,\frac{2\left(r_n+\frac12\right)}{2N+1}\right)}{(r_1,\dots,r_n)\in D^+_N}
\end{equation} 
pairwise discretely orthogonal, i.e.,  for any $k,k'\in D^+_N$ it holds that 
\begin{equation}\label{orthogVIII+}
\sum_{s\in F^{\mathrm{VIII},+}_N}H_s^{-1}\cos_{k+\rho}^+(s)\cos_{k'+\rho}^+(s)=H_k\left(\frac{2N+1}{4}\right)^n\delta_{kk'}\,.
\end{equation}
Therefore, we expand any function $f: F_N^{\mathrm{VIII},+}\map\R$ in terms of symmetric cosine functions as 
$$f(s)=\sum_{k\in D^+_N}A_k\cos^+_{k+\rho}(s),\quad A_k=\frac{1}{H_k}\left(\frac{4}{2N+1}\right)^n\sum_{s\in F^{\mathrm{VIII},+}_N}H_s^{-1}f(s)\cos_{k+\rho}^+(s)\,.$$

\subsection{Interpolations by symmetric discrete multivariate cosine functions}\

Suppose we have a real-valued function $f$ given on $F(\wt{S}^{\mathrm{aff}}_n)$. In Section \ref{SMDCT} we defined three finite grids in $F(\wt{S}^{\mathrm{aff}}_n)$, namely, $F^{\mathrm{VI},+}_N, F^{\mathrm{V},+}_N\equiv F^{\mathrm{VII},+}_N$, and $F^{\mathrm{VIII},+}_N$. We are interested in finding the interpolating polynomial of $f$ in the form of a finite sum of symmetric  multivariate cosine functions labeled by $k$ or by $k+\rho$, $k\in D^+_N$ in such a way that it coincides with $f$ on one of the grids in $F(\wt{S}^{\mathrm{aff}}_n)$. We distinguish between four types of interpolating polynomials defined for $x\in F(\wt{S}^{\mathrm{aff}}_n)$ and satisfying different conditions:
\begin{equation}\label{inter}
\begin{alignedat}{3}
&\psi^{\mathrm{V},+}_N(x)&=&\sum_{k\in D^+_N}B_k\cos_k^+(x)\,,&\qquad\qquad& \psi^{\mathrm{V},+}_N(s)=f(s)\,,\quad s\in F^{\mathrm{V},+}_N\,,\\
&\psi^{\mathrm{VI},+}_N(x)&=&\sum_{k\in D^+_N}B_k\cos_{k+\rho}^+(x)\,,&\qquad\qquad& \psi^{\mathrm{VI},+}_N(s)=f(s)\,,\quad s\in F^{\mathrm{VI},+}_N\,,\\
&\psi^{\mathrm{VII},+}_N(x)&=&\sum_{k\in D^+_N}B_k\cos_k^+(x)\,,&\qquad\qquad& \psi^{\mathrm{VII},+}_N(s)=f(s)\,,\quad s\in F^{\mathrm{VII},+}_N\,,\\
&\psi^{\mathrm{VIII},+}_N(x)&=&\sum_{k\in D^+_N}B_k\cos_{k+\rho}^+(x)\,,&\qquad\qquad& \psi^{\mathrm{VIII},+}_N(s)=f(s)\,,\quad s\in F^{\mathrm{VIII},+}_N\,.\\
\end{alignedat}
\end{equation} 
According to Section \ref{SMDCT}, the coefficients $B_k$ are chosen to be equal to corresponding $A_k$. In fact, it is not possible to have other values of the coefficients $B_k$ since it would contradict the fact that the antisymmetric cosine functions labeled by $k$ or by $k+\rho$ with $k\in D^+_N$ are basis vectors of the space of functions given on the corresponding grid.

\begin{example}
For $n=3$ is the smooth function \eqref{modfun} with values of parameters given by \eqref{vp}  chosen as a model function. We interpolate this function by polynomials of symmetric cosine functions $\psi^{\mathrm{V},+}_N(x,y,z)$ and $\psi^{\mathrm{VII},+}_N(x,y,z)$ with $N=5,10,15$.
Plots of this symmetric cosine interpolating polynomials are depicted in Figures \ref{fig3} and \ref{fig4}.

\begin{figure}[ht]
\includegraphics[width=2.5cm]{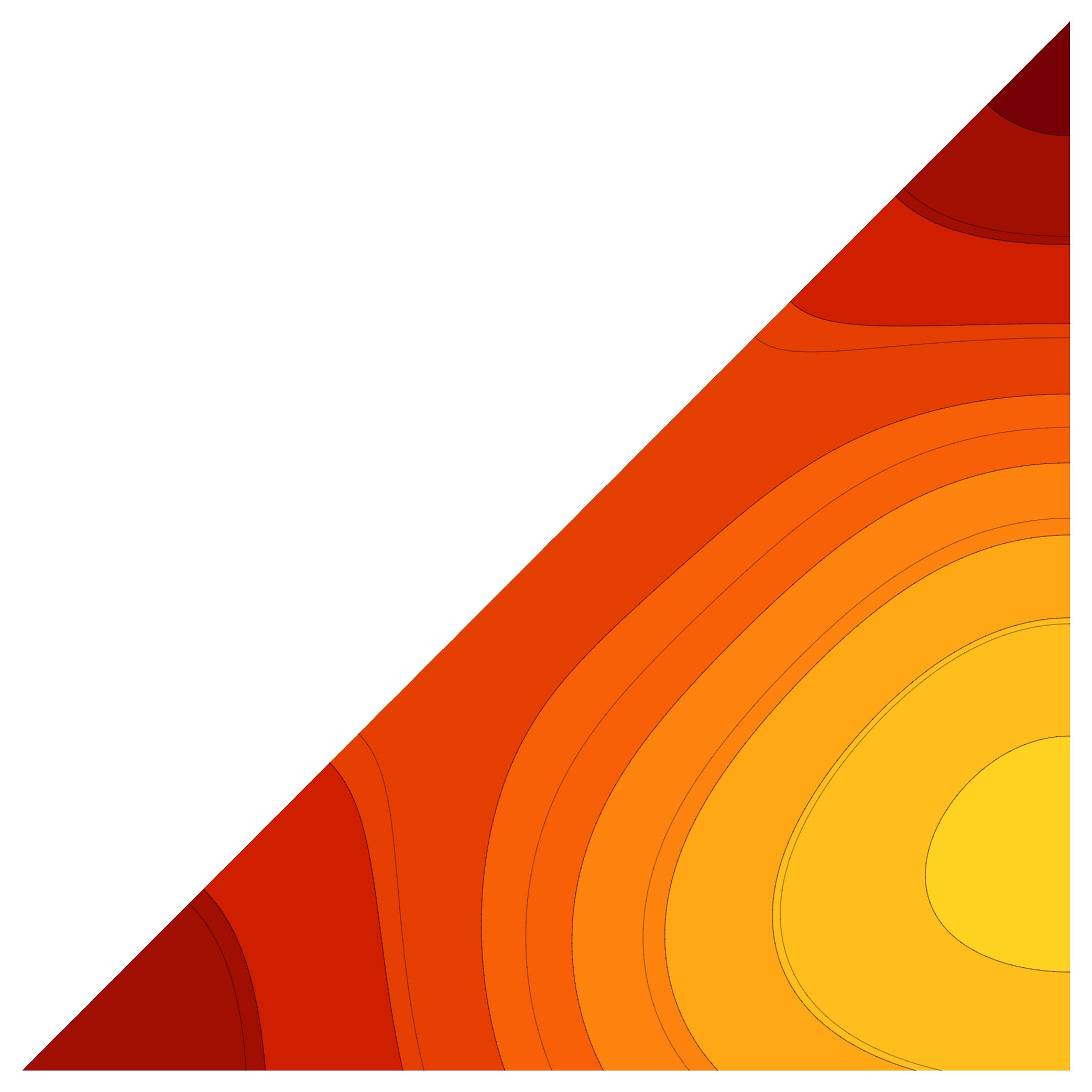}
\hspace{1.3cm}
\includegraphics[width=2.5cm]{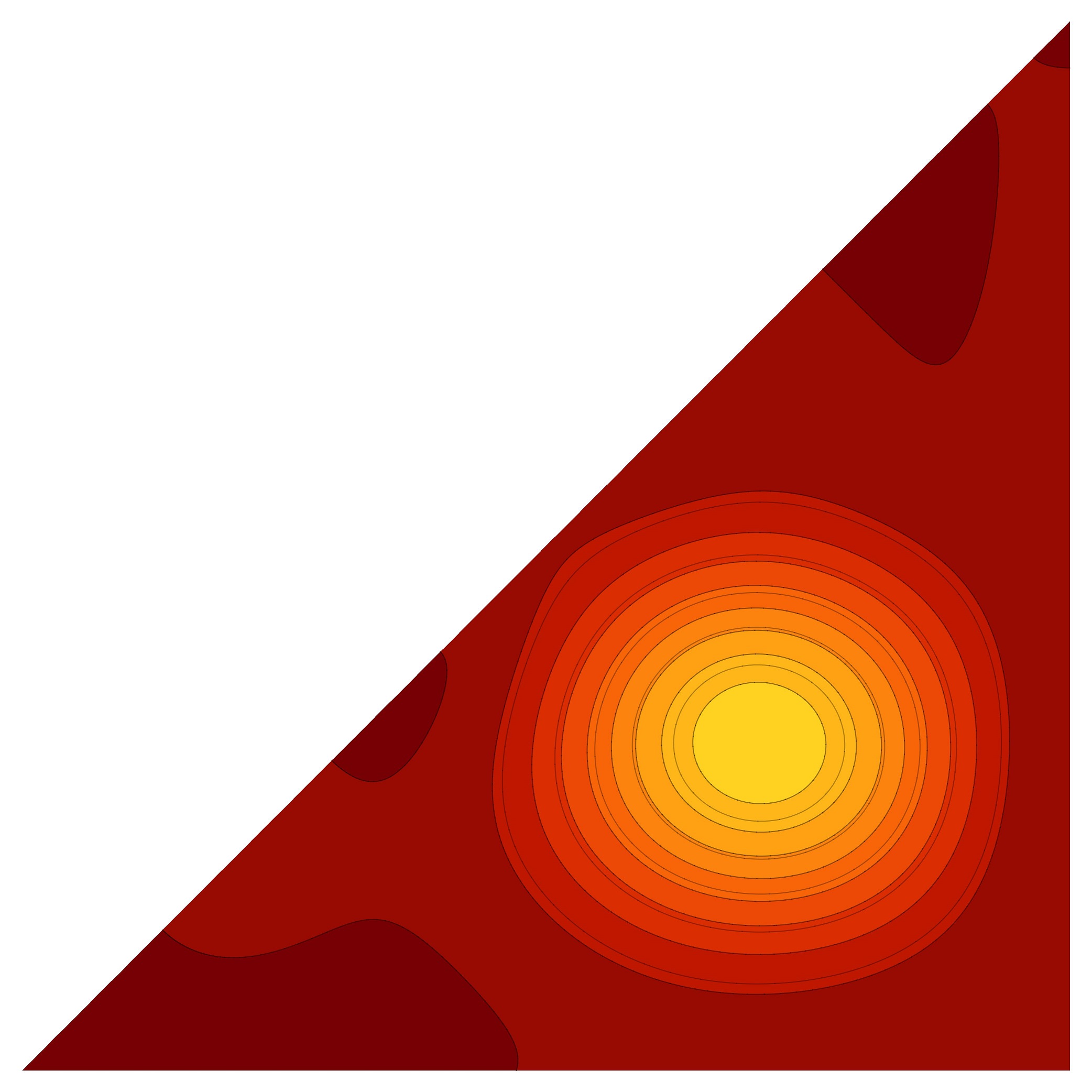}
\hspace{1.3cm}
\includegraphics[width=2.5cm]{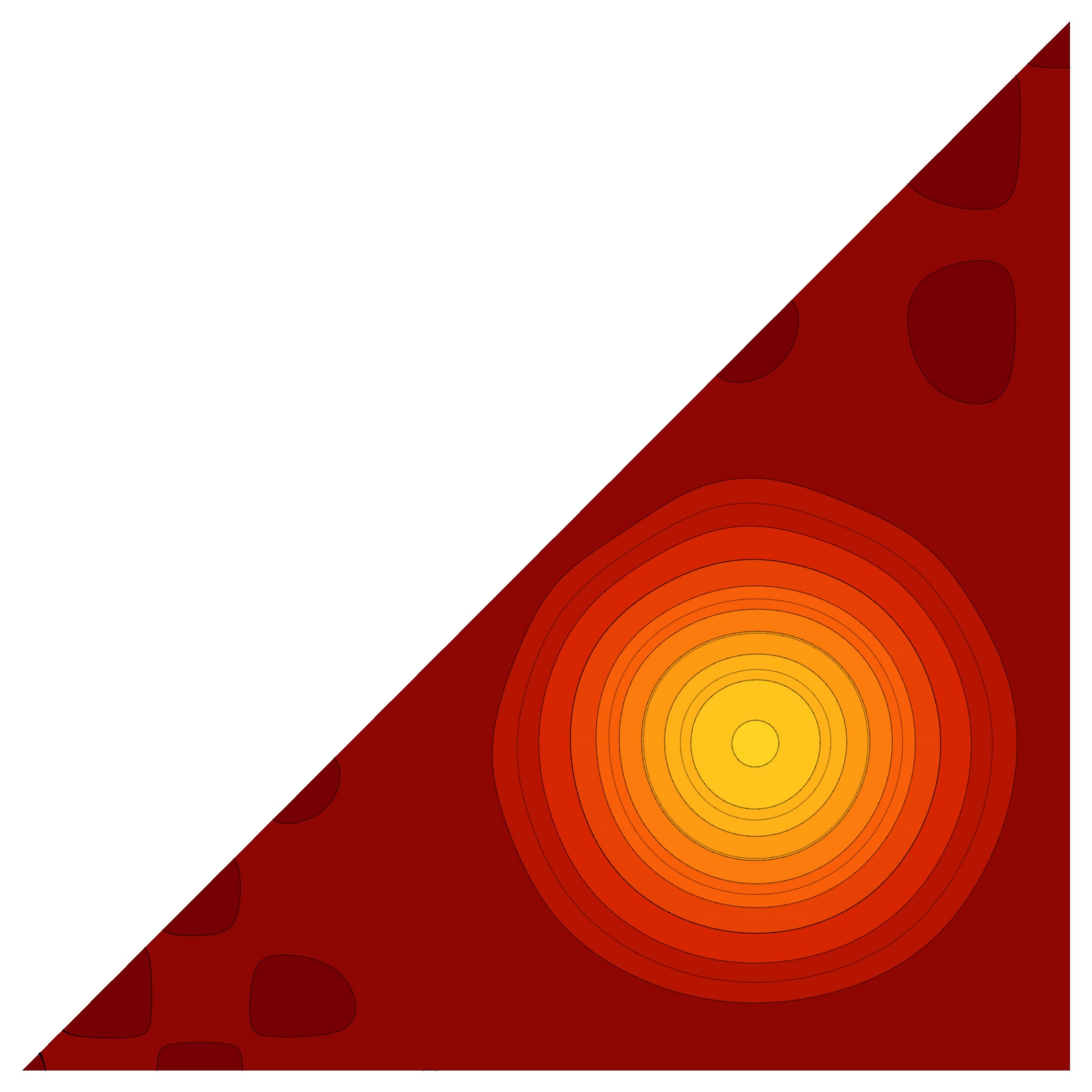}
\\
\includegraphics[width=4cm]{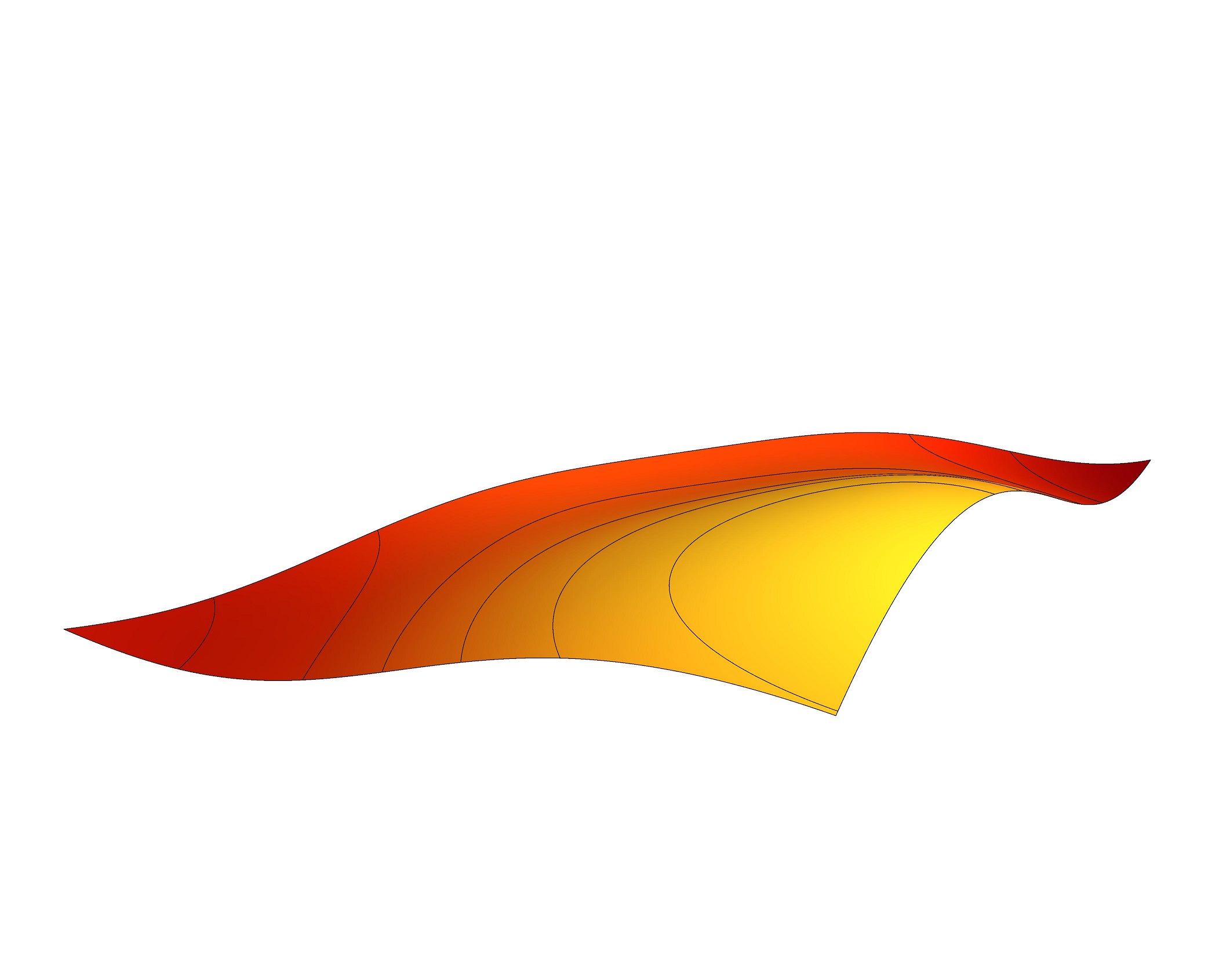}
\includegraphics[width=4cm]{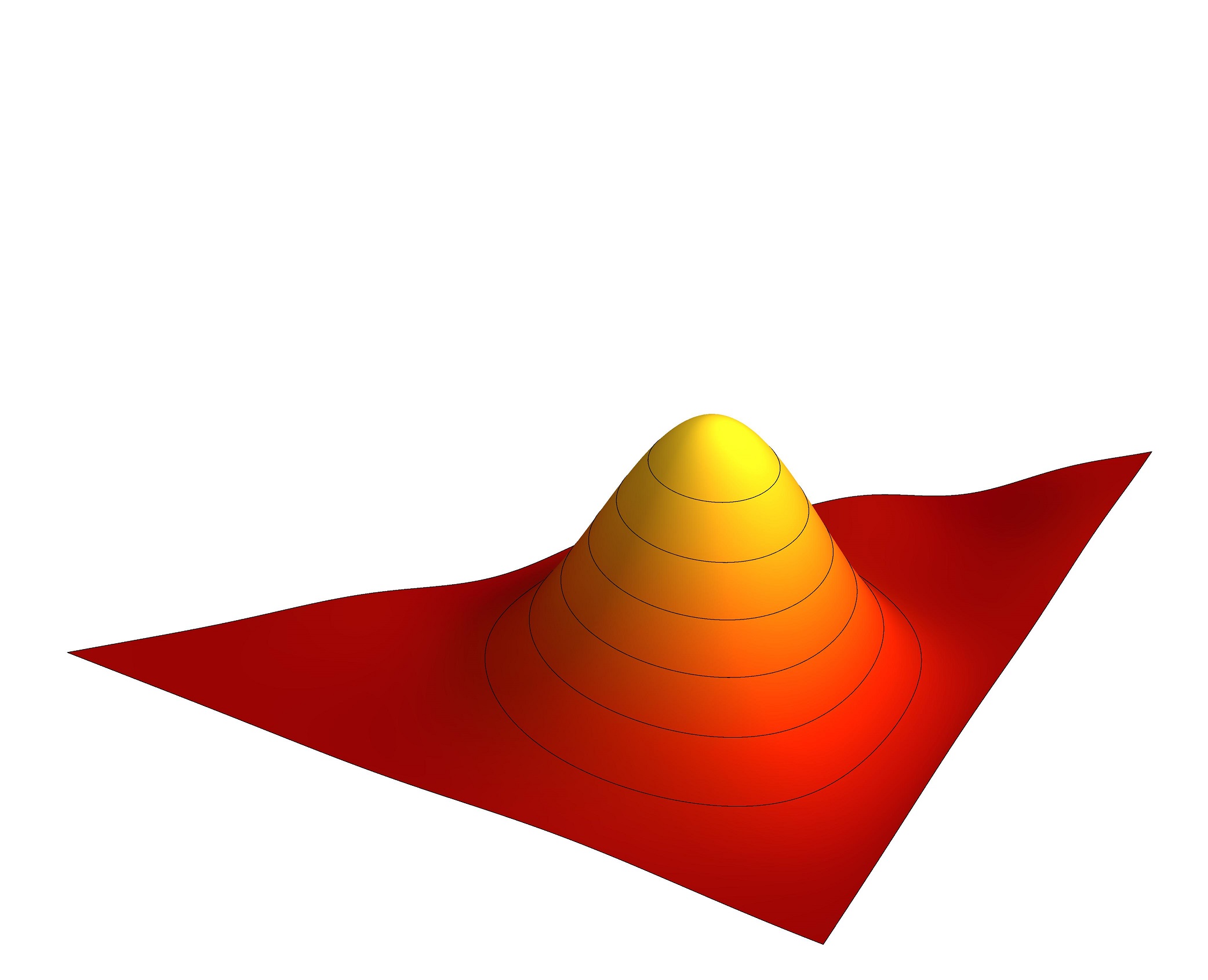}
\includegraphics[width=4cm]{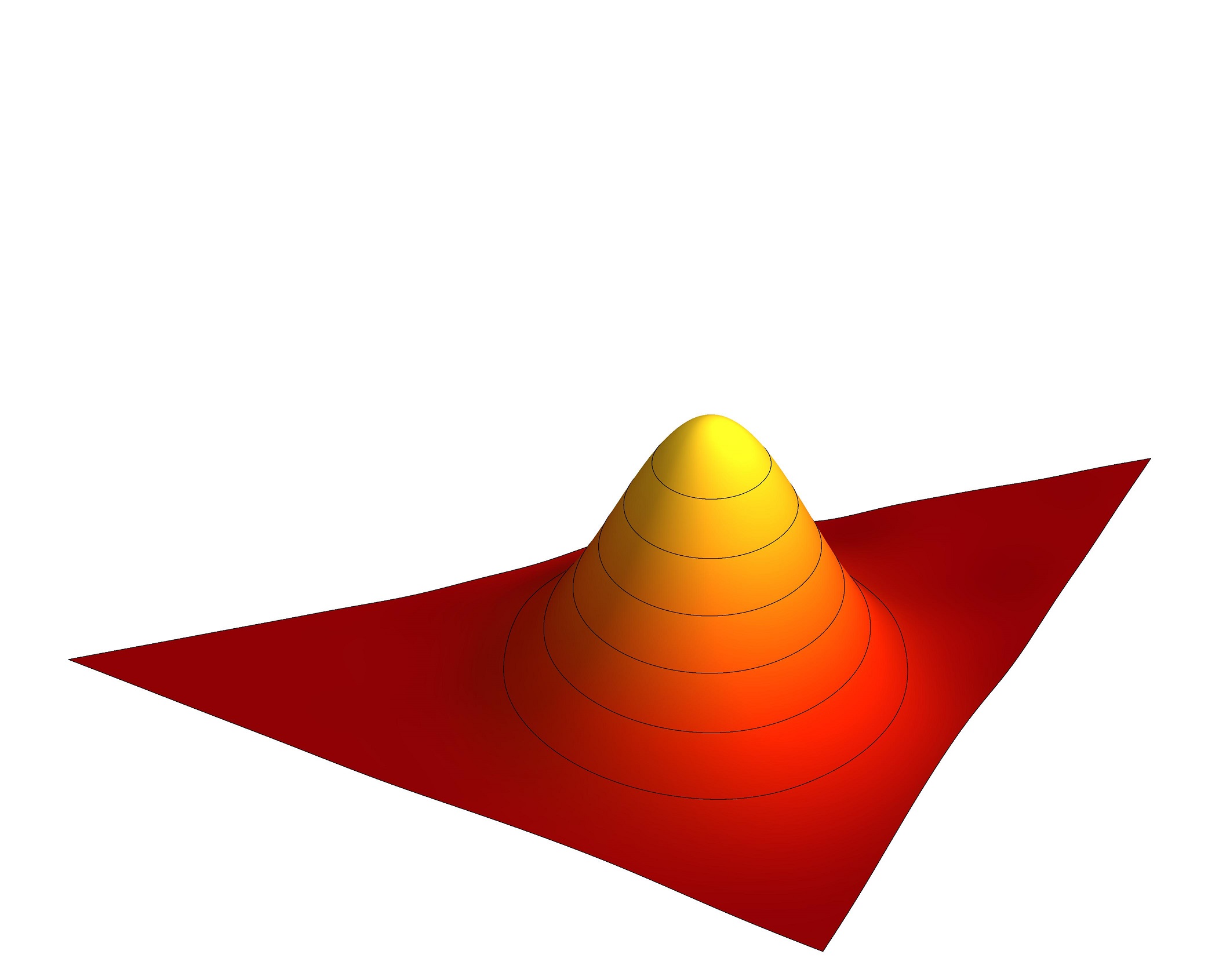}
\caption{The symmetric cosine interpolating polynomials $\psi^{\mathrm{V},+}_N(x,y,\frac{1}{3})$ of the model function \eqref{modfun} shown in Fig. \ref{model} with $N=5,10,15$.}
\label{fig3}
\end{figure}

\begin{figure}[ht]
\includegraphics[width=2.5cm]{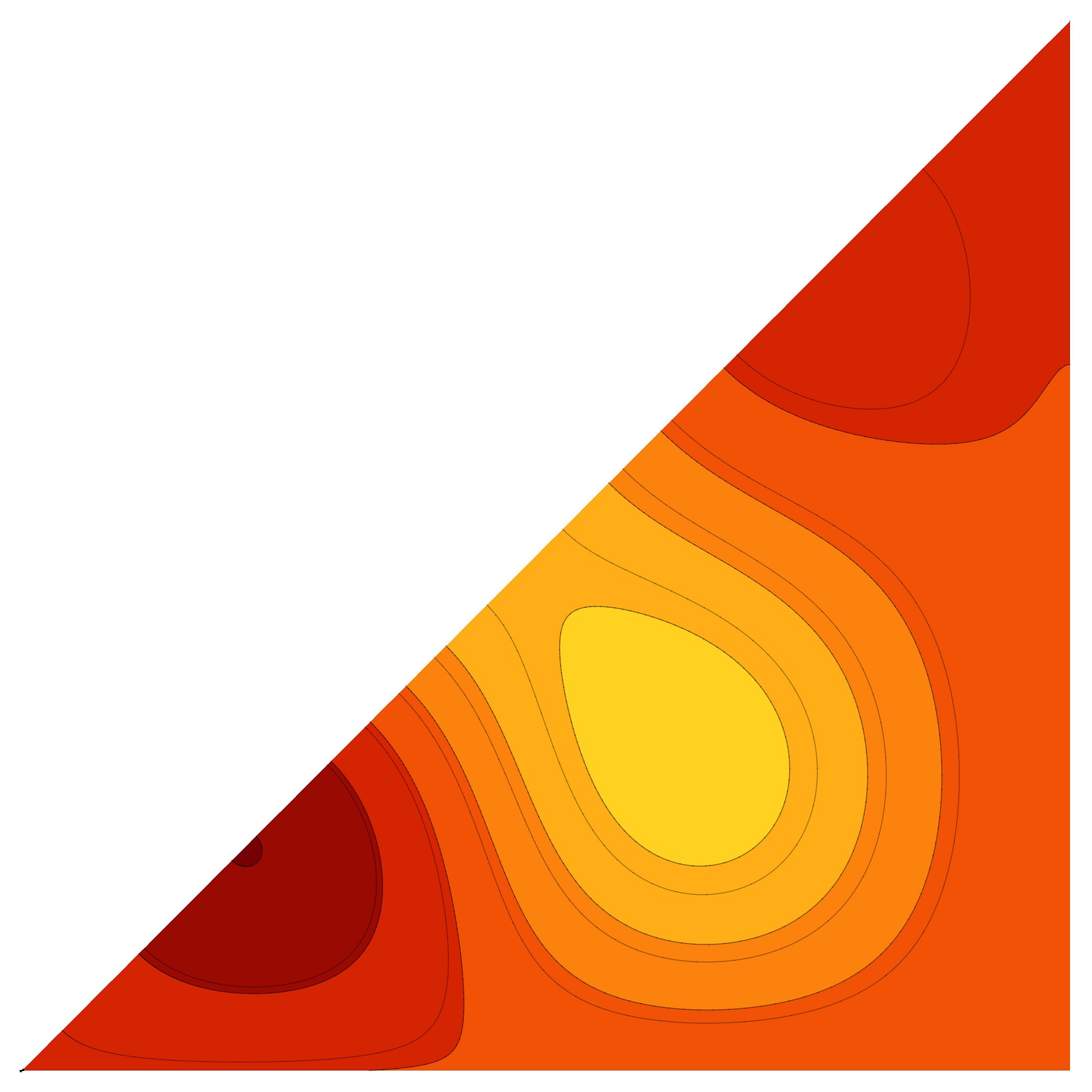}
\hspace{1.3cm}
\includegraphics[width=2.5cm]{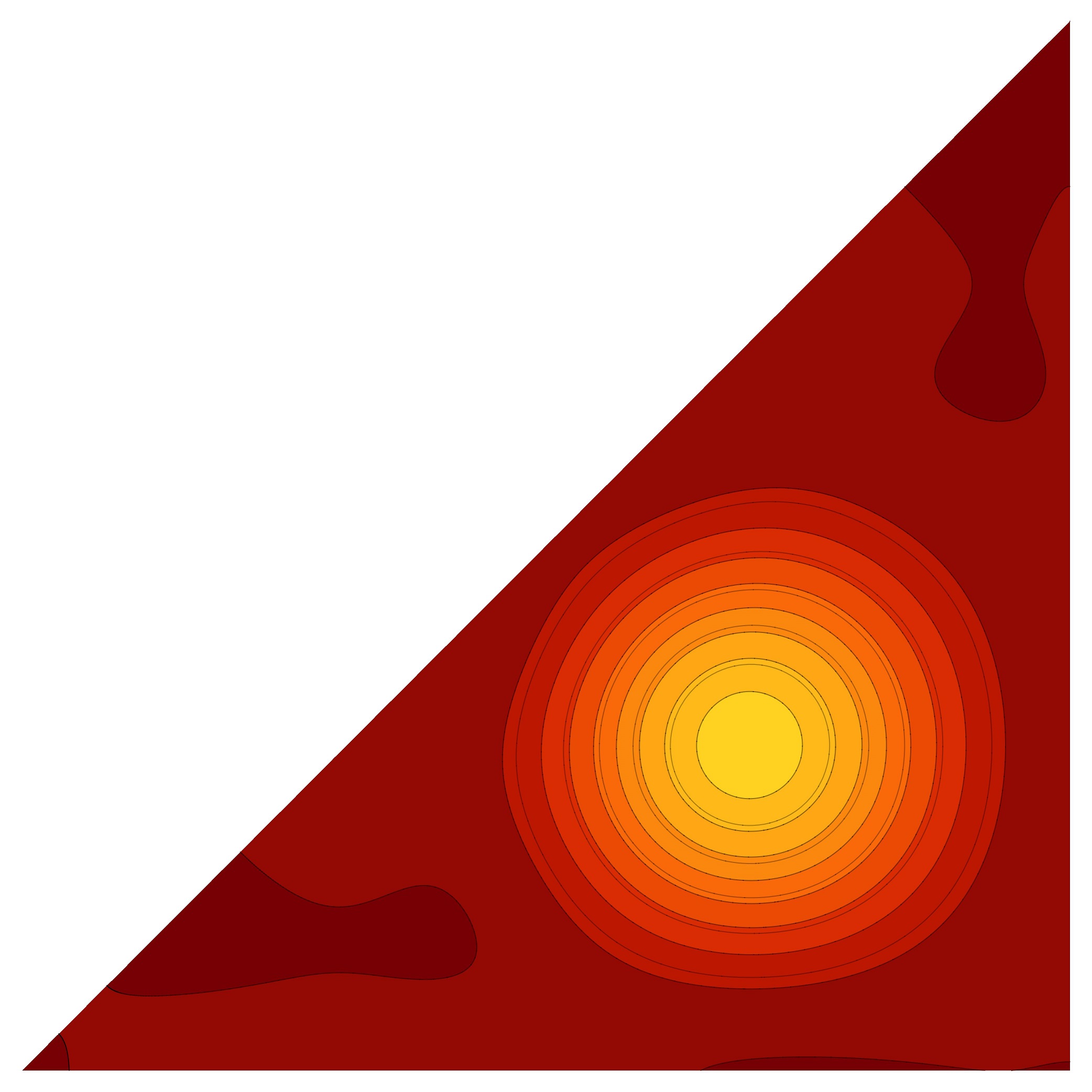}
\hspace{1.3cm}
\includegraphics[width=2.5cm]{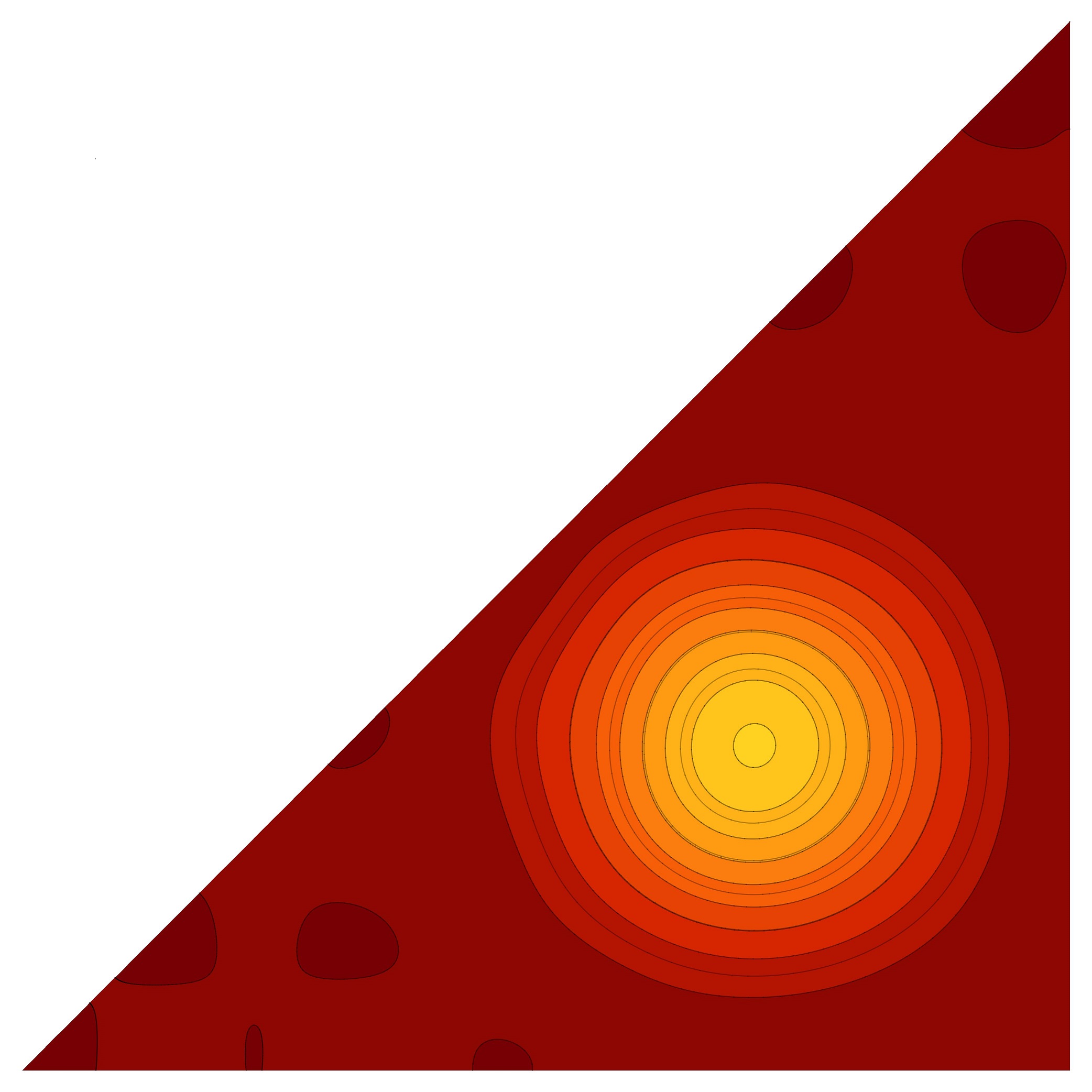}
\\
\includegraphics[width=4cm]{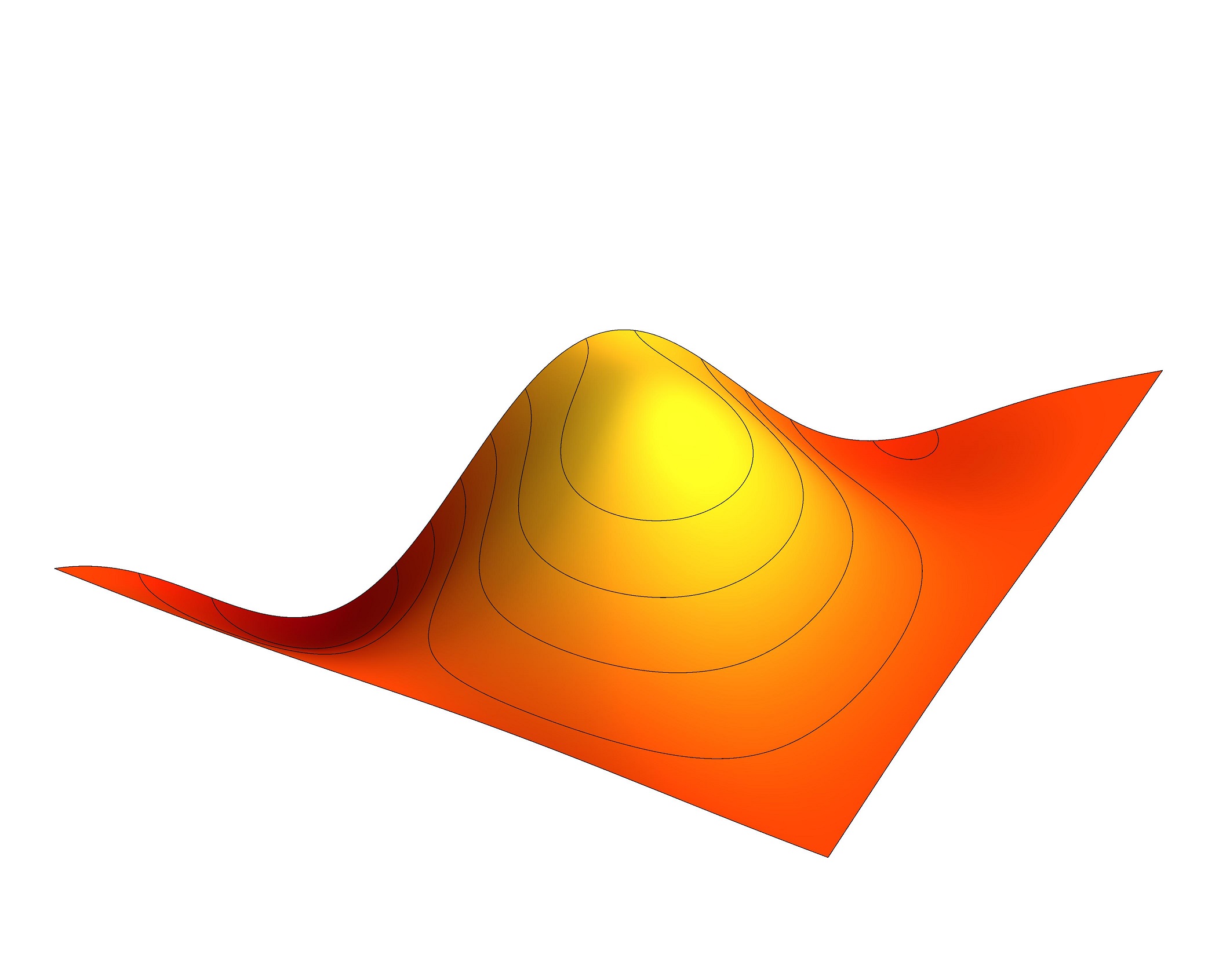}
\includegraphics[width=4cm]{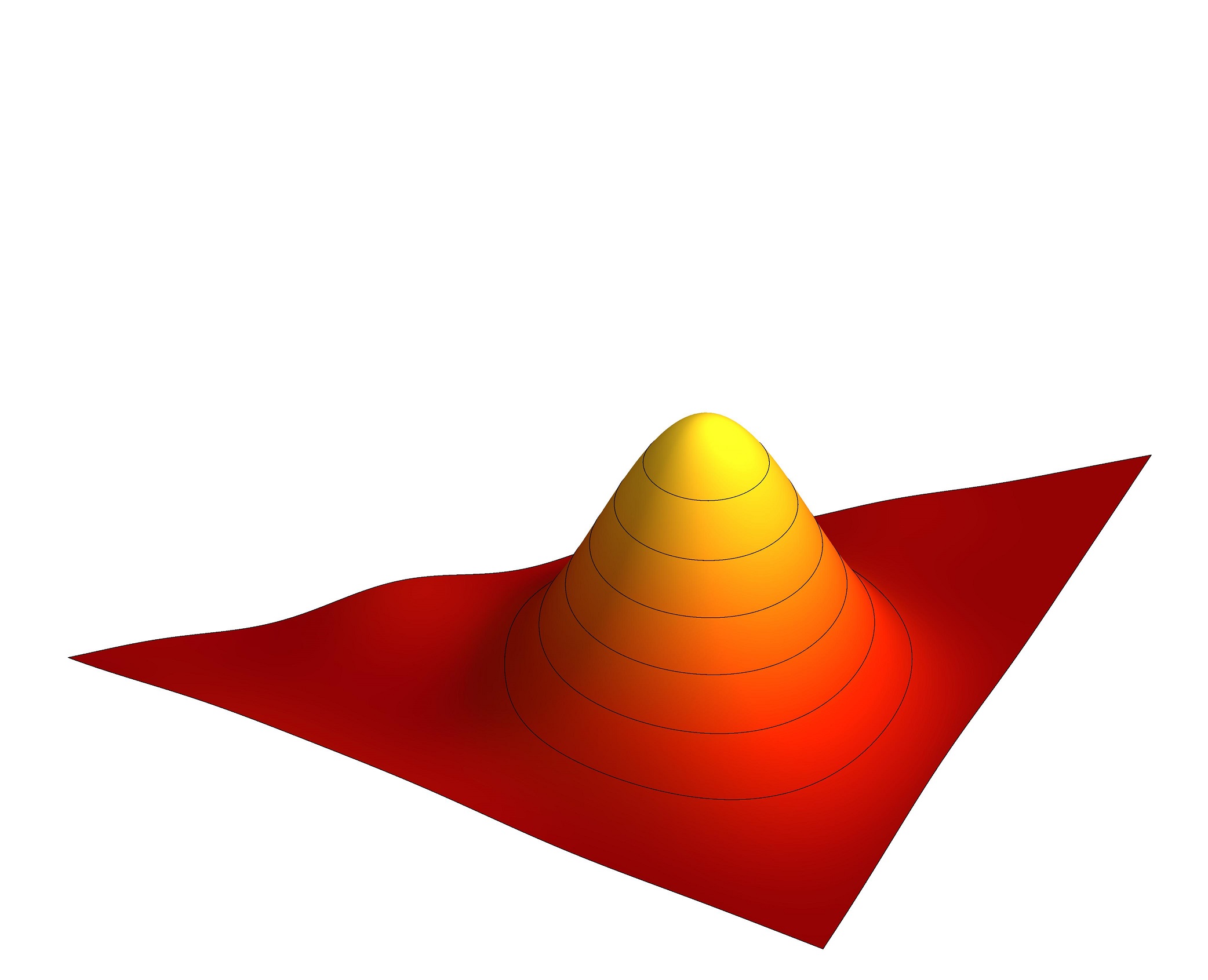}
\includegraphics[width=4cm]{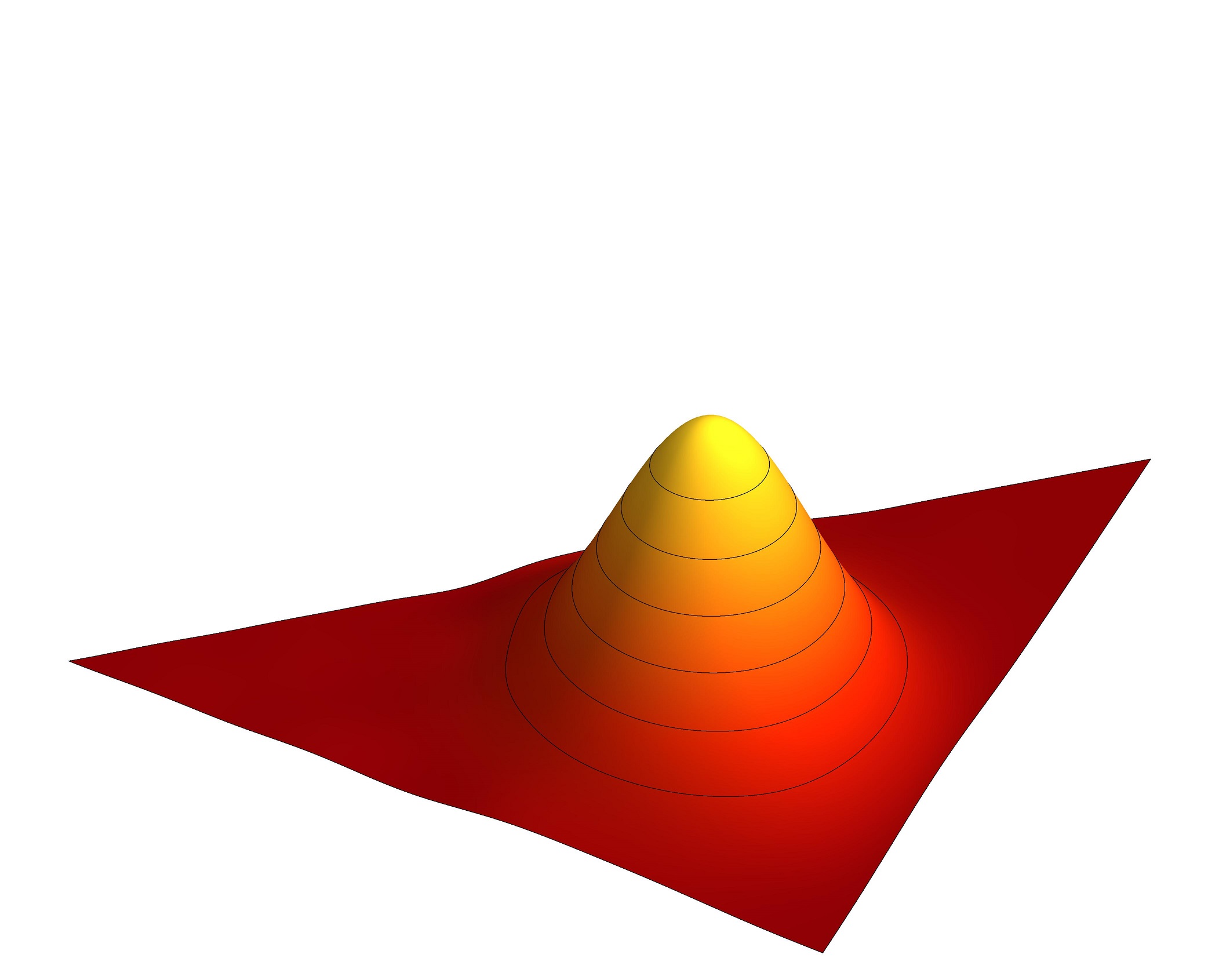}
\caption{The symmetric cosine interpolating polynomials $\psi^{\mathrm{VII},+}_N(x,y,\frac{1}{3})$ of the model function \eqref{modfun} shown in Fig. \ref{model} with $N=5,10,15$.}
\label{fig4}
\end{figure}

Integral error estimates of the  approximations of the model function \eqref{modfun} by the symmetric interpolating polynomials of types V and VII, 
$$\int_{F(\wt{S}_n^\mathrm{aff})}\abs{f(x,y,z)-\psi^{\mathrm{V},+}_N(x,y,z)}^2\,dx\,dy\,dz\,,\quad \int_{F(\wt{S}_n^\mathrm{aff})}\abs{f(x,y,z)-\psi^{\mathrm{VII},+}_N(x,y,z)}^2\,dx\,dy\,dz\,,$$
are shown in Table \ref{interror} for $N=5,10,15,20,25,30$.

\begin{table}
\begin{tabular}{|c||c|c|c|c|}
\hline 
$N$ & $\int_{F(\wt{S}_n^\mathrm{aff})}\abs{f-\psi^{\mathrm{V},-}_N}^2$ & $\int_{F(\wt{S}_n^\mathrm{aff})}\abs{f-\psi^{\mathrm{VII},-}_N}^2$ & $\int_{F(\wt{S}_n^\mathrm{aff})}\abs{f-\psi^{\mathrm{V},+}_N}^2$ & $\int_{F(\wt{S}_n^\mathrm{aff})}\abs{f-\psi^{\mathrm{VII},+}_N}^2$ \\ 
\hline \hline
5 & 0.648691 & 1.396870 & 0.725031 & 1.502161 \\ 
\hline 
10 & 0.007940 & 0.007599 & 0.007191 & 0.006471 \\ 
\hline 
15 & 0.001350 & 0.001407 & 0.000440 & 0.000492 \\ 
\hline 
20 & 0.001034 & 0.001058 & 0.000171 & 0.000195 \\ 
\hline 
25 & 0.000835 & 0.000847 & 0.000084 & 0.000097 \\ 
\hline 
30 & 0.000698 & 0.000705 & 0.000047 & 0.000054 \\ 
\hline 
\end{tabular}
\medskip
\caption{Integral error estimates of the polynomial approximations of the model function $f$ \eqref{modfun} for $N=5,10,15,20,25,30$.}
\label{interror}
\end{table} 
\end{example}

\section{Chebyshev-like multivariate orthogonal polynomials}\label{seccheb}\

Recall that the vectors $\rho,\rho_1$, and $\rho_2$ are defined by \eqref{rho}, \eqref{rho1}, and \eqref{rho2}, respectively. Introducing the $n$ functions $X_1,X_2,\dots,X_n$ given by 
\begin{equation}\label{var}
X_1=\cos^+_{(1,0,\dots,0)}\,,\quad X_2=\cos^+_{(1,1,0,\dots,0)}\,,\quad X_3=\cos^+_{(1,1,1,0,\dots,0)}\,,\quad \dots \,,\quad X_n=\cos^+_{(1,1,\dots,1)}\,,
\end{equation} 
we demonstrate that the following defining relations, valid for all points from the interior of  $F(\wt{S}^{\mathrm{aff}}_n)$, $x\in F(\wt{S}^{\mathrm{aff}}_n)^\circ $, $k\in P^+$,
\begin{equation}\label{newpol}
\begin{alignedat}{4}
&\mathcal{P}^{I,+}_k (X_1(x),\dots,X_n(x)) =\cos^+_k(x)\,,&\quad\quad&\mathcal{P}^{I,-}_k (X_1(x),\dots,X_n(x))=\frac{\cos^-_{k+\rho_1}(x)}{\cos^-_{\rho_1}(x)}\,,\\
&\mathcal{P}^{III,+}_k(X_1(x),\dots,X_n(x))=\frac{\cos^+_{k+\rho}(x)}{\cos^+_{\rho}(x)}\,,&\quad\quad&\mathcal{P}^{III,-}_k(X_1(x),\dots,X_n(x))=\frac{\cos^-_{k+\rho_2}(x)}{\cos^-_{\rho_2}(x)}\,,
\end{alignedat}
\end{equation}
determine four classes of orthogonal polynomials $\mathcal{P}^{I,+}_k, \mathcal{P}^{I,-}_k, \mathcal{P}^{III,-}_k,\mathcal{P}^{III,-}_k \in \R[X_1,\dots,X_n] $ of degree $k_1$.  
Since from Corollary \ref{nonzero} it follows that the three functions in the denominators of \eqref{newpol} have nonzero values inside $F(\wt{S}^{\mathrm{aff}}_n)$, the functions \eqref{newpol} are well defined for any point in the interior of $F(\wt{S}^{\mathrm{aff}}_n)$. Moreover, whenever one of the denominators is zero at some points of the boundary of $F(\wt{S}^{\mathrm{aff}}_n)$ the corresponding nominator is zero as well. 
We introduce an ordering within each family of polynomials \eqref{newpol}. We say that a polynomial $p_1$  depending on $k=(k_1,\dots,k_n)\in P^+$ is greater than any polynomial $p_2$  depending on $k'=(k_1',\dots,k'_n)\neq k$ if  for all $i\in \{1,\dots,n\} $ it holds that $k_i\geq k_i'$; equally, we state that $p_2$ is lower than $p_1$. 
Note that the functions \eqref{newpol} can be viewed as generalizations of Chebyshev polynomials of the first and third kinds \cite{Hand}.

\subsection{Recurrence relations}\label{secrec}\

The construction of polynomials is based on the decomposition of products of symmetric and antisymmetric cosine functions. There are three types of products which decompose to a sum of either symmetric or antisymmetric cosine functions. Such a decomposition is obtained by using classical trigonometric identities and is completely described by 
\begin{equation}\label{prod}
\begin{aligned}
\cos^+_{(\lambda_1,\dots,\lambda_n)}\cdot\cos^+_{(\mu_1,\dots,\mu_n)}&=\frac{1}{2^n}\sum_{\sigma\in S_n}\sum_{\setcomb{a_i=\pm1}{i=1,\dots,n}}\cos^+_{(\la_1+a_1\mu_{\sigma(1)},\dots, \la_n+a_n\mu_{\sigma(n)})}\,,\\
\cos^-_{(\lambda_1,\dots,\lambda_n)}\cdot\cos^+_{(\mu_1,\dots,\mu_n)}&=\frac{1}{2^n}\sum_{\sigma\in S_n}\sum_{\setcomb{a_i=\pm1}{i=1,\dots,n}}\cos^-_{(\la_1+a_1\mu_{\sigma(1)},\dots, \la_n+a_n\mu_{\sigma(n)})}\,,\\
\cos^-_{(\lambda_1,\dots,\lambda_n)}\cdot\cos^-_{(\mu_1,\dots,\mu_n)}&=\frac{1}{2^n}\sum_{\sigma\in S_n}\sum_{\setcomb{a_i=\pm1}{i=1,\dots,n}}\mathrm{sgn}\,(\sigma)\cos^+_{(\la_1+a_1\mu_{\sigma(1)},\dots, \la_n+a_n\mu_{\sigma(n)})}\,.
\end{aligned}
\end{equation}

Using \eqref{prod}, we obtain the following pertinent recurrence relations. Let $\la=(\la_1,\dots,\la_n)$, and let $l_i$ be a vector in $\R^n$ with $i$th component equal to $1$ and others to $0$, then 
\begin{align}
\cos^{\pm}_\la&=\frac{2^1}{1!(n-1)!}\cos^{\pm}_{\la-l_1}X_1-\cos^{\pm}_{\la-2l_1}-\sum_{i=2}^n\left(\cos^\pm_{\la-l_1+l_i}+\cos^\pm_{\la-l_1-l_i}\right)\,,\notag\\
\cos^{\pm}_\la&=\frac{2^2}{2!(n-2)!}\cos^{\pm}_{\la-l_1-l_2}X_2-\cos^{\pm}_{\la-2l_1-2l_2}-\cos^{\pm}_{\la-2l_1}- \cos^{\pm}_{\la-2l_2}\notag\\&
-\sum_{i=3}^n\left(\cos^{\pm}_{\la-l_2+l_i}+\cos^\pm_{\la-2l_1-l_2-l_i}+\cos^{\pm}_{\la-l_2-l_i}+\cos^\pm_{\la-2l_1-l_2+l_i}\right)\notag\\&-\sum_{\setcomb{i,j=2}{i< j}}^n\left(\cos^\pm_{\la-l_1-l_2+l_i+l_j}+\cos^\pm_{\la-l_1-l_2-l_i-l_j}+\cos^{\pm}_{\la-l_1-l_2+l_i-l_j}+\cos^{\pm}_{\la-l_1-l_2-l_i+l_j}\right)\,,\notag\\
&\vdots\notag\\
\cos^{\pm}_\la&=\frac{2^n}{n!}\cos^\pm_{\la-l_1-l_2-\dots-l_n}X_n-\sum_{i=1}^n\cos^\pm_{\la-2l_i}-\sum_{\setcomb{i,j=1}{i<j}}^n\cos^\pm_{\la-2l_i-2l_j}-\dots-\cos^\pm_{\la-2l_1-2l_2-\dots-2l_n}\,.\label{perrel}
\end{align}
Using the relations \eqref{perrel} and the symmetry properties of $\cos^\pm_\lambda$ each polynomial of \eqref{newpol} is expressed as a linear combination of lower polynomials and a product of some lower polynomial with some $X_i$. Therefore, all polynomials \eqref{newpol} are built recursively.
\begin{tvr}\label{proppol}
Let $k\in P^+$. The functions $\mathcal{P}^{I,\pm}_k$ and $\mathcal{P}^{III,\pm}_k$  are expressed as polynomials of degree $k_1$ in variables $X_1,\dots, X_n$. The number of $\mathcal{P}^{I,\pm}_k$ or $\mathcal{P}^{III,\pm}_k$ with $k_1=d$ is equal to the number of monomials of degree $d$, i.e.,
\begin{equation}
\binom{d+n-1}{ d }\,.
\end{equation}
\end{tvr}
\begin{proof}
At first, we proceed by induction on $k_1$ to show  that any $\mathcal{P}^{I,+}_k$ is expressed as a polynomial of degree $k_1$ in the variables \eqref{var}. 
\begin{itemize}
\item If $k_1=0$, $\mathcal{P}^{I,+}_{(0,\dots,0)}$ is trivially the constant polynomial $n!$ of degree $0$. 
\item If $k_1=1$, the polynomials $X_1,X_2,\dots,X_n$ of degree $1$ correspond by definition to the set of functions $\mathcal{P}^{I,+}_{(1,k_2,\dots,k_n)}$, where $1\geq k_2\geq\dots\geq k_n\geq0$. 
\item If $k_1=2$, using relations \eqref{perrel} and the basic properties of symmetric cosine functions, we deduce that any $\mathcal{P}^{I,+}_{(2,k_2,\dots,k_n)}$ is constructed from the decomposition of the products $X_iX_j$. Indeed, we start by obtaining the lowest polynomial $\mathcal{P}^{I,+}_{(2,0,\dots,0)}$ from the product $X_1^2$ which decomposes into the linear combination of $\mathcal{P}^{I,+}_{(2,0,\dots,0)}$ and polynomials of degree $1$. We continue by the decomposition of the product $X_1X_2$ to have the polynomial expression of $\mathcal{P}^{I,+}_{(2,1,\dots,0)}$  and so on. Finally, we obtain that each $\mathcal{P}^{I,+}_{(2,k_2,\dots,k_n)}$ is expressed as a polynomial of degree $2$ in variables \eqref{var}.
\item We next suppose that every $\mathcal{P}^{I,+}_{(d-1,k_2,\dots,k_n)}$ is a polynomial of degree $d-1$ in variables \eqref{var}, then again using relations \eqref{perrel} the polynomial  $\mathcal{P}^{I,+}_{(d,k_2,\dots,k_n)}$ is a linear combination of lower polynomials of degree at most $d$ and $\mathcal{P}^{I,+}_{(d-1,\wt{k}_2,\dots,\wt{k}_n)}X_i$ of degree $d$. 
\end{itemize}
By induction, this results in the fact that any $\mathcal{P}^{I,+}_k$ is expressed as a polynomial of degree $k_1$ in the variables \eqref{var}. 
By similar arguments, we obtain the same statement for the functions $\mathcal{P}^{I,-}_k$ and $\mathcal{P}^{III,\pm}_k$. Note that we start with the products of $\cos^-_{\rho_1}$, $\cos^+_{\rho}$, $\cos^-_{\rho_2}$, respectively, with the variables $X_i$.

To prove the second statement, we observe that the number of polynomials of type $\mathcal{P}^{I,\pm}_k$ or $\mathcal{P}^{III,\pm}_k$ with $k_1=d$ is equal to the number of elements in
$\set{(k_2,k_3,\dots,k_n)\in\Z^{n-1}}{d\geq k_2\geq k_3\geq\dots\geq k_n\geq0}$ which is the same as
$$\#\set{(l_n+l_{n-1}+\dots +l_2,\dots, l_n+l_{n-1},l_n)}{l_i,l_0\in\Z^{\geq0},l_0+\sum_{i=2}^nl_i= d}$$
or equivalently$$\#\set{X_1^{\alpha_1}X_2^{\alpha_2}\dots X_n^{\alpha_n}}{\alpha_i\geq0,\sum_{i=1}^n\alpha_i=d}=\comb{d+n-1}{ d}\,.$$
The proof of the last equality is found in \cite{Xuort}.
\end{proof}

\begin{example}
In particular, the relations \eqref{perrel} together with the symmetry properties of symmetric cosine functions imply the following set of recurrence relations for $P^{I,+}_{(k_1,k_2,k_3)}$:
\begin{equation}
\mathcal{P}^{I,+}_{(0,0,0)}=6\,,\quad\mathcal{P}^{I,+}_{(1,0,0)}=X_1\,,\quad \mathcal{P}^{I,+}_{(1,1,0)}=X_2\,,\quad
\mathcal{P}^{I,+}_{(1,1,1)}=X_3\,,
\end{equation}
\begin{alignat}{3}
&k_1\geq2,\,k_2=k_3=0:&\quad& \mathcal{P}^{I,+}_{(k_1,0,0)}=\mathcal{P}^{I,+}_{(k_1-1,0,0)}X_1-\mathcal{P}^{I,+}_{(k_1-2,0,0)}-4\mathcal{P}^{I,+}_{(k_1-1,1,0)}\,,\notag\\
&k_1-1 > k_2>k_3=0      :&\quad &\mathcal{P}^{I,+}_{(k_1,k_2,0)}=\mathcal{P}^{I,+}_{(k_1-1,k_2,0)}X_1 -\mathcal{P}^{I,+}_{(k_1-2,k_2,0)}
   -2\mathcal{P}^{I,+}_{(k_1-1,k_2,1)} -\mathcal{P}^{I,+}_{(k_1-1,k_2+1,0)}\notag\\ && &\qquad -\mathcal{P}^{I,+}_{(k_1-1,k_2-1,0)}\,,\notag\\
&k_1-1> k_2=k_3>0:&\quad& \mathcal{P}^{I,+}_{(k_1,k_2,k_2)}=\mathcal{P}^{I,+}_{(k_1-1,k_2,k_2)}X_1-\mathcal{P}^{I,+}_{(k_1-2,k_2,k_2)}-2\mathcal{P}^{I,+}_{(k_1-1,k_2+1,k_2)}
-2\mathcal{P}^{I,+}_{(k_1-1,k_2,k_2-1)}\,,\notag\\
&k_1-1> k_2>k_3>0:&\quad &\mathcal{P}^{I,+}_{(k_1,k_2,k_3)}=\mathcal{P}^{I,+}_{(k_1-1,k_2,k_3)}X_1-\mathcal{P}^{I,+}_{(k_1-2,k_2,k_3)}-\mathcal{P}^{I,+}_{(k_1-1,k_2,k_3+1)}
-\mathcal{P}^{I,+}_{(k_1-1,k_2,k_3-1)}\notag\\&&&\qquad-\mathcal{P}^{I,+}_{(k_1-1,k_2+1,k_3)}-\mathcal{P}^{I,+}_{(k_1-1,k_2-1,k_3)}\,,\notag\\
&k_1-1=k_2>k_3=0      :&\quad& \mathcal{P}^{I,+}_{(k_1,k_1-1,0)}=\frac12\mathcal{P}^{I,+}_{(k_1-1,k_1-1,0)}X_1-\mathcal{P}^{I,+}_{(k_1-1,k_1-2,0)}
    -\mathcal{P}^{I,+}_{(k_1-1,k_1-1,1)}\,,\notag\\
 &k_1-1=k_2>k_3>0:&\quad &\mathcal{P}^{I,+}_{(k_1,k_1-1,k_3)}=\frac12\mathcal{P}^{I,+}_{(k_1-1,k_1-1,k_3)}X_1-\mathcal{P}^{I,+}_{(k_1-1,k_1-2,k_3)}-\frac12\mathcal{P}^{I,+}_{(k_1-1,k_1-1,k_3+1)}
\notag\\&&&\qquad-\frac12\mathcal{P}^{I,+}_{(k_1-1,k_1-1,k_3-1)}\,,\notag\\
&k_1-1=k_2=k_3>0:&\quad& \mathcal{P}^{I,+}_{(k_1,k_1-1,k_1-1)}=\frac13\mathcal{P}^{I,+}_{(k_1-1,k_1-1,k_1-1)}X_1-\mathcal{P}^{I,+}_{(k_1-1,k_1-1,k_1-2)}\,,\notag\\
&k_1=k_2>2,k_3=0 :&\quad &\mathcal{P}^{I,+}_{(k_1,k_1,0)}=2\mathcal{P}^{I,+}_{(k_1-1,k_1-1,0)}X_2-2\mathcal{P}^{I,+}_{(k_1-1,k_1-2,0)}X_1-2\mathcal{P}^{I,+}_{(k_1-1,k_1-1,1)}X_1
        \notag\\&&&\qquad+\mathcal{P}^{I,+}_{(k_1-2,k_1-2,0)}  +4\mathcal{P}^{I,+}_{(k_1-1,k_1-1,0)}+4\mathcal{P}^{I,+}_{(k_1-1,k_1-2,1)}+2\mathcal{P}^{I,+}_{(k_1-1,k_1-3,0)}\notag\\ && &\qquad+2\mathcal{P}^{I,+}_{(k_1-1,k_1-1,2)}\,,\notag\\
&k_1=k_2>k_3+2>2:&\quad& \mathcal{P}^{I,+}_{(k_1,k_1,k_3)}=2\mathcal{P}^{I,+}_{(k_1-1,k_1-1,k_3)}X_2-2\mathcal{P}^{I,+}_{(k_1-1,k_1-2,k_3)}X_1-\mathcal{P}^{I,+}_{(k_1-1,k_1-1,k_3+1)}X_1 \notag\\&&&\qquad-\mathcal{P}^{I,+}_{(k_1-1,k_1-1,k_3-1)}X_1
       +\mathcal{P}^{I,+}_{(k_1-2,k_1-2,k_3)}+2\mathcal{P}^{I,+}_{(k_1-1,k_1-2,k_3+1)}\notag\\&&&\qquad+2\mathcal{P}^{I,+}_{(k_1-1,k_1-2,k_3-1)}
       +4\mathcal{P}^{I,+}_{(k_1-1,k_1-1,k_3)}+2\mathcal{P}^{I,+}_{(k_1-1,k_1-3,k_3)}+\mathcal{P}^{I,+}_{(k_1-1,k_1-1,k_3+2)}\notag\\&&&\qquad+\mathcal{P}^{I,+}_{(k_1-1,k_1-1,k_3-2)}\,,\notag\\
&k_1=k_2=k_3+2>3:&\quad&\mathcal{P}^{I,+}_{(k_1,k_1,k_1-2)}=2\mathcal{P}^{I,+}_{(k_1-1,k_1-1,k_1-2)}X_2-2\mathcal{P}^{I,+}_{(k_1-1,k_1-2,k_1-2)}X_1
       -\frac23\mathcal{P}^{I,+}_{(k_1-1,k_1-1,k_1-1)}X_1\notag\\&&&\qquad-\mathcal{P}^{I,+}_{(k_1-1,k_1-1,k_1-3)}X_1+\mathcal{P}^{I,+}_{(k_1-2,k_1-2,k_1-2)}+5\mathcal{P}^{I,+}_{(k_1-1,k_1-1,k_1-2)}
     \notag \\&&&\qquad +4\mathcal{P}^{I,+}_{(k_1-1,k_1-2,k_1-3)}+\mathcal{P}^{I,+}_{(k_1-1,k_1-1,k_1-4)}\,,\notag\\
&k_1=k_2=k_3+1>2:&\quad& \mathcal{P}^{I,+}_{(k_1,k_1,k_1-1)}=\frac23\mathcal{P}^{I,+}_{(k_1-1,k_1-1,k_1-1)}X_2-\mathcal{P}^{I,+}_{(k_1-1,k_1-1,k_1-2)}X_1
        +\mathcal{P}^{I,+}_{(k_1-1,k_1-2,k_1-2)}\notag\\&&&\qquad+\mathcal{P}^{I,+}_{(k_1-1,k_1-1,k_1-1)}+\mathcal{P}^{I,+}_{(k_1-1,k_1-1,k_1-3)}\,,\notag\\
&k_1=k_2=k_3>3:&\quad& \mathcal{P}^{I,+}_{(k_1,k_1,k_1)}=\frac43 \mathcal{P}^{I,+}_{(k_1-1,k_1-1,k_1-1)}X_3-6\mathcal{P}^{I,+}_{(k_1-1,k_1-1,k_1-2)}X_2+3\mathcal{P}^{I,+}_{(k_1-1,k_1-2,k_1-2)}X_1\notag\\&&    &\qquad+2\mathcal{P}^{I,+}_{(k_1-1,k_1-1,k_1-1)}X_1+3\mathcal{P}^{I,+}_{(k_1-1,k_1-1,k_1-3)}X_1-\mathcal{P}^{I,+}_{(k_1-2,k_1-2,k_1-2)}
\notag\\&&    &\qquad-9\mathcal{P}^{I,+}_{(k_1-1,k_1-1,k_1-2)}-6\mathcal{P}^{I,+}_{(k_1-1,k_1-2,k_1-3)}-3\mathcal{P}^{I,+}_{(k_1-1,k_1-1,k_1-4)}
 \end{alignat}
together with additional polynomial expressions
\begin{equation}
\begin{alignedat}{2}
&\mathcal{P}^{I,+}_{(2,2,0)}&=&2X_2^2-2X_1^2
       -\frac43 X_1X_3+8X_2+6\,,\\
& \mathcal{P}^{I,+}_{(2,2,1)}&=&\frac23X_2X_3-X_1X_2+2X_3
        +X_1\,,\\      
&\mathcal{P}^{I,+}_{(2,2,2)}&=&\frac43X_3^2-6X_2^2+3X_1^2+4X_1X_2-12X_2-6 \,,\\
&\mathcal{P}^{I,+}_{(3,3,1)}&=&\frac43X_2^2X_3-\frac89X_1X_3^2-2X_1^2X_3+8X_2X_3+9X_3\,,\\
&\mathcal{P}^{I,+}_{(3,3,3)}&=&\frac{16}{9}X_3^3-12X_2^2X_3+8X_1X_3^2+9X_1^2X_3-36X_2X_3-27X_3 \,.\\       
\end{alignedat}
\end{equation}

Similarly, one may find recurrence relations for $\mathcal{P}^{I,-}_{(k_1,k_2,k_3)}$ and $\mathcal{P}^{III,\pm}_{(k_1,k_2,k_3)}$. The polynomials $\mathcal{P}^{I,\pm}_{(k_1,k_2,k_3)}$ and $\mathcal{P}^{III,\pm}_{(k_1,k_2,k_3)}$ of degree at most two are shown in Tables \ref{tabpIe} -- \ref{tabpIII-}.

\begin{table}
{\small
\begin{tabular}{|c|r|r|r|r|r|r|}
\hline
&$1$&$X_2$&$X_1^2$&$X_1X_3$&$X_2^2$&$X_3^2$\\ \hline
$\mathcal{P}^{I,+}_{(0,0,0)}$&$6$&&&&&\\ \hline
$\mathcal{P}^{I,+}_{(1,1,0)}$&$0$&$1$&&&&\\ \hline
$\mathcal{P}^{I,+}_{(2,0,0)}$&$-6$&$-4$&$1$&&&\\ \hline
$\mathcal{P}^{I,+}_{(2,1,1)}$&$0$&$-1$&$0$&$\frac13$&&\\ \hline
$\mathcal{P}^{I,+}_{(2,2,0)}$&$6$&$8$&$-2$&$-\frac43$&$2$&\\ \hline
$\mathcal{P}^{I,+}_{(2,2,2)}$&$-6$&$-12$&$3$&$4$&$-6$&$\frac43$\\ \hline
\end{tabular}}
\hspace{0.5cm}
{\small
\begin{tabular}{|c|r|r|r|r|r|r|}
\hline
&$1$&$X_2$&$X_1^2$&$X_1X_3$&$X_2^2$&$X_3^2$\\ \hline
$\mathcal{P}^{I,-}_{(0,0,0)}$&$1$&&&&&\\ \hline
$\mathcal{P}^{I,-}_{(1,1,0)}$&$3$&$2$&&&&\\ \hline
$\mathcal{P}^{I,-}_{(2,0,0)}$&$-4$&$-2$&$1$&&&\\ \hline
$\mathcal{P}^{I,-}_{(2,1,1)}$&$-3$&$-2$&$1$&$\frac23$&&\\ \hline
$\mathcal{P}^{I,-}_{(2,2,0)}$&$12$&$14$&$-3$&$-\frac43$&$4$&\\ \hline
$\mathcal{P}^{I,-}_{(2,2,2)}$&$-9$&$-12$&$3$&$\frac{10}{3}$&$-4$&$\frac89$\\ \hline
\end{tabular}}
\medskip
\caption{The coefficients of the polynomials $P^{I,\pm}_{(k_1,k_2,k_3)}$ with $k_1\leq2$ and $k_1+k_2+k_3$ even.}
\label{tabpIe}
\end{table}

\begin{table}
{\small
\begin{tabular}{|c|r|r|r|r|}
\hline
&$X_1$&$X_3$&$X_1X_2$&$X_2X_3$\\ \hline
$\mathcal{P}^{I,+}_{(1,0,0)}$&$1$&&&\\ \hline
$\mathcal{P}^{I,+}_{(1,1,1)}$&$0$&$1$&&\\ \hline
$\mathcal{P}^{I,+}_{(2,1,0)}$&$-1$&$-1$&$\frac12$&\\ \hline
$\mathcal{P}^{I,+}_{(2,2,1)}$&$1$&$2$&$-1$&$\frac23$\\ \hline
\end{tabular}}
\hspace{0.5cm}
{\small
\begin{tabular}{|c|r|r|r|r|}
\hline
&$X_1$&$X_3$&$X_1X_2$&$X_2X_3$\\ \hline
$\mathcal{P}^{I,-}_{(1,0,0)}$&$1$&&&\\ \hline
$\mathcal{P}^{I,-}_{(1,1,1)}$&$1$&$\frac23$&&\\ \hline
$\mathcal{P}^{I,-}_{(2,1,0)}$&$0$&$-\frac43$&$2$&\\ \hline
$\mathcal{P}^{I,-}_{(2,2,1)}$&$1$&$\frac83$&$0$&$\frac43$\\ \hline
\end{tabular}}
\medskip
\caption{The coefficients of the polynomials $\mathcal{P}^{I,\pm}_{(k_1,k_2,k_3)}$ with $k_1\leq2$ and $k_1+k_2+k_3$ odd.}
\label{tabpIo}
\end{table}

\begin{table}
{\small
\begin{tabular}{|c|r|r|r|r|r|r|r|r|r|r|}
\hline
&$1$&$X_1$&$X_2$&$X_3$&$X_1^2$&$X_1X_2$&$X_1X_3$&$X_2^2$&$X_2X_3$&$X_3^2$\\ \hline
$\mathcal{P}^{III,+}_{(0,0,0)}$&$1$&&&&&&&&&\\ \hline
$\mathcal{P}^{III,+}_{(1,0,0)}$&$-1$&$\frac13$&&&&&&&&\\ \hline
$\mathcal{P}^{III,+}_{(1,1,0)}$&$1$&$-\frac23$&$\frac23$&&&&&&&\\ \hline
$\mathcal{P}^{III,+}_{(1,1,1)}$&$-1$&$1$&$-2$&$\frac43$&&&&&&\\ \hline
$\mathcal{P}^{III,+}_{(2,0,0)}$&$-1$&$-\frac13$&$-\frac43$&$0$&$\frac13$&&&&&\\ \hline
$\mathcal{P}^{III,+}_{(2,1,0)}$&$1$&$0$&$\frac23$&$-\frac23$&$-\frac13$&$\frac13$&&&&\\ \hline
$\mathcal{P}^{III,+}_{(2,1,1)}$&$-1$&$\frac13$&$-\frac23$&$0$&$\frac13$&$-\frac23$&$\frac49$&&&\\ \hline
$\mathcal{P}^{III,+}_{(2,2,0)}$&$1$&$\frac23$&$\frac{10}{3}$&$\frac43$&$-\frac23$&$-\frac23$&$-\frac89$&$\frac43$&&\\ \hline
$\mathcal{P}^{III,+}_{(2,2,1)}$&$-1$&$-\frac13$&$-2$&$\frac43$&$\frac23$&$0$&$0$&$-\frac43$&$\frac89$&\\ \hline
$\mathcal{P}^{III,+}_{(2,2,2)}$&$-1$&$-1$&$-6$&$-\frac{16}{3}$&$1$&$2$&$4$&$-4$&$-\frac83$&$\frac{16}{9}$\\ \hline
\end{tabular}}
\medskip
\caption{The coefficients of the polynomials $\mathcal{P}^{III,+}_{(k_1,k_2,k_3)}$ with $k_1\leq2$.}
\label{tabpIII+}
\end{table}

\begin{table}
{\small
\begin{tabular}{|c|r|r|r|r|r|r|r|r|r|r|}
\hline
&$1$&$X_1$&$X_2$&$X_3$&$X_1^2$&$X_1X_2$&$X_1X_3$&$X_2^2$&$X_2X_3$&$X_3^2$\\ \hline
$\mathcal{P}^{III,-}_{(0,0,0)}$&$1$&&&&&&&&&\\ \hline
$\mathcal{P}^{III,-}_{(1,0,0)}$&$-1$&$1$&&&&&&&&\\ \hline
$\mathcal{P}^{III,-}_{(1,1,0)}$&$3$&$-1$&$2$&&&&&&&\\ \hline
$\mathcal{P}^{III,-}_{(1,1,1)}$&$-3$&$2$&$-2$&$\frac43$&&&&&&\\ \hline
$\mathcal{P}^{III,-}_{(2,0,0)}$&$-3$&$-1$&$-2$&$0$&$1$&&&&&\\ \hline
$\mathcal{P}^{III,-}_{(2,1,0)}$&$1$&$1$&$0$&$-\frac43$&$-1$&$2$&&&&\\ \hline
$\mathcal{P}^{III,-}_{(2,1,1)}$&$-3$&$-2$&$-2$&$0$&$2$&$-2$&$\frac43$&&&\\ \hline
$\mathcal{P}^{III,-}_{(2,2,0)}$&$8$&$0$&$12$&$\frac43$&$-2$&$-2$&$-\frac43$&$4$&&\\ \hline
$\mathcal{P}^{III,-}_{(2,2,1)}$&$-6$&$2$&$-10$&$4$&$1$&$2$&$0$&$-4$&$\frac83$&\\ \hline
$\mathcal{P}^{III,-}_{(2,2,2)}$&$-6$&$-2$&$-10$&$-\frac{16}{3}$&$3$&$0$&$\frac{16}{3}$&$-4$&$-\frac83$&$\frac{16}{9}$\\ \hline
\end{tabular}}
\medskip
\caption{The coefficients of the polynomials $\mathcal{P}^{III,-}_{(k_1,k_2,k_3)}$ with $k_1\leq2$.}
\label{tabpIII-}
\end{table}
\end{example}

\subsection{Continuous orthogonality}\

Continuous orthogonality of antisymmetric and symmetric cosine functions within each family is detailed in Section \ref{seccon}. Our goal is to reformulate the orthogonality relations \eqref{O1} -- \eqref{O4}  after the change of variables $(x_1,\dots,x_n)$ to polynomial variables $(X_1,\dots, X_n)$ and obtain the orthogonality relations for polynomials $\mathcal{P}^{I,\pm}_k$ and $\mathcal{P}^{III,\pm}_k$. In order to determine the corresponding weight functions in the integrals defining continuous orthogonality, we calculate the Jacobian of the change of variables $(X_1(x_1,\dots, x_n),\dots, X_n(x_1,\dots, x_n))$ to variables $(x_1,\dots, x_n)$.
\begin{tvr}
The determinant of the Jacobian matrix $J(x_1,\dots, x_n)\equiv\det\frac{\partial(X_1,\dots, X_n)}{\partial (x_1,\dots, x_n)}$ for the coordinate change from $(X_1,\dots, X_n)$ to $(x_1,\dots, x_n)$ is given by
\begin{equation}\label{jac}
J(x_1,\dots,x_n)=(-1)^{\frac{n(n+1)}{2}}\left(\frac12\right)^{\frac{n(n-1)}{2}}\pi^n\left(\prod_{i=1}^n(n-i)!i!\right)\sin^-_{(1,2,\dots,n)}(x_1,\dots,x_n)\,,
\end{equation}
where \begin{equation}
\begin{aligned}\label{for1}
\sin^-_{(1,2,\dots,n)}(x_1,\dots,x_n)&\equiv
2^{n(n-1)}\prod_{i=1}^n\sin(\pi x_i)\prod_{1\leq i<j\leq n}\sin\left(\frac{\pi}{2}(x_i+x_j)\right)\sin\left(\frac{\pi}{2}(x_i-x_j)\right)\,.
\end{aligned}
\end{equation}
\end{tvr}
\begin{proof} We prove the formulas \eqref{jac}, \eqref{for1} by direct calculation.
By the definition of the variables $X_j$ via symmetric cosine functions, we have
$$X_j=(n-j)!j!\sum_{\setcomb{i_1,\dots,i_j=1}{i_1<\dots <i_j}}^n\cos{(\pi x_{i_1})}\dots \cos{(\pi x_{i_j})}\,.$$
Thus, the partial derivatives of $X_j$ with respect to $x_k$ are given by
\begin{align*}
j=1: \quad \frac{\partial X_1}{\partial x_k}&=-\pi (n-1)!1! \sin{(\pi x_k)}\,,\\
j\geq2:\quad \frac{\partial X_j}{\partial x_k}&=-\pi (n-j)!j! \sin{(\pi x_k)}\left(\sum_{\setcomb{i_1,\dots,i_{j-1}=1}{i_1<\dots <i_{j-1}}}^n\cos{(\pi x_{i_1})}\dots \cos{(\pi x_{i_{j-1}})}\right.\\&
-\cos{(\pi x_k)}\left(\sum_{\setcomb{i_1,\dots,i_{j-2}=1}{i_1<\dots <i_{j-2}}}^n\cos{(\pi x_{i_1})}\dots \cos{(\pi x_{i_{j-2}})}\right.\\&\left.\left.-\cos{(\pi x_k)}\left(\dots -\cos{(\pi x_k)}\left(\sum_{i_1=1}^n\cos{(\pi x_{i_1})}-\cos{(\pi x_k)}\right)\dots\right)\right)\right)\,.
\end{align*}
Therefore, using the properties of determinants, we rewrite the Jacobian as 
$$J(x_1,\dots, x_n)=\det\left(\begin{smallmatrix}
-\pi(n-1)!1!\sin(\pi x_1)&\pi(n-2)!2!\sin(\pi x_1)\cos(\pi x_1)&\dots&(-1)^{n}\pi 0!n!\sin(\pi x_1)\cos^{n-1}(\pi x_1)\\
-\pi(n-1)!1!\sin(\pi x_2)&\pi(n-2)!2!\sin(\pi x_2)\cos(\pi x_2)&\dots&(-1)^{n}\pi 0!n!\sin(\pi x_2)\cos^{n-1}(\pi x_2)\\
\vdots&\vdots&&\vdots\\
-\pi(n-1)!1!\sin(\pi x_n)&\pi(n-2)!2!\sin(\pi x_n)\cos(\pi x_n)&\dots&(-1)^{n}\pi 0!n!\sin(\pi x_n)\cos^{n-1}(\pi x_n)
\end{smallmatrix}\right)\,.$$
Comparing this form of the Jacobian with \eqref{jac}, we observe that
$$\sin_{(1,2,\dots,n)}^-(x_1,\dots,x_n)=2^{\frac{n(n-1)}{2}}\det\left(\begin{smallmatrix}
\sin(\pi x_1)&\sin(\pi x_1)\cos(\pi x_1)&\dots&\sin(\pi x_1)\cos^{n-1}(\pi x_1)\\
\sin(\pi x_2)&\sin(\pi x_2)\cos(\pi x_2)&\dots&\sin(\pi x_2)\cos^{n-1}(\pi x_2)\\
\vdots&\vdots&&\vdots\\
\sin(\pi x_n)&\sin(\pi x_n)\cos(\pi x_n)&\dots&\sin(\pi x_n)\cos^{n-1}(\pi x_n)
\end{smallmatrix}\right)\,.$$
Therefore, the equality \eqref{for1} follows directly from the proof of Proposition \ref{prop1}.
\end{proof}

One verifies that $\sin^-_{(1,2,\dots,n)}\cdot\sin^-_{(1,2,\dots,n)}$ decomposes into a sum of symmetric cosine functions labeled by integer parameters
$$\sin^-_{(1,2,\dots,n)}\cdot\sin^-_{(1,2,\dots,n)}=\frac{1}{2^n}\sum_{\sigma\in S_n} \mathrm{sgn}\,(\sigma)\sum_{\setcomb{a_i=\pm1}{i=1,\dots,n}}(-1)^{\alpha(a_1,\dots,a_n)}\cos^+_{(1+a_1\sigma(1),\dots, n+a_n\sigma(n))}\,,$$
where $\alpha(a_1,\dots,a_n)$ is the number of positive numbers among $a_1,\dots, a_n$.
Therefore due to Proposition~\ref{proppol} the square of the Jacobian $J^2$ is expressed as a polynomial $p^{I,+}$ in $X_1,\dots, X_n$ of degree $2n$,
\begin{equation*}
J(x)^2=p^{I,+}(X_1(x),\dots, X_n(x)).	
\end{equation*}
The polynomial $p^{I,+}$, being the square of the Jacobian $J$, satisfies
\begin{equation}\label{pos}
p^{I,+}(X_1(x),\dots, X_n(x))\geq 0.	
\end{equation}
Moreover, note that due to the equality \eqref{for1} the Jacobian $J$ does not vanish for any point in the interior of $F(\wt{S}_n^{\mathrm{aff}})$ and therefore
\begin{equation}
	p^{I,+}(X_1(x),\dots, X_n(x))>0,\qquad x\in F(\wt{S}_n^{\mathrm{aff}})^\circ .
\end{equation}

The domain $\mathfrak{F}(\wt{S}_n^{\mathrm{aff}})$ is the defined as the transformed domain $F(\wt{S}_n^{\mathrm{aff}})$, i.e.,
\begin{equation*}
\mathfrak{F}(\wt{S}_n^{\mathrm{aff}})=\set{(X_1(x),\dots,X_n(x))}{x\in F(\wt{S}_n^{\mathrm{aff}})}.
\end{equation*}
Fig. \ref{figgauss} shows the domain $\mathfrak{F}(\wt{S}_n^{\mathrm{aff}})$  for $n=3$.
Note that the relation \eqref{pos} implies that
\begin{equation*}
p^{I,+}(X_1,\dots, X_n)\geq 0, \qquad (X_1,\dots,X_n)\in \mathfrak{F}(\wt{S}_n^{\mathrm{aff}}). 
\end{equation*}
Thus, we use the polynomial $p^{I,+}$ to rewrite the absolute value of the Jacobian $J$ as a function $\mathcal{J} $ of $X_1,\dots, X_n$ on $\mathfrak{F}(\wt{S}_n^{\mathrm{aff}})$, 
\begin{equation}
\mathcal{J}(X_1,\dots, X_n)=\sqrt{p^{I,+}(X_1,\dots, X_n)},\quad (X_1,\dots,X_n)\in \mathfrak{F}(\wt{S}_n^{\mathrm{aff}}),
\end{equation}
i.e., it holds that
\begin{equation*}
|J(x)|=\sqrt{p^{I,+}(X_1(x),\dots, X_n(x))}.	
\end{equation*}
Similarly, it follows from \eqref{prod} that the products $\cos^-_{\rho_1}\cdot\cos^-_{\rho_1}$, $\cos^+_\rho\cdot\cos^+_\rho$,  and $\cos^-_{\rho_2}\cdot\cos^-_{\rho_2}$ are expressed as polynomials $\mathcal{J}^{I,-}$, $\mathcal{J}^{III,+}$, and  $\mathcal{J}^{III,-}$ in $X_1,\dots, X_n$ such that
\begin{align}
\mathcal{J}^{I,-}(X_1(x),\dots, X_n(x))&=\cos^-_{\rho_1}(x)\cos^-_{\rho_1}(x), \\
\mathcal{J}^{III,+}(X_1(x),\dots, X_n(x))&=\cos^+_\rho(x)\cos^+_\rho (x),\\
\mathcal{J}^{III,-}(X_1(x),\dots, X_n(x))&=\cos^-_{\rho_2}(x)\cos^-_{\rho_2}(x).
\end{align}

\begin{example}
For $n=3$ are the explicit forms of the polynomial $p^{I,+}$  determined by
\begin{align*}
p^{I,+}(X_1,X_2, X_3) &=\pi^6\left(-8 X_2^3 + X_1^2X_2^2 - 12 X_3^2 + 
   12 X_1X_2X_3 - 
   \frac{4}{3} X_1^3X_3\right) ((3X_2+6)^2-(X_3+3X_1)^2), 
	\end{align*}
and the polynomials $\mathcal{J}^{I,-}$, $\mathcal{J}^{III,+}$, and  $\mathcal{J}^{III,-}$	are given by
\begin{align*}
\mathcal{J}^{I,-}(X_1,X_2, X_3)&=\frac14\left(-8 X_2^3 + X_1^2X_2^2 - 12 X_3^2 + 
   12 X_1X_2X_3 - 
   \frac{4}{3} X_1^3X_3\right)\,,\\
\mathcal{J}^{III,+}(X_1,X_2, X_3)&=\frac34(X_3+3X_2+3X_1+6)\,,\\
\mathcal{J}^{III,-}(X_1,X_2, X_3)&=\frac18 (X_3+3X_2+3X_1+6) \left(-8 X_2^3 + X_1^2X_2^2 - 12 X_3^2 + 
   12 X_1X_2X_3 - 
   \frac{4}{3} X_1^3X_3\right)\,.
\end{align*}
\end{example}

In order to use the integration by substitution, we show that the transform 
\begin{equation}\label{subst}
	\varphi:x\in F(\wt{S}^\mathrm{aff}_n)\map (X_1(x),\dots, X_n(x))\in\mathfrak{F}(\wt{S}_n^{\mathrm{aff}})
\end{equation}
is in fact one-to-one correspondence. Any continuous function $f$ on $F(\wt{S}^\mathrm{aff}_n)$ is expanded in terms of symmetric cosine functions \cite{KPtrig} and thus also in terms of polynomials $\mathcal{P}^{I,+}_k$ as $$f(x)=\sum_k B_k \mathcal{P}^{I,+}_k(X_1(x),\dots, X_n(x))\,.$$
Assume that there are $y,\wt{y}\in F(\wt{S}^\mathrm{aff}_n)$ such that $y\neq \wt{y}$ and $(X_1(y),\dots, X_n(y))=(X_1(\wt{y}),\dots, X_n(\wt{y}))$. Let us define a continuous function $f(x)=x_i$, where the index $i$ is chosen in such a way that $y_i\neq \wt{y}_i$. Using the expansion in polynomials, we obtain that $f(y)=f(\wt{y})$, which is in contradiction with $y_i\neq \wt{y}_i$.
Therefore, we conclude that the transform $\varphi$ is injective.

Due to the orthogonality relations \eqref{O1}--\eqref{O4}, Proposition \ref{proppol}, and integration by substitution \eqref{subst}, we deduce the continuous orthogonality relations for polynomials $\mathcal{P}^{I,\pm}_k,\mathcal{P}^{III,\pm}_k$ in the following statement.
\begin{tvr}\label{orthogbasis}
Let $X=(X_1,\dots, X_n)$ and $dX=dX_1\dots dX_n$.
Each family of polynomials $\mathcal{P}^{I,\pm}_k$, $\mathcal{P}^{III,\pm}_k$ forms an orthogonal basis of the vector space of all polynomials in $\R[X]$ with the scalar product defined by the weighted integral of two polynomials $f,g$,
$$\int_{\mathfrak{F}(\wt{S}_n^{\mathrm{aff}})} f(X)g(X)w(X)\,dX\,,$$
where
$w(X)$ denotes the weight function for each family of polynomials, $$w(X)=\begin{cases}
w^{I,+}(X)=\frac{1}{\mathcal{J}(X)}&\text{for } \mathcal{P}^{I,+}_k\,,\\
w^{I,-}(X)=\frac{\mathcal{J}^{I,-}(X)}{\mathcal{J}(X)}&\text{for } \mathcal{P}^{I,-}_k\,,\\
w^{III,+}(X)=\frac{\mathcal{J}^{III,+}(X)}{\mathcal{J}(X)}&\text{for } \mathcal{P}^{III,+}_k\,,\\
w^{III,-}(X)=\frac{\mathcal{J}^{III,-}(X)}{\mathcal{J}(X)}&\text{for } \mathcal{P}^{III,-}_k\,.
\end{cases}$$
The continuous orthogonality relations are of the form
\begin{equation}\label{contorthogpol}
\begin{alignedat}{2}
&\int\limits_{\mathfrak{F}(\wt{S}_n^{\mathrm{aff}})}
    \mathcal{P}^{I,+}_k(X)\mathcal{P}^{I,+}_{k'}(X)w^{I,+}(X)\,dX
   &= &h_kH_k\delta_{kk'}\,,\\
&\int\limits_{\mathfrak{F}(\wt{S}_n^{\mathrm{aff}})}
    \mathcal{P}^{I,-}_k(X)\mathcal{P}^{I,-}_{k'}(X)w^{I,-}(X)\,dX
   &= &h_k\delta_{kk'}\,, \\  
   &\int\limits_{\mathfrak{F}(\wt{S}_n^{\mathrm{aff}})}
    \mathcal{P}^{III,+}_k(X)\mathcal{P}^{III,+}_{k'}(X)w^{III,+}(X)\,dX
   &= &2^{-n}H_k\delta_{kk'}\,,\\
  & \int\limits_{\mathfrak{F}(\wt{S}_n^{\mathrm{aff}})}
    \mathcal{P}^{III,-}_k(X)\mathcal{P}^{III,-}_{k'}(X)w^{III,-}(X)\,dX
   &= &2^{-n}\delta_{kk'}\,.
   \end{alignedat}
\end{equation}
\end{tvr}

\section{Cubature formulas}\label{seccub}


\subsection{Sets of nodes}\

In Sections \ref{AMDCT} and \ref{SMDCT}, the discrete orthogonality relations of the four symmetric cosine functions and the four antisymmetric cosine transforms over finite sets $F_N^{\mathrm{V},\pm},\dots, F_N^{\mathrm{VIII},\pm}$ are detailed. The orthogonality relations of the types I--IV, which are described in  \cite{KPtrig},  are defined over the sets $F^{\mathrm{I},\pm}_N,\dots, F^{\mathrm{IV},\pm}_N$ given by 
\begin{align}
F_N^{\mathrm{I},\pm}&=\setb{\left(\frac{r_1}{N},\dots,\frac{r_n}{N}\right)}{(r_1,\dots, r_n)\in D^\pm_{N+1}}\,,\\
F_N^{\mathrm{II},\pm}&=\setb{\left(\frac{r_1+\frac12}{N},\dots,\frac{r_n+\frac12}{N}\right)}{(r_1,\dots, r_n)\in D^\pm_N}\,,\\
F_N^{\mathrm{III},\pm}&=\setb{\left(\frac{r_1}{N},\dots,\frac{r_n}{N}\right)}{(r_1,\dots,r_n)\in D^\pm_N}\,,\\
F_N^{\mathrm{IV},\pm}&= F_N^{\mathrm{II},\pm}\,.
\end{align}

The families of polynomials $\mathcal{P}_k^{I,\pm}$ and $\mathcal{P}_k^{III,\pm}$ inherit these orthogonality relations which  are crucial for the validity of cubature formulas. 
The sets of nodes arise from the finite sets $F_N^{\mathrm{I},\pm},\dots, F_N^{\mathrm{VIII},\pm}$  as images of the transform $\varphi$, given by \eqref{subst}. Denoting these images of the sets $F_N^{\mathrm{I},\pm},\dots, F_N^{\mathrm{VIII},\pm}$ as $\mathfrak{F}_N^{\mathrm{I},\pm},\dots, \mathfrak{F}_N^{\mathrm{VIII},\pm}$, i.e.,
\begin{equation}
\mathfrak{F}_N^{t,\pm}\equiv \set{\varphi(s)}{s\in F^{t,\pm}_N}\,,\quad t\in\{\mathrm{I},\dots,\mathrm{VIII}\}\,,
\end{equation} 
we prove that the restriction of $\varphi$ on any set $F^{t,\pm}_N$ is injective in the following Proposition.
\begin{tvr}\label{mrizky}
$$\abs{\mathfrak{F}_N^{t,\pm}}=\abs{F_N^{t,\pm}}\,,\quad t\in\{\mathrm{I},\dots,\mathrm{VIII}\}\,.$$
\end{tvr}
\begin{proof}
Suppose that there exists $s,\wt s\in F_N^{\mathrm{I},+} $ such that $s\neq \wt s$ and $\varphi(s)=\varphi(\wt s)$, i.e., $$(X_1(s),\dots,X_n(s))=(X_1(\wt s),\dots, X_n(\wt s)).$$
According to the expansion resulting from SMDCT I in \cite{KPtrig}, any function $f$ given on $F_N^{\mathrm{I},+}$ is written as a linear combination of a finite number of functions $\cos^+_k$ and therefore as a polynomial in $X_1,\dots, X_n$. Using this expansion in polynomials,  we have for the function $f(s)=s_i$ (with $i$ chosen in such a way that $s_i\neq \wt s_i$) that $f(s)= f(\wt s)$ -- which contradicts $s\neq \wt s$. Thus, we obtain $\abs{\mathfrak{F}_N^{\mathrm{I},+}}=\abs{F_N^{\mathrm{I},+}}$.
Using the corresponding discrete orthogonality relations, the proof for other grids is similar.
\end{proof}

As an example, the points of the set $\mathcal{F}^{\mathrm{II},+}_{N}$ for $N=5,10,15$ in dimension $3$ are depicted in Fig. \ref{figgauss}.

\begin{figure}
\includegraphics[width=5cm]{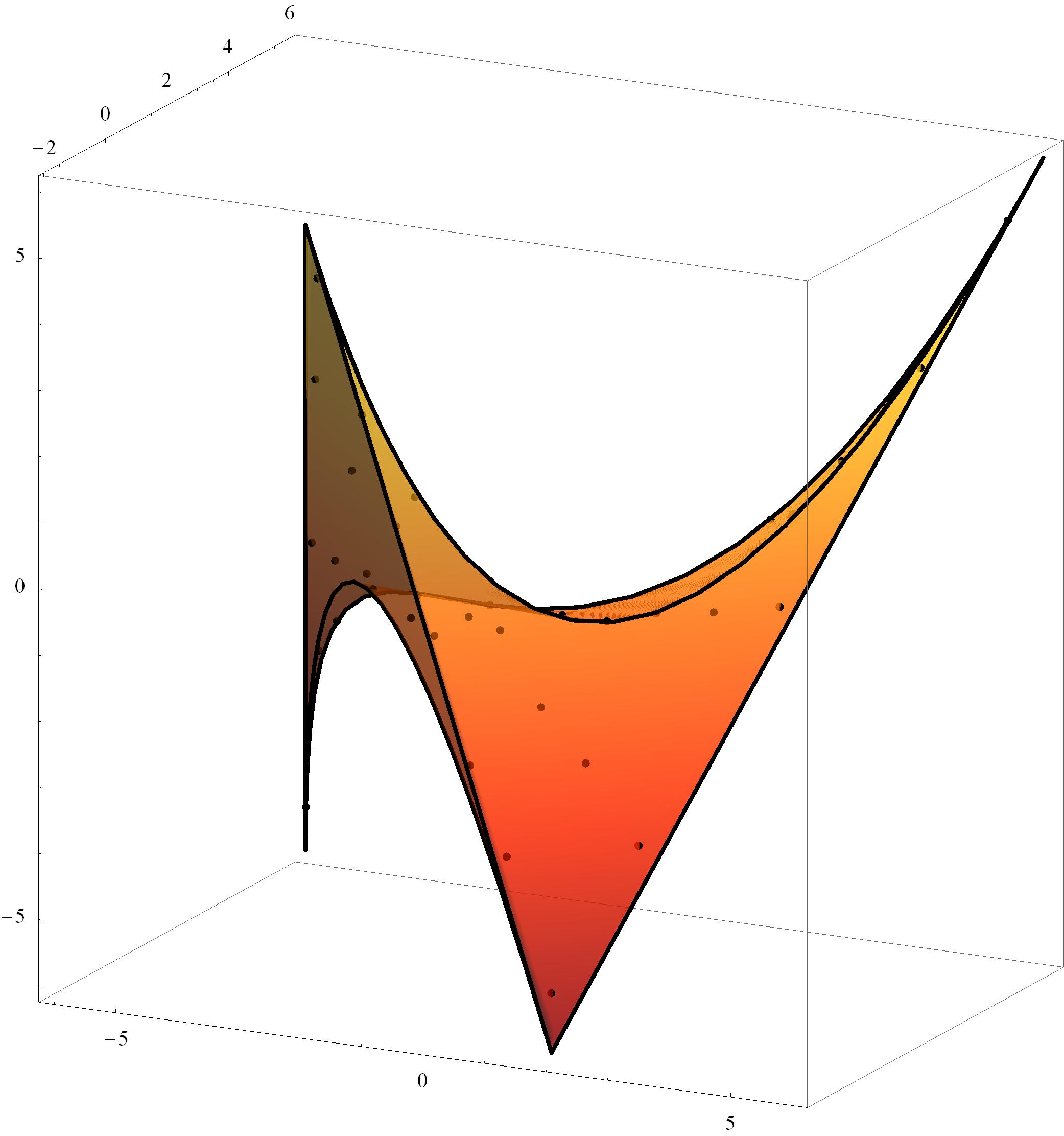}
\includegraphics[width=5cm]{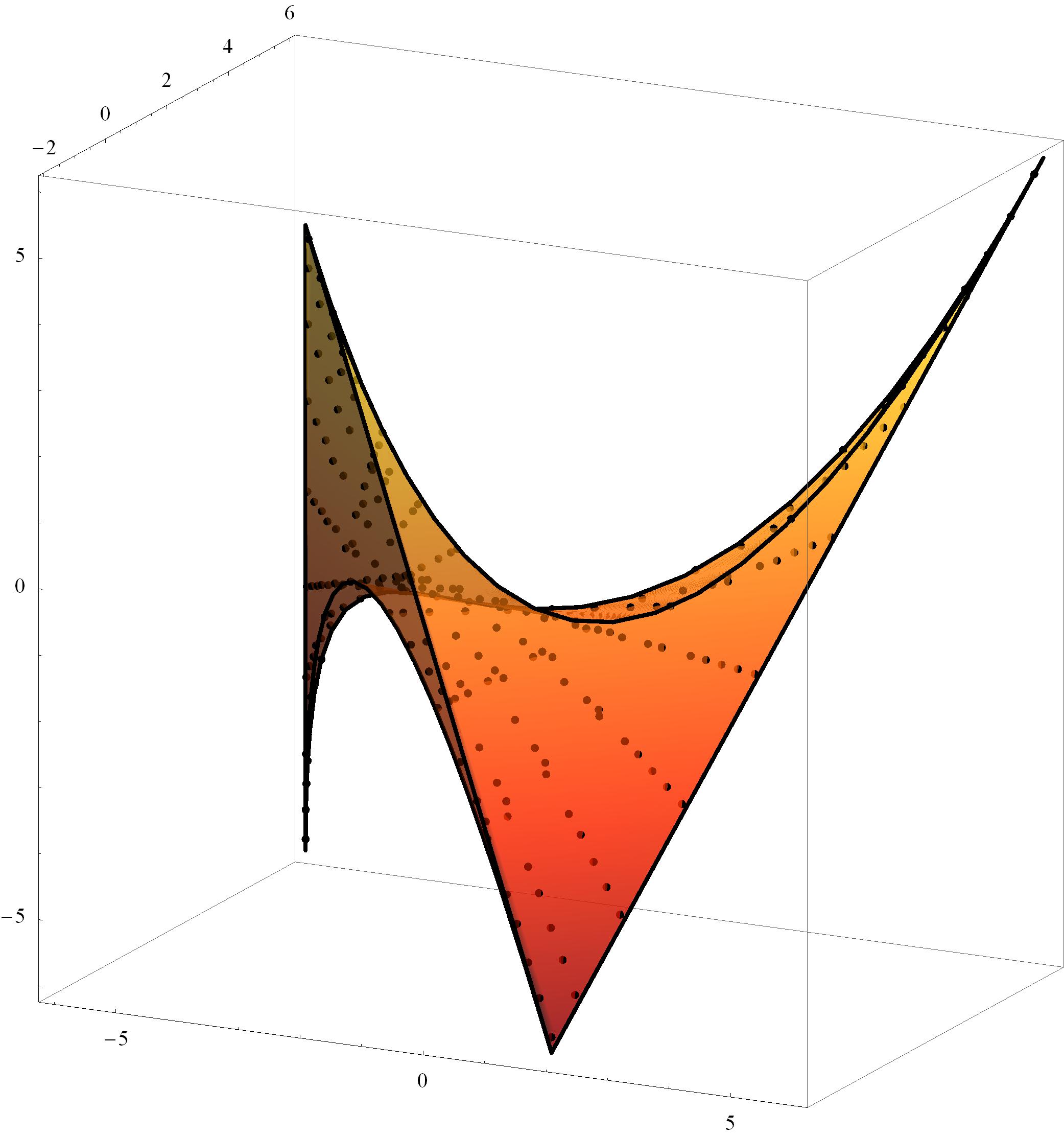}
\includegraphics[width=5cm]{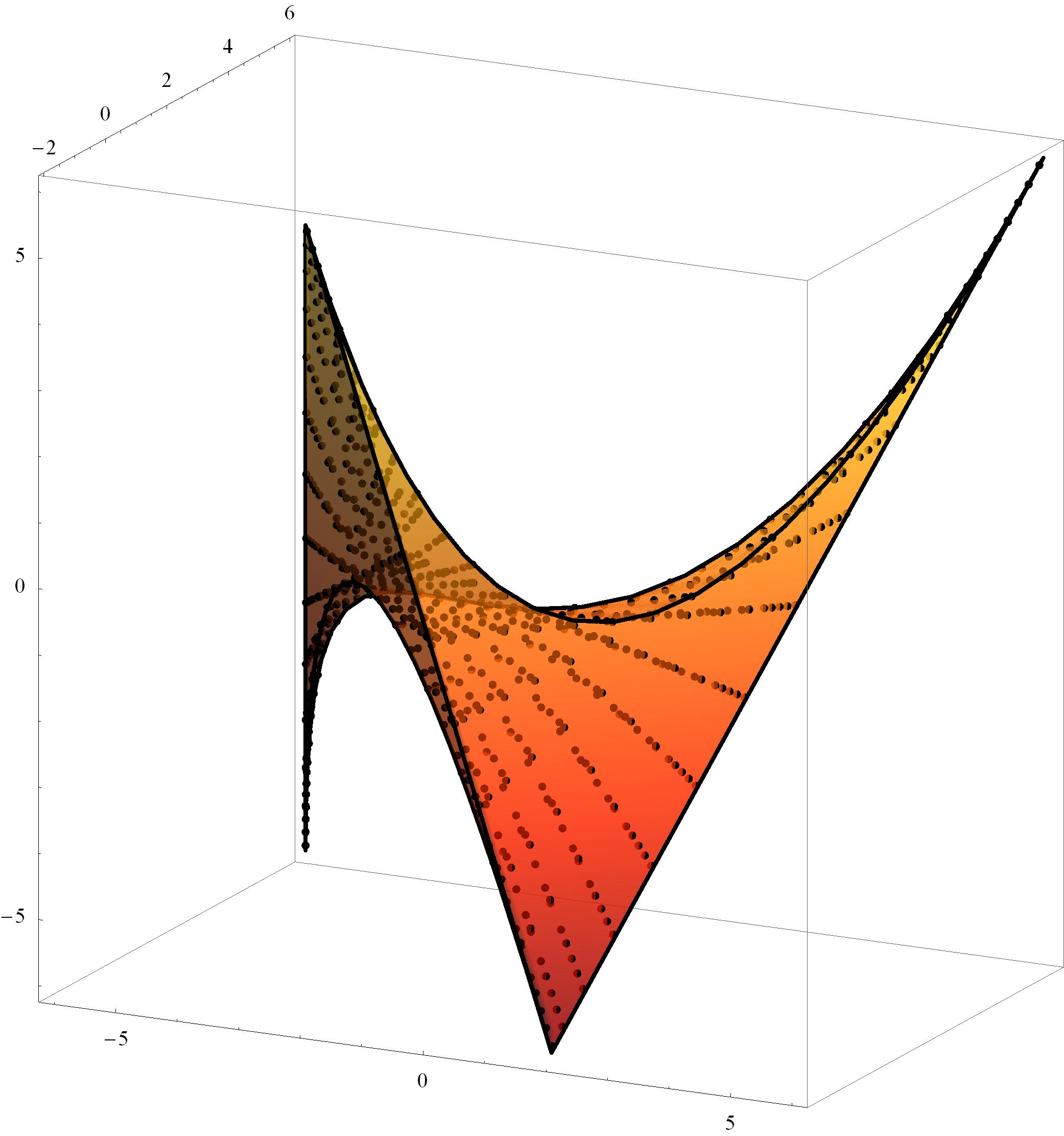}
\caption{The integration domain $\mathcal{F}(\wt{S}_3^{\mathrm{aff}})$ together with points of $\mathcal{F}^{\mathrm{II},+}_{N}$ for $N=5,10,15$.}
\label{figgauss}
\end{figure}

Due to the one-to-one correspondence given by $\varphi$ restricted to the grids $F_N^{t,\pm}$, the symbols $\mathcal{H}_Y,\mathcal{E}_Y$, and $\wt{\mathcal{E}}_Y$ are for $\varphi(s)=Y$ well defined by the relations
\begin{equation}
\mathcal{H}_Y\equiv H_s\,,\quad \mathcal{E}_Y\equiv\ep_s\,,\quad \wt{\mathcal{E}}_Y\equiv\wt\ep_s  .
\end{equation}

\subsection{Gaussian cubature formulas}\

Each family of polynomials $\mathcal{P}^{I,\pm}_k, \mathcal{P}^{III,\pm}_k$ has properties which yield to optimal Gaussian cubature formulas.

\begin{thm}\label{thmcub}\
\begin{enumerate}
\item  For any $N\in\N$ and any polynomial $f$ of degree at most $2N-1$, the following cubature formulas are exact:
\begin{align}
\int_{\mathfrak{F}(\wt{S}_n^\mathrm{aff})}f(Y)\omega^{I,+}(Y)\,dY&=\left(\frac{1}{N}\right)^n\sum_{Y\in \mathfrak{F}_N^{\mathrm{II},+}}\mathcal{H}_Y^{-1}f(Y)\,,  \label{cub1}\\
\int_{\mathfrak{F}(\wt{S}_n^\mathrm{aff})}f(Y)\omega^{III,+}(Y)\,dY&=\left(\frac{2}{2N+1}\right)^n\sum_{Y\in \mathfrak{F}_N^{\mathrm{VIII},+}}\mathcal{H}_Y^{-1}f(Y)\mathcal{J}^{III,+}(Y)\,. \label{cubproof}
\end{align}
\item For any $N\in\N$ and any polynomial $f$ of degree at most $2(N-n)+1$, the following cubature formulas are exact:
\begin{align}
\int_{\mathfrak{F}(\wt{S}_n^\mathrm{aff})}f(Y)\omega^{I,-}(Y)\,dY&=\left(\frac{1}{N}\right)^n\sum_{Y\in \mathfrak{F}_N^{\mathrm{II},-}}f(Y)\mathcal{J}^{I,-}(Y)\,,\\
\int_{\mathfrak{F}(\wt{S}_n^\mathrm{aff})}f(Y)\omega^{III,-}(Y)\,dY&=\left(\frac{2}{2N+1}\right)^n\sum_{Y\in \mathfrak{F}_N^{\mathrm{VIII},-}}f(Y)\mathcal{J}^{III,-}(Y)\,.\label{cub2}
\end{align}
\end{enumerate} 
\end{thm}

\begin{proof}
We show that the formula \eqref{cubproof} holds. Due to the linearity of the integrals and the sums, it is sufficient to prove it only for monomials $f$ of degree not exceeding $2N-1$. Such monomials are expressed as a product of two monomials --  $g_1$ of degree at most $N$ and $g_2$ of degree at most $N-1$. From Propositions \ref{orthogbasis} and \ref{proppol}, the polynomials $g_1$ and $g_2$  are rewritten as linear combinations of polynomials $\mathcal{P}^{III,+}_{k}$ with $k_1\leq N$ and $\mathcal{P}^{III,+}_{k'}$ with  $k'_1 \leq N-1$.
Therefore, we only need to show the formula for those polynomials. 
Using the continuous orthogonality \eqref{contorthogpol}, we obtain
 $$\int_{\mathfrak{F}(\wt{S}_n^\mathrm{aff})}\mathcal{P}^{III,+}_{k}(Y)\mathcal{P}^{III,+}_{k'}(Y)\omega^{III,+}(Y)\,dY=\frac{1}{2^n}H_k\delta_{kk'}\,.$$
Similarly, from equality \eqref{orthogVIII+} it follows that
$$\left(\frac{2}{2N+1}\right)^n\sum_{s\in F_N^{\mathrm{VIII},+}}H_s^{-1}\mathcal{P}^{III,+}_{k}(X(s))\mathcal{P}^{III,+}_{k'}(X(s))(\cos^+_{\rho}(s))^2=\frac{1}{2^n}H_k\delta_{kk'}$$
if $k_1\leq N-1$ and $k_1'\leq N-1$. Note that if $k_1=N$, then $\mathcal{P}^{III,+}_{k}(Y)=0$ for all points $Y$ from $\mathfrak{F}_N^{\mathrm{VIII},+}$. This implies that we extend the discrete orthogonality relation for any $k$ with $k_1\leq N$ and any $k'$ with $k_1'\leq N-1$.
Consequently, the second cubature formula follows from the continuous and discrete orthogonality of polynomials $\mathcal{P}^{III,+}_{k}$. We prove the other formulas similarly. Note that SMDCT II and AMDCT II needed to derive two of the results are found in \cite{KPtrig}.
\end{proof}

\begin{thm}
The cubature formulas \eqref{cub1}--\eqref{cub2} are optimal Gaussian cubature formulas. Moreover, it holds that
\begin{itemize}
\item the orthogonal polynomials $\mathcal{P}^{I,+}_k$ with $k_1=N$ vanish for all points of the set $\mathfrak{F}_N^{\mathrm{II},+}$,
\item the orthogonal polynomials $\mathcal{P}^{III,+}_k$ with $k_1=N$ vanish for all points of the set $\mathfrak{F}_N^{\mathrm{VIII},+}$,
\item the orthogonal polynomials $\mathcal{P}^{I,-}_k$ with $k_1=N-n+1$ vanish for all points of the set $\mathfrak{F}_N^{\mathrm{II},-}$,
\item the orthogonal polynomials $\mathcal{P}^{III,-}_k$ with $k_1=N-n+1$ vanish for all points of the set $\mathfrak{F}_N^{\mathrm{VIII},-}$.
\end{itemize} 
\end{thm}

\begin{proof}
The fact that the nodes are common zeros of the specific sets of orthogonal polynomials follows directly by substituting the grid points to the definition of the polynomials via symmetric and antisymmetric cosine functions \eqref{newpol}. By Propositions \ref{proppol} and \ref{mrizky} we obtain that the number of polynomials $\mathcal{P}^{I,+}_k$ of degree $N-1$ is equal to the number of nodes in $\mathfrak{F}^{\mathrm{II},+}_N$, and therefore the cubature formula \eqref{cub1} is Gaussian. The proof for other cubature formulas is similar.
\end{proof}

The Gaussian formulas \eqref{cub1}--\eqref{cub2} are special cases of formulas derived in a general setting as an (anti)symmetrization of orthogonal polynomials of one variable \cite{xumult}. In particular, the variables $(u_1,\dots,u_n)$ from \cite{xumult} are related to our variables \eqref{var} by 
the relations $$u_i=\frac{X_i}{(n-i)!i!}$$ 
and our orthogonal polynomials \eqref{newpol} correspond to the measure of the form
$$\prod_{i=1}^n(1-y_i)^\alpha(1+y_i)^\beta\left(\prod_{1\leq i<j\leq n}(y_i-y_j)^2\right)^\gamma\,dy_1\dots dy_n\,,\quad -1< y_1<\dots< y_n<1\,,$$
where $$u_k\equiv u_k(y_1,\dots,y_n)=\sum_{1\leq i_1<\dots <i_k\leq n}y_{i_1}\dots y_{i_k}\,,\quad k=1,\dots,n\,.$$
The parameters $\alpha,\beta,\gamma$ take for each class of the polynomials \eqref{newpol} the following values:
\begin{align*}
\mathcal{P}^{I,+}_k: \quad &\alpha=\beta=\gamma=-\frac12\,,\\
\mathcal{P}^{I,-}_k: \quad &\alpha=\beta=-\frac12\,,\; \gamma=\frac12\,,\\
\mathcal{P}^{III,+}_k: \quad &\alpha=\gamma=-\frac12\,,\;\beta=\frac12\,,\\
\mathcal{P}^{III,-}_k: \quad &\alpha=-\frac12\,,\; \beta=\gamma=\frac12\,.
\end{align*}
The advantage of the current specialization is, besides more explicit formulation ready for application, the theoretical link between the polynomials and the (anti)symmetric cosine functions.  This connection, similar to the connection of the classical Chebyshev polynomials of the first kind and the cosine function of one variable, may serve as a starting point for further investigation of the properties of the multivariate polynomials \eqref{newpol} such as their zeroes, discrete orthogonality, and Lebesgue constant.
\subsection{Other cubature formulas}\

Similarly as in  Theorem \ref{thmcub}, one uses the remaining 12 discrete transforms to derive additional cubature formulas for the orthogonal polynomials.  Note that these formulas are slightly less efficient than Gaussian cubature formulas.

\subsubsection{Formulas related to $\mathcal{P}^{I,+}_k$}\

The transforms SMDCT I, V, and VI are used to derive additional cubature formulas for the orthogonal polynomials $\mathcal{P}^{I,+}_k$.
\begin{enumerate}
\item  For any $N\in\N$ and any polynomial $f$ of degree at most $2N-1$, the following cubature formula is exact: 
$$\int_{\mathfrak{F}(\wt{S}_n^\mathrm{aff})}f(Y)\omega^{I,+}(Y)\,dY=\left(\frac{1}{N}\right)^n\sum_{Y\in \mathfrak{F}_N^{\mathrm{I},+}}\mathcal{E}_Y\mathcal{H}_Y^{-1}f(Y)\,.$$
\item 
 For any $N\in\N$, $N\geq 2$, and any polynomial $f$ of degree at most $2(N-1)$, the following cubature formulas are exact: 
$$\int_{\mathfrak{F}(\wt{S}_n^\mathrm{aff})}f(Y)\omega^{I,+}(Y)\,dY=\left(\frac{2}{2N-1}\right)^n\sum_{Y\in \mathfrak{F}_N^{\mathrm{V},+}}\mathcal{E}_Y\mathcal{H}_Y^{-1}f(Y)\,, $$
$$\int_{\mathfrak{F}(\wt{S}_n^\mathrm{aff})}f(Y)\omega^{I,+}(Y)\,dY=\left(\frac{2}{2N-1}\right)^n\sum_{Y\in \mathfrak{F}_N^{\mathrm{VI},+}}\wt{\mathcal{E}}_Y\mathcal{H}_Y^{-1}f(Y)\,. $$
\end{enumerate}

\subsubsection{Formulas related to $\mathcal{P}^{I,-}_k$}\

The transforms AMDCT I, V, and VI are used to derive additional cubature formulas for the orthogonal polynomials $\mathcal{P}^{I,-}_k$. 
\begin{enumerate}
\item 
For any $N\in\N$, $N\geq n$, and any polynomial $f$ of degree at most $2(N-n)+1$, the following cubature formula is exact:
$$\int_{\mathfrak{F}(\wt{S}_n^\mathrm{aff})}f(Y)\omega^{I,-}(Y)\,dY=\left(\frac{1}{N}\right)^n\sum_{Y\in \mathfrak{F}_N^{\mathrm{I},-}}\mathcal{E}_Y f(Y)\mathcal{J}^{I,-}(Y)\,. $$
\item For any $N\in\N$, $N> n$, and any polynomial $f$ of degree at most $2(N-n)$, the following cubature formulas are exact: 
$$\int_{\mathfrak{F}(\wt{S}_n^\mathrm{aff})}f(Y)\omega^{I,-}(Y)\,dY=\left(\frac{2}{2N-1}\right)^n\sum_{Y\in \mathfrak{F}_N^{\mathrm{V},-}}\mathcal{E}_Yf(Y)\mathcal{J}^{I,-}(Y)\,, $$
$$\int_{\mathfrak{F}(\wt{S}_n^\mathrm{aff})}f(Y)\omega^{I,-}(Y)\,dY=\left(\frac{2}{2N-1}\right)^n\sum_{Y\in \mathfrak{F}_N^{\mathrm{VI},-}}\wt{\mathcal{E}}_Y f(Y)\mathcal{J}^{I,-}(Y)\,. $$
\end{enumerate}

\subsubsection{Formulas related to $\mathcal{P}^{III,+}_k$}\

The transforms SMDCT III, IV, and VII are used to derive additional cubature formulas for the polynomials $\mathcal{P}^{III,+}_k$. 
\begin{enumerate}
\item 
For any $N\in\N$, $N\geq 2$, and any polynomial $f$ of degree at most $2(N-1)$, the following cubature formulas are exact: 
$$\int_{\mathfrak{F}(\wt{S}_n^\mathrm{aff})}f(Y)\omega^{III,+}(Y)\,dY=\left(\frac{1}{N}\right)^n\sum_{Y\in \mathfrak{F}_N^{\mathrm{III},+}}\mathcal{E}_Y\mathcal{H}_Y^{-1}f(Y)\mathcal{J}^{III,+}(Y)\,, $$
$$\int_{\mathfrak{F}(\wt{S}_n^\mathrm{aff})}f(Y)\omega^{III,+}(Y)\,dY=\left(\frac{1}{N}\right)^n\sum_{Y\in \mathfrak{F}_N^{\mathrm{IV},+}}\mathcal{H}_Y^{-1}f(Y)\mathcal{J}^{III,+}(Y)\,. $$
\item 
For any $N\in\N$, $N\geq 2$, and any polynomial $f$ of degree at most $2(N-1)-1$, the following cubature formula is exact: 
$$\int_{\mathfrak{F}(\wt{S}_n^\mathrm{aff})}f(Y)\omega^{III,+}(Y)\,dY=\left(\frac{2}{2N-1}\right)^n\sum_{Y\in \mathfrak{F}_N^{\mathrm{VII},+}}\mathcal{E}_Y\mathcal{H}_Y^{-1}f(Y)\mathcal{J}^{III,+}(Y)\,. $$
\end{enumerate}

\subsubsection{Formulas related to $\mathcal{P}^{III,-}_k$}\

The transforms AMDCT III, IV, and VII are used to derive additional cubature formulas for the polynomials $\mathcal{P}^{III,-}_k$.
\begin{enumerate}
\item
For any $N\in\N$, $N> n$, and any polynomial $f$ of degree at most $2(N-n)$, the following cubature formulas are exact: 
$$\int_{\mathfrak{F}(\wt{S}_n^\mathrm{aff})}f(Y)\omega^{III,-}(Y)\,dY=\left(\frac{1}{N}\right)^n\sum_{Y\in \mathfrak{F}_N^{\mathrm{III},-}}\mathcal{E}_Y f(Y)\mathcal{J}^{III,-}(Y)\,, $$
$$\int_{\mathfrak{F}(\wt{S}_n^\mathrm{aff})}f(Y)\omega^{III,-}(Y)\,dY=\left(\frac{1}{N}\right)^n\sum_{Y\in \mathfrak{F}_N^{\mathrm{IV},-}}f(Y)\mathcal{J}^{III,-}(Y)\,. $$
\item 
For any $N\in\N$, $N> n$, and any polynomial $f$ of degree at most $2(N-n)-1$, the following cubature formula is exact: 
$$\int_{\mathfrak{F}(\wt{S}_n^\mathrm{aff})}f(Y)\omega^{III,-}(Y)\,dY=\left(\frac{2}{2N-1}\right)^n\sum_{Y\in \mathfrak{F}_N^{\mathrm{VII},-}}\mathcal{E}_Y f(Y)\mathcal{J}^{III,-}(Y)\,. $$
\end{enumerate}

\section{Concluding remarks}
\begin{itemize}
\item In this paper, only the antisymmetric and symmetric generalizations of cosine functions are investigated. Similarly to the cosine functions, generalizations of the common sine functions of one variable are defined in \cite{KPtrig} and these generalizations have several remarkable discretization properties. The discrete antisymmetric and symmetric sine transforms of types I--IV  are developed in \cite{sin} for the two-dimensional case only and the discrete multivariate sine transforms of type I are found in \cite{KPtrig}. The generalization of the antisymmetric and symmetric discrete sine transforms,  analogous to the discrete multivariate sine transforms of types II--VIII \cite{Brit}, have not   yet been described.  
\item The 16 DCTs, which are described on the grids $F^{\mathrm{I},\pm}_N,\dots, F^{\mathrm{VIII},\pm}_N$, are straightforwardly translated into 16 transforms of the corresponding polynomials on the grids, $$\mathfrak{F}^{\mathrm{I},\pm}_N,\dots, \mathfrak{F}^{\mathrm{VIII},\pm}_N.$$ Polynomial interpolation formulas, similar to formulas (3.30) and \eqref{inter}, can obviously be formulated. The interpolation properties of these polynomial formulas as well as the convergence of the corresponding polynomials series poses an open problem. 
\item The question of whether one can introduce multivariate Chebyshev-like polynomials of the second and fourth kinds in connection with the generalizations of the sine functions forms another open problem. The decomposition of products of two-variable antisymmetric and symmetric sine functions from \cite{sin} indicates the possibility of construction of such multivariate polynomials. Moreover, the continuous and discrete orthogonality of the antisymmetric and symmetric sine functions further indicates that the corresponding cubature formulas, useful in numerical analysis, can again be developed.
\item Since the Weyl groups corresponding to the simple Lie algebras $B_n$ and $C_n$  are isomorphic to $(\Z/2\Z)^n\rtimes S_n$ \cite{hum}, it is possible to show that the antisymmetric and symmetric generalizations of trigonometric functions are related to the orbit functions -- $C$-, $S$-, $S^s$-, and $S^l$-functions -- studied in \cite{OF2,OF1,SsSlcub}. For example, the symmetric cosine functions coincide, up to a multiplication by constant, with $C$-functions and the antisymmetric cosine functions become, up to a multiplication by constant, $S^l$-functions in the case $B_n$, and $S^s$-functions in the case $C_n$. The grids considered in this paper, which allow the eight types of the transforms for each case, are, however, different from the grids on which the discrete calculus of the orbit functions is described. Thus, the collection of the resulting cubature formulas is richer and includes formulas of the Gaussian type.

\item The exploration of the vast number of theoretical aspects as well as applications of the Chebyshev polynomials \cite{Hand,riv} is beyond the scope of this work. However, since the basic properties of these polynomials -- such as the discrete and continuous orthogonality together with the cubature formulas -- are replicated for the multivariate symmetric and antisymmetric cosine functions, one may expect that many other properties such as existence of FFT-based algorithms and the Clenshaw-Curtis quadrature technique \cite{tref} will find their corresponding multidimensional (anti)symmetric generalizations as well. The work presented in this paper may represent a starting point for these open problems and further research. 
\end{itemize}

\section*{Acknowledgments}

The authors gratefully acknowledge the support of this work by the Natural Sciences and Engineering Research Council of Canada, by the Doppler Institute of the Czech Technical University in Prague and by RVO68407700. JH is grateful for the hospitality extended to him at the Centre de recherches math\'ematiques, Universit\'e de Montr\'eal. LM would also like to express her gratitude to the Department of Mathematics and Statistic at Universit\'e de Montr\'eal, for the hospitality extended to her during her doctoral studies, and to the Institute de Sciences Math\'ematiques de Montr\'eal.  


\end{document}